\documentclass[10]{article}
\usepackage{amsfonts,amssymb,amsmath,amsthm,cite}
\usepackage{graphicx}

\usepackage[applemac]{inputenc}
\usepackage{amsmath,amssymb,amsthm, hyperref, euscript}
\usepackage[matrix,arrow,curve]{xy}
\usepackage{graphicx}
\usepackage{tabularx}
\usepackage{float}
\usepackage{hyperref}
\usepackage{tikz}
\usepackage{slashed}
\usepackage{mathrsfs}

\usetikzlibrary{matrix}

\usepackage[T1]{fontenc}
\usepackage{amsfonts,cite}
\usepackage{graphicx}

\usepackage[applemac]{inputenc}

\usepackage[sc]{mathpazo}
\DeclareFontFamily{T1}{pzc}{}
\DeclareFontShape{T1}{pzc}{m}{it}{1.8 <-> pzcmi8t}{}
\DeclareMathAlphabet{\mathpzc}{T1}{pzc}{m}{it}
\theoremstyle{plain}
\newtheorem{prop}{Proposition}[section]

\newtheorem{cor}[prop]{Corollary}%[section]
\newtheorem{thm}[prop]{Theorem}%[section]
\newtheorem{theorem}[prop]{Theorem}
\newtheorem{lemma}[prop]{Lemma}

\newtheorem{corollary}[prop]{Corollary}

\theoremstyle{definition}
\newtheorem{defn}[prop]{Definition}%[section]
\newtheorem{empt}[prop]{}%[section]
\theoremstyle{definition}
\newtheorem{definition}[prop]{Definition}
\newtheorem{example}[prop]{Example}
\newtheorem{remark}[prop]{Remark}
\numberwithin{equation}{section}

\newcommand{\vertiii}[1]{{\left\vert\kern-0.25ex\left\vert\kern-0.25ex\left\vert #1
    \right\vert\kern-0.25ex\right\vert\kern-0.25ex\right\vert}}
\newcommand{\Coo}{C^\infty}                  %% smooth functions
\usepackage[sc]{mathpazo}
\newbox\ncintdbox \newbox\ncinttbox %% noncommutative integral symbols
\setbox0=\hbox{$-$} \setbox2=\hbox{$\displaystyle\int$}
\setbox\ncintdbox=\hbox{\rlap{\hbox
    to \wd2{\hskip-.125em \box2\relax\hfil}}\box0\kern.1em}
\setbox0=\hbox{$\vcenter{\hrule width 4pt}$}
\setbox2=\hbox{$\textstyle\int$} \setbox\ncinttbox=\hbox{\rlap{\hbox
    to \wd2{\hskip-.175em \box2\relax\hfil}}\box0\kern.1em}

\newcommand{\Id}{\mathrm{Id}}                %% identity map
\newcommand{\C}{\mathbb{C}}                  %% complex numbers
\renewcommand{\H}{\mathcal{H}}               %% Hilbert space
\newcommand{\hookto}{\hookrightarrow}        %% abbreviation
\newcommand{\N}{\mathbb{N}}                  %% nonnegative integers
\newcommand{\eps}{\varepsilon}                    %% tensor product
\newcommand{\R}{\mathbb{R}}                  %% real numbers
\renewcommand{\SS}{\mathcal{S}}              %% Schwartz space
\DeclareMathOperator{\supp}{\mathfrak{supp}}         %% support
\newcommand{\T}{\mathbb{T}}                  %% circle as a group
\renewcommand{\th}{\theta}                   %% short for \theta
\newcommand{\Z}{\mathbb{Z}}                  %% integers
\renewcommand{\th}{\theta}        %% short for  \theta
\def\<#1|#2>{\langle#1\stroke#2\rangle} %% \braket (Dirac notation)
\def\?#1|#2?{\{#1\stroke#2\}}        %% left-linear pairing

\def\<#1,#2>{\langle#1,#2\rangle}            %% bilinear pairing
\def\ee_#1{e_{{\scriptscriptstyle#1}}}       %% basis projector
\def\wick:#1:{\mathopen:#1\mathclose:}       %% Wick-ordered operator

\newbox\ncintdbox \newbox\ncinttbox %% noncommutative integral symbols
\setbox0=\hbox{$-$}
\setbox2=\hbox{$\displaystyle\int$}
\setbox\ncintdbox=\hbox{\rlap{\hbox
           to \wd2{\box2\relax\hfil}}\box0\kern.1em}
\setbox0=\hbox{$\vcenter{\hrule width 4pt}$}
\setbox2=\hbox{$\textstyle\int$}
\setbox\ncinttbox=\hbox{\rlap{\hbox
           to \wd2{\hskip-.05em\box2\relax\hfil}}\box0\kern.1em}

\newcommand{\stroke}{\mathbin|}   %% (for `\bbraket' and such)
\title{Coverings of foliation algebras}
\begin{document}
\maketitle  \setlength{\parindent}{0pt}
\begin{center}
\author{
{\textbf{Petr R. Ivankov*}\\
e-mail: * monster.ivankov@gmail.com \\
}
}
\end{center}

\vspace{1 in}

\begin{abstract}
\noindent

\paragraph{}

This article is devoted to the geometric construction which states a natural correspondence between topological  coverings of a foliated manifolds and noncommutative coverings of the operator algebras. However this correspondence is not one to one because there are noncommutaive coverings of foliations which do not comply with discussed in this article construction.
\end{abstract}
%\tableofcontents
\section{Motivation. Preliminaries}
\paragraph*{}
Gelfand-Na\u{\i}mark theorem \cite{arveson:c_alg_invt} states the correspondence between  locally compact Hausdorff topological spaces and commutative $C^*$-algebras.

\begin{theorem}\label{gelfand-naimark}\cite{arveson:c_alg_invt} (Gelfand-Na\u{\i}mark). 
	Let $A$ be a commutative $C^*$-algebra and let $\mathcal{X}$ be the spectrum of A. There is the natural $*$-isomorphism $\gamma:A \xrightarrow{\approx} C_0(\mathcal{X})$.
\end{theorem}

So any (noncommutative) $C^*$-algebra may be regarded as a generalized (noncommutative)  locally compact Hausdorff topological space. Following theorem yields a pure algebraic description of finite-fold coverings of compact spaces.
\begin{theorem}\label{pavlov_troisky_thm}\cite{pavlov_troisky:cov}
	Suppose $\mathcal X$ and $\mathcal Y$ are compact Hausdorff connected spaces and $p :\mathcal  Y \to \mathcal X$
	is a continuous surjection. If $C(\mathcal Y )$ is a projective finitely generated Hilbert module over
	$C(\mathcal X)$ with respect to the action
	\begin{equation*}
	(f\xi)(y) = f(y)\xi(p(y)), ~ f \in  C(\mathcal Y ), ~ \xi \in  C(\mathcal X),
	\end{equation*}
	then $p$ is a finite-fold  covering.
\end{theorem}

 An article \cite{ivankov:qnc} contains   pure algebraic generalizations of following topological objects:
\begin{itemize}
	\item Coverings of noncompact spaces,
	\item Infinite coverings. 
\end{itemize}
Here the described in \cite{ivankov:qnc} theory is applied to operator algebras of foliations.

% For any additive category the component $A_\iota$ of a direct product $\prod_{\iota \in I}A_\iota$ or a direct sum $\oplus_{\iota \in I}A_\iota$ can be regarded as the subobject, i.e.
% \begin{equation}\label{dir_summand_eqn}
% \begin{split}
% A_\iota \subset \prod_{\iota \in I}A_\iota,  \\
% A_\iota \subset \bigoplus_{\iota \in I}A_\iota.
%\end{split}
% \end{equation}
% In this article the above inclusions will be implicitly used.

Following table contains  special symbols.

\begin{tabular}{|c|c|}
\hline
Symbol & Meaning\\
\hline
&\\
$A^+$  & Unitisation of $C^*$- algebra $A$\\
%$A^0$  & Opposite algebra of  $A$  consisting of elements
%$\{a^0 : a \in A\}$ \\
%& with product $a^0b^0 = (ba)^0$.\\

%$\hat{A}$ & Spectrum of a  $C^*$- algebra $A$  with the hull-kernel topology \\
% & (or Jacobson topology)\\
$A_+$  & Cone of positive elements of $C^*$- algebra, i.e. $A_+ = \left\{a\in A \ | \ a \ge 0\right\}$\\
$A^G$  & Algebra of $G$ - invariants, i.e. $A^G = \left\{a\in A \ | \ ga=a, \forall g\in G\right\}$\\
%$\mathrm{Aut}(A)$ & Group of * - automorphisms of $C^*$- algebra $A$\\
$A''$  & Enveloping von Neumann algebra  of $A$\\

$B(\H)$ & Algebra of bounded operators on a Hilbert space $\H$\\
%$B_{\infty}=B_{\infty}(\{z\in \mathbb{C} \ | \ |z|=1\})$  & Algebra of Borel measured functions on the $\{z\in \mathbb{C} \ | \ |z|=1\}$ set. \\
$\mathbb{C}$ (resp. $\mathbb{R}$)  & Field of complex (resp. real) numbers \\
%$\mathbb{C}^*$ & $\{z \in \mathbb{C} \ | \ |z| = 1\}$ \\
$C(\mathcal{X})$ & $C^*$- algebra of continuous complex valued \\
 & functions on a compact  space $\mathcal{X}$\\
$C_0(\mathcal{X})$ & $C^*$- algebra of continuous complex valued functions on a locally \\
 &   compact  topological space $\mathcal{X}$ equal to $0$ at infinity\\
%$C_c(\mathcal{X})$ & Algebra of continuous complex valued functions on a \\
% &  topological  space $\mathcal{X}$ with compact support\\
 $C_b(\mathcal{X})$ & $C^*$- algebra of bounded  continuous complex valued \\
  & functions on a locally compact topological space $\mathcal{X}$ \\
 %$\cl\left( \mathcal{U}\right)  $ & The closure of the subset $\mathcal{U} \subset \mathcal{X}$ of the topological space $\mathcal{X}$ \\
 % $\mathfrak{int}\left( \mathcal{U}\right)  $ & The interior of the subset $\mathcal{U} \subset \mathcal{X}$ of the topological space $\mathcal{X}$ \\
  %$G_{tors} \subset G$  & The torsion subgroup of an Abelian group\\
$G\left( \widetilde{\mathcal{X}}~ |~ \mathcal{X}\right) $ & Group of covering transformations of covering  $\widetilde{\mathcal{X}} \to \mathcal{X}$ \cite{spanier:at}  \\
%$\overline{G/G'}\subset G$  & A set of representatives of a quotient set $G/G'$\\
$\delta_{jk}$ & Delta symbol. If $j = k$ then $\delta_{jk}=1$.  If $j \neq k$ then $\delta_{jk}=0$  \\
%$\Ga(\mathcal X, E)$ & A $C(\mathcal{X})$-module of sections of a locally trivial vector bundle $E \in \mathrm{Vect}(\mathcal{X})$ \\
$\H$ & Hilbert space \\
%$H_A$ & Hilbert space over  $A$ (definition \ref{hilb_a}) \\
$\mathcal{K}= \mathcal{K}\left(\H \right) $ & $C^*$- algebra of compact operators on the separable Hilbert space $\H$  \\
%$\mathcal{K}(X_A)$ & $C^*$ - algebra of compact operators of a Hilbert $A$ module $X_A$ \\
%$K_i(A)$ ($i = 0, 1$) & $K$ groups of $C^*$-algebra $A$\\
%$I = [0, 1] \subset \mathbb{R}$ & Closed unit  interval\\
%$K(A)$ & Pedersen ideal of $C^*$-algebra $A$\\
%$\mathcal{K}(H)$ or $\mathcal{K}$ & Algebra of compact operators on Hilbert space $H$\\
$\ell^2\left(A \right)$ & Standard Hilbert $A$-module\\
$\varinjlim$ & Direct limit \\
$\varprojlim$ & Inverse limit \\
$M(A)$  & A multiplier algebra of $C^*$-algebra $A$\\
%$\mathbb{M}_n(A)$  & The $n \times n$ matrix algebra over $C^*-$ algebra $A$\\
$\mathbb{N}$  & A set of positive integer numbers\\
$\mathbb{N}^0$  & A set of nonnegative integer numbers\\
%$S^n$ & The $n$-dimensional sphere\\
%$SU(n)$ & Special unitary group \\

%$\mathscr{P}(\mathcal{X})$  & Fundamental groupoid of a topological space $\mathcal{X}$\\

%$\mathbb{Q}$  & Field of rational numbers \\
 % $\mathrm{sp}(a)$ & Spectrum of element of $C^*$-algebra $a\in A$  \\
$\supp \varphi$ & Support of a continuous map $\varphi: \mathcal X \to \mathbb{C}$\\
%$TM$ (resp. $T^*M$) & Tangent (resp. cotangent) bundle of differentiable manifold $M$ \cite{koba_nomi:fgd}\\$U(H) \subset \mathcal{B}(H) $ & Group of unitary operators on Hilbert space $H$\\
%$U(A) \subset A $ & Group of unitary operators of algebra $A$\\
%$U(n) \subset GL(n, \mathbb{C}) $ & Unitary subgroup of general linear group\\
%$\mathrm{Vect}(\mathcal{X})$ & A category of locally trivial vector bundles over a topological space $\mathcal X$ \cite{karoubi:k}\\ 
$\mathbb{Z}$ & Ring of integers \\

$\mathbb{Z}_n$ & Ring of integers modulo $n$ \\
%$\overline{k} \in \mathbb{Z}_n$ & An element in $\mathbb{Z}_n$ represented by $k \in \mathbb{Z}$  \\
%$\Omega$ &  Natural contravariant functor from category  of commutative \\ & $C^*$ - algebras, to category of Hausdorff spaces\\
$X \backslash A$ & Dif and only iference of sets  $X \backslash A= \{x \in X \ | \ x\notin A\}$\\
$|X|$ & Cardinal number of a finite set $X$\\ 
%$\left[x\right]$ & The range projection of element $x$ of a von Neumann algebra.\\ 
$f|_{A'}$& Restriction of a map $f: A\to B$ to $A'\subset A$, i.e. $f|_{A'}: A' \to B$\\ 
\hline
\end{tabular}

\break

%\begin{defn}
%Let 
%\begin{equation*}
%G_1 \leftarrow G_2 \leftarrow ...
%\end{equation*}

%$G_1 \leftarrow G_2 \leftarrow ...$ be infinite sequence of groups and epimorphisms, and let $G$ be a group with epimorpisms $h_n: G \to G_n$. The sequence is said to be {\it coherent} if $\bigcap\mathrm{ker} \ h_n$ is trivial and a following diagram is commutative.

%\begin{tikzpicture}\label{borel_local_comm}
 % \matrix (m) [matrix of math nodes,row sep=3em,column sep=4em,minimum width=2em]
%  {
%       & G  &  \\
%    G_n &  &  G_{n-1} \\};
%  \path[-stealth]
%    (m-1-2) edge node [left] {$h_n$} (m-2-1)
%    (m-1-2) edge node [right] {$h_{n-1}$} (m-2-3)
 %   (m-2-1) edge node [left] {}  (m-2-3);
  
%\end{tikzpicture}

%\end{defn}

%Henceforth  $\left\{x_{\iota}\right\}_{\iota \in I}$ means a set indexed by finite or countable set $I$ of indexes.
\subsection{Hilbert modules}
\paragraph*{} We refer to \cite{blackadar:ko} for the definition of Hilbert $C^*$-modules, or simply Hilbert modules. Let $A$ be a $C^*$- algebra, and let $X_A$ be an $A$-Hilbert module. Let $\langle \cdot, \cdot \rangle_{X_A}$ be the $A$-valued product on $X_A$. For any $\xi, \zeta \in X_A$ let us define an $A$-endomorphism $\theta_{\xi, \zeta}$ given by  $\theta_{\xi, \zeta}(\eta)=\xi \langle \zeta, \eta \rangle_{X_A}$ where $\eta \in X_A$. The operator  $\theta_{\xi, \zeta}$ shall be denoted by $\xi \rangle\langle \zeta$. The norm completion of a generated by operators $\theta_{\xi, \zeta}$ algebra is said to be an algebra of compact operators $\mathcal{K}(X_A)$. We suppose that there is a left action of $\mathcal{K}(X_A)$ on $X_A$ which is $A$-linear, i.e. action of  $\mathcal{K}(X_A)$ commutes with action of $A$. For any $C^*$-algebra $A$ denote by $\ell^2\left( A\right)$ the \textit{standard Hilbert $A$-module} given by
\begin{equation}\label{st_hilb_end}
\begin{split}
\ell^2\left( A\right) = \left\{\left\{a_n\right\}_{n \in \N}\in A^{\N}~|~\sum_{n =1}^\infty a^*_na_n < \infty \right\},\\
\left\langle\left\{a_n\right\}, \left\{b_n\right\}\right\rangle_{\ell^2\left( A\right)}=\sum_{n =1}^\infty a^*_nb_n.
\end{split}
\end{equation}

  \subsection{$C^*$-algebras and von Neumann algebras}
 \paragraph*{} In this section I follow to \cite{pedersen:ca_aut}.
  % \begin{definition}\cite{pedersen:ca_aut}
  %	Let $A$ be a $C^*$-algebra.  The {\it strict topology} on the multiplier algebra $M(A)$ is the topology generated by seminorms $\vertiii{x}_a = \|ax\| + \|xa\|$, ($a\in A$). If $x \in M(A)$  and a sequence of partial sums $\sum_{i=1}^{n}a_i$ ($n = 1,2, ...$), ($a_i \in A$) tends to $x$ in the strict topology then we shall write
  % 	\begin{equation*}
  % 	x = \sum_{i=1}^{\infty}a_i.
  %  	\end{equation*}
  % \end{definition}
  \begin{definition}\cite{pedersen:ca_aut} Let $\H$ be a Hilbert space. The {\it strong} topology on $B\left(\H\right)$ is the locally convex vector space topology associated with the family of seminorms of the form $x \mapsto \|x\xi\|$, $x \in B(\H)$, $\xi \in \H$.
  \end{definition}
  \begin{definition}\cite{pedersen:ca_aut} Let $\H$ be a Hilbert space. The {\it weak} topology on $B\left(\H\right)$ is the locally convex vector space topology associated with the family of seminorms of the form $x \mapsto \left|\left(x\xi, \eta\right)\right|$, $x \in B(\H)$, $\xi, \eta \in \H$.
  \end{definition}
  
  \begin{theorem}\label{vN_thm}\cite{pedersen:ca_aut}
  	Let $M$ be a $C^*$-subalgebra of $B(\H)$, containing the identity operator. The following conditions are equivalent:
  	\begin{itemize}
  		\item $M=M''$ where $M''$ is the bicommutant of $M$;
  		\item $M$ is weakly closed;
  		\item $M$ is strongly closed.
  	\end{itemize}
  \end{theorem}
  
  \begin{definition}
  	Any $C^*$-algebra $M$ is said to be a {\it von Neumann algebra} or a {\it $W^*$- algebra} if $M$ satisfies to the conditions of the Theorem \ref{vN_thm}.
  \end{definition}
  \begin{definition} \cite{pedersen:ca_aut}
  	Let $A$ be a $C^*$-algebra, and let $S$ be the state space of $A$. For any $s \in S$ there is an associated representation $\pi_s: A \to B\left( \H_s\right)$. The representation $\bigoplus_{s \in S} \pi_s: A \to \bigoplus_{s \in S} B\left(\H_s \right)$ is said to be the \textit{universal representation}. The universal representation can be regarded as $A \to B\left( \bigoplus_{s \in S}\H_s\right)$.  
  \end{definition} 
  \begin{definition}\label{env_alg_defn}\cite{pedersen:ca_aut}
  	Let   $A$ be a $C^*$-algebra, and let $A \to B\left(\H \right)$ be the universal representation $A \to B\left(\H \right)$. The strong closure of $\pi\left( A\right)$ is said to be   the  {\it enveloping von Neumann algebra} or  the {\it enveloping $W^*$-algebra}  of $A$. The enveloping  von Neumann algebra will be denoted by $A''$.
  \end{definition}

  \section{Noncommutative finite-fold coverings}\label{ff_c_sec}

\paragraph*{}
\begin{definition}
	If $A$ is a $C^*$- algebra then an action of a group $G$ is said to be {\it involutive } if $ga^* = \left(ga\right)^*$ for any $a \in A$ and $g\in G$. The action is said to be \textit{non-degenerated} if for any nontrivial $g \in G$ there is $a \in A$ such that $ga\neq a$. 
\end{definition}
\begin{definition}\label{fin_def_uni}
	Let $A \hookto \widetilde{A}$ be an injective *-homomorphism of unital $C^*$-algebras. Suppose that there is a non-degenerated involutive action $G \times \widetilde{A} \to \widetilde{A}$ of a finite group $G$, such that $A = \widetilde{A}^G\stackrel{\text{def}}{=}\left\{a\in \widetilde{A}~|~ a = g a;~ \forall g \in G\right\}$. There is an $A$-valued product on $\widetilde{A}$ given by
	\begin{equation}\label{finite_hilb_mod_prod_eqn}
	\left\langle a, b \right\rangle_{\widetilde{A}}=\sum_{g \in G} g\left( a^* b\right) 
	\end{equation}
	and $\widetilde{A}$ is an $A$-Hilbert module. We say that $\left(A, \widetilde{A}, G \right)$ is an \textit{unital noncommutative finite-fold  covering} if $\widetilde{A}$ is a finitely generated projective $A$-Hilbert module.
\end{definition}
\begin{remark}
	Above definition is motivated by the Theorem \ref{pavlov_troisky_thm}.
\end{remark}
\begin{definition}\label{fin_def}
	Let $A$, $\widetilde{A}$ be $C^*$-algebras such  that following conditions hold:
	\begin{enumerate}
		\item[(a)] 
		There are unital $C^*$-algebras $B$, $\widetilde{B}$  and inclusions 
		$A \subset B$,  $\widetilde{A}\subset \widetilde{B}$ such that $A$ (resp. $B$) is an essential ideal of $\widetilde{A}$ (resp. $\widetilde{B}$),
		\item[(b)] There is an unital  noncommutative finite-fold covering $\left(B ,\widetilde{B}, G \right)$,
		\item[(c)] 
		\begin{equation}\label{wta_eqn}
		\widetilde{A} =  \left\{a\in \widetilde{B}  ~|~ \left\langle \widetilde{B} ,a  \right\rangle_{\widetilde{B} } \in A \right\}.
		\end{equation}
	\end{enumerate}
	
	The triple $\left(A, \widetilde{A},G \right)$ is said to be a \textit{noncommutative finite-fold covering}. The group $G$ is said to be the \textit{covering transformation group} (of $\left(A, \widetilde{A},G \right)$ ) and we use the following notation
	\begin{equation}\label{group_cov_eqn}
	G\left(\widetilde{A}~|~A \right) \stackrel{\mathrm{def}}{=} G.
	\end{equation}
\end{definition}
\begin{lemma}\label{fin_cov_constr}\cite{ivankov:qnc}Let us consider the situation of the Definition \ref{fin_def}. Following conditions hold:
	\begin{enumerate}
		\item [(i)] From \eqref{wta_eqn} it turns out that $\widetilde{A}$ is a closed two sided ideal of $\widetilde{B}$,
		\item[(ii)] The action of $G$ on $\widetilde{B}$ is such that $G\widetilde{A}=\widetilde{A}$, i.e. there is the natural action of $G$ on $\widetilde{A}$,
		\item[(iii)] \begin{equation}\label{wtag_eqn}
		A \cong \widetilde{A}^G=\left\{a\in \widetilde{A}~|~ a = g a;~ \forall g \in G\right\}.
		\end{equation}
\end{enumerate}	\end{lemma}
\begin{remark}
	The Definition \ref{fin_cov_constr} is motivated by the Theorem \ref{comm_fin_thm}.
\end{remark}

%\begin{remark}
%	Any	unital noncommutative finite-fold  covering is a special case of a noncommutative finite-fold  covering.
%\end{remark}
\begin{definition}
	The injective *-homomorphism $A \hookto \widetilde{A}$, which follows from \eqref{wtag_eqn}  is said to be a \textit{noncommutative finite-fold covering}.
\end{definition}
\begin{definition}\label{hilbert_product_defn}
	Let $\left(A, \widetilde{A}, G\right)$ be a    noncommutative finite-fold covering.  Algebra  $\widetilde{A}$  is a Hilbert $A$-module with an $A$-valued  product given by
	\begin{equation}\label{fin_form_a}
	\left\langle a, b \right\rangle_{\widetilde{A}} = 
	\sum_{g \in G} g(a^*b); ~ a,b \in \widetilde{A}.
	\end{equation}
	We say that this structure of Hilbert $A$-module is {\it induced by the covering} $\left(A, \widetilde{A}, G\right)$. Henceforth we shall consider $\widetilde{A}$ as a right $A$-module, so we will write $\widetilde{A}_A$. 
\end{definition}

%\begin{empt}\label{dir_sum_constr}
%	Let $\left(A, \widetilde{A}, G\right)$ be a noncommutative finite covering. If $\widetilde{a} \in \widetilde{A}$ then $\widetilde{a} = a + p$ where $a = \frac{1}{|G|}\sum_{g \in G}g\widetilde{a}$ and $p = \widetilde{a}-a$. It is clear that $\sum_{g \in G}gp = 0$ and for any $G$-invariant $b \in \widetilde{A}$ we have 
%	\begin{equation*}
%	\langle b, p \rangle_{\widetilde{A}} = \frac{1}{|G|}\widetilde{b}\sum_{g \in G}gp = 0.
%	\end{equation*}
%	Otherwise the set of $G$-invariant elements is just a sublagebra $A \subset \widetilde{A}$.
%	So  $\widetilde{A}_A$ can be decomposed into the direct orthogonal sum, i.e.
%	\begin{equation}\label{hilb_mod_direct_sum}
%	\begin{split}
%	\widetilde{A}_A= A \oplus P; \  
%	A \perp P \ , \text{i.e. } \langle \widetilde{b}, p \rangle_{\widetilde{A}}= 0; \text{ for any } a \in A; \ p \in P. 
%	\end{split}
%	\end{equation}
%\end{empt}

\section{Noncommutative infinite coverings}\label{bas_constr}

\paragraph*{}
This section contains a noncommutative generalization of infinite coverings.
\begin{definition}\label{comp_defn}
	Let
	\begin{equation*}
	\mathfrak{S} =\left\{ A =A_0 \xrightarrow{\pi_1} A_1 \xrightarrow{\pi_2} ... \xrightarrow{\pi_n} A_n \xrightarrow{\pi^{n+1}} ...\right\}
	\end{equation*}
	be a sequence of $C^*$-algebras and noncommutative finite-fold coverings such that:
	\begin{enumerate}
		\item[(a)] Any composition $\pi_{n_1}\circ ...\circ\pi_{n_0+1}\circ\pi_{n_0}:A_{n_0}\to A_{n_1}$ corresponds to the noncommutative covering $\left(A_{n_0}, A_{n_1}, G\left(A_{n_1}~|~A_{n_0}\right)\right)$;
		\item[(b)] If $k < l < m$ then $G\left( A_m~|~A_k\right)A_l = A_l$ (Action of $G\left( A_m~|~A_k\right)$ on $A_l$ means that $G\left( A_m~|~A_k\right)$ acts on $A_m$, so $G\left( A_m~|~A_k\right)$ acts on $A_l$ since $A_l$ a subalgebra of $A_m$);
		\item[(c)] If $k < l < m$ are nonegative integers then there is the natural exact sequence of covering transformation groups
		\begin{equation*}
		\{e\}\to G\left(A_{m}~|~A_{l}\right) \xrightarrow{\iota} G\left(A_{m}~|~A_{k}\right)\xrightarrow{\pi}G\left(A_{l}~|~A_{k}\right)\to\{e\}
		\end{equation*}
		where the existence of the homomorphism $G\left(A_{m}~|~A_{k}\right)\xrightarrow{\pi}G\left(A_{l}~|~A_{k}\right)$ follows from (b).
		
	\end{enumerate}
	The sequence
	$\mathfrak{S}$
	is said to be an \textit{(algebraical)  finite covering sequence}. 
	For any finite covering sequence we will use the notation $\mathfrak{S} \in \mathfrak{FinAlg}$.
\end{definition}
\begin{definition}\label{equiv_act_defn}
	Let $\widehat{A} = \varinjlim A_n$  be the $C^*$-inductive limit \cite{murphy}, and suppose that $\widehat{G}= \varprojlim G\left(A_n~|~A \right) $ is the projective limit of groups \cite{spanier:at}. There is the natural action of $\widehat{G}$ on $\widehat{A}$. A non-degenerate faithful representation $\widehat{A} \to B\left( \H\right) $ is said to be \textit{equivariant} if there is an action of $\widehat{G}$ on $\H$ such that for any $\xi \in \H$ and $g \in  \widehat{G}$ following condition holds
	\begin{equation}\label{equiv_act_eqn}
	\left(ga \right) \xi = g\left(a\left(g^{-1}\xi \right)  \right) .
	\end{equation}
\end{definition}
\begin{definition}\label{special_el_defn}
	Let $\pi:\widehat{A} \to B\left( \H\right) $ be an equivariant representation.  A positive element  $\overline{a}  \in B\left(\H \right)_+ $ is said to be \textit{special} (with respect to $\mathfrak S$) if following conditions hold:
	\begin{enumerate}
		\item[(a)] For any $n \in \mathbb{N}^0$  the following  series 
		\begin{equation*}
		\begin{split}
		a_n = \sum_{g \in \ker\left( \widehat{G} \to  G\left( A_n~|~A \right)\right)} g  \overline{a} 
		\end{split}
		\end{equation*}
		is strongly convergent and the sum lies in   $A_n$, i.e. $a_n \in A_n $;		
		\item[(b)]
	If $f_\eps: \R \to \R$ is given by 
	\begin{equation}\label{f_eps_eqn}
	f_\eps\left( x\right)  =\left\{
	\begin{array}{c l}
	0 &x \le \eps \\
	x - \eps & x > \eps
	\end{array}\right.
	\end{equation}
	then for any $n \in \mathbb{N}^0$ and for any $z \in A$   following  series 
		\begin{equation*}
		\begin{split}
		b_n = \sum_{g \in \ker\left( \widehat{G} \to  G\left( A_n~|~A \right)\right)} g \left(z  \overline{a} z^*\right) ,\\
		c_n = \sum_{g \in \ker\left( \widehat{G} \to  G\left( A_n~|~A \right)\right)} g \left(z  \overline{a} z^*\right)^2,\\
		d_n = \sum_{g \in \ker\left( \widehat{G} \to  G\left( A_n~|~A \right)\right)} g f_\eps\left( z  \overline{a} z^* \right) 
		\end{split}
		\end{equation*}
		are strongly convergent and the sums lie in   $A_n$, i.e. $b_n,~ c_n,~ d_n \in A_n $; 
		\item[(c)] For any $\eps > 0$ there is $N \in \N$ (which depends on $\overline{a}$ and $z$) such that for any $n \ge N$ a following condition holds
		\begin{equation}\label{square_condition_equ}
		\begin{split}
		\left\| b_n^2 - c_n\right\| < \eps.
		\end{split}
		\end{equation}	
	\end{enumerate}
	
	An element  $\overline{   a}' \in B\left( \H\right) $ is said to be \textit{weakly special} if 
	$$
	\overline{   a}' = x\overline{a}y; \text{ where }   x,y \in \widehat{A}, \text{ and } \overline{a} \in B\left(\H \right)  \text{ is special}.
	$$
	
\end{definition}
\begin{lemma}\cite{ivankov:qnc}\label{stong_conv_inf_lem}
	If $\overline{a} \in B\left( \H\right)_+$ is a special element and ${G}_n=\ker\left( \widehat{G} \to  G\left( A_n~|~A \right)\right)$ then from
	\begin{equation*}
	\begin{split}
	a_n = \sum_{g \in {G}_n} g \overline{a},
	\end{split}
	\end{equation*}
	it follows that $\overline{a} = \lim_{n \to \infty} a_n$ in the sense of the strong convergence. Moreover  one has $\overline{a} =\inf_{n \in \N}a_n$.
\end{lemma}
\begin{corollary}\cite{ivankov:qnc}\label{special_cor}
	Any weakly special element lies in the enveloping von Neumann algebra $\widehat{A}''$ of $\widehat{A}=\varinjlim A_n$. If $\overline{A}_\pi \subset B\left( \H\right)$ is the $C^*$-norm completion of an algebra generated by weakly special elements then $\overline{A}_\pi \subset \widehat{A}''$.
\end{corollary}
\begin{lemma}\cite{ivankov:qnc}
	If $\overline{a}\in B\left( \H\right)$ is special, (resp.  $\overline{a}'\in B\left( \H\right)$ weakly special) then for any  $g \in \widehat{G}$ the element  $g\overline{a}$ is special, (resp. $g\overline{a}'$ is weakly special).
\end{lemma}
\begin{corollary}\label{disconnect_group_action_cor}\cite{ivankov:qnc}
	If $\overline{A}_\pi \subset B\left( \H\right)$ is the $C^*$-norm completion of algebra generated by weakly special elements, then there is a natural action of $\widehat{G}$ on $\overline{A}_\pi$.
\end{corollary}

\begin{definition}\label{main_defn_full}
	Let $\mathfrak{S} =\left\{ A =A_0 \xrightarrow{\pi^1} A_1 \xrightarrow{\pi^2} ... \xrightarrow{\pi^n} A_n \xrightarrow{\pi^{n+1}} ...\right\}$ be an  algebraical  finite covering sequence. Let  $\pi:\widehat{A} \to B\left( \H\right) $ be an equivariant representation.	 Let $\overline{A}_\pi \subset B\left( \H\right)$ be the $C^*$-norm completion of algebra generated by weakly special elements. We say that $\overline{A}_\pi$ is the {\it disconnected inverse noncommutative limit} of $\downarrow\mathfrak{S}$ (\textit{with respect to $\pi$}). 
	The triple  $\left(A, \overline{A}_\pi, G\left(\overline{A}_\pi~|~ A\right)\stackrel{\mathrm{def}}{=} \widehat{G}\right)$ is said to be the  {\it disconnected infinite noncommutative covering } of $\mathfrak{S}$ (\textit{with respect to $\pi$}). If $\pi$ is the universal representation then "with respect to $\pi$" is dropped and we will write
	$\left(A, \overline{A}, G\left(\overline{A}~|~ A\right)\right)$.	
\end{definition}
\begin{definition}\label{main_sdefn}
	A  maximal irreducible subalgebra $\widetilde{A}_\pi \subset \overline{A}_\pi$  is said to be a {\it connected component} of $\mathfrak{S}$ ({\it with respect to $\pi$}). The maximal subgroup $G_\pi\subset G\left(\overline{A}_\pi~|~ A\right)$ among subgroups $G\subset G\left(\overline{A}_\pi~|~ A\right)$ such that $G\widetilde{A}_\pi=\widetilde{A}_\pi$  is said to be the $\widetilde{A}_\pi$-{\it  invariant group} of $\mathfrak{S}$. If $\pi$ is the universal representation then "with respect to $\pi$" is dropped. 
\end{definition}

\begin{remark}
	From the Definition \ref{main_sdefn} it follows that $G_\pi \subset G\left(\overline{A}_\pi~|~ A\right)$ is a normal subgroup.
\end{remark}
\begin{definition}\label{good_seq_defn} Let $$\mathfrak{S} = \left\{ A =A_0 \xrightarrow{\pi^1} A_1 \xrightarrow{\pi^2} ... \xrightarrow{\pi^n} A_n \xrightarrow{\pi^{n+1}} ...\right\} \in \mathfrak{FinAlg},$$ and let $\left(A, \overline{A}_\pi, G\left(\overline{A}_\pi~|~ A\right)\right)$ be a disconnected infinite noncommutative covering of $\mathfrak{S}$ with respect to an equivariant representation $\pi: \varinjlim A_n\to B\left(\H \right) $. Let $\widetilde{A}_\pi\subset \overline{A}_\pi$  be a connected component of $\mathfrak{S}$ with respect to $\pi$, and let $G_\pi \subset  G\left(\overline{A}_\pi~|~ A\right)$  be the $\widetilde{A}_\pi$ -  invariant group of $\mathfrak{S}$.
	Let  $h_n : G\left(\overline{A}_\pi~|~ A\right) \to  G\left( A_n~|~A \right)$ be the natural surjective  homomorphism. The  representation $\pi: \varinjlim A_n\to B\left(\H \right)$ is said to be \textit{good} if it satisfies to following conditions:
	\begin{enumerate}
		\item[(a)] The natural *-homomorphism $ \varinjlim A_n \to  M\left(\widetilde{A}_\pi \right)$ is injective,
		\item[(b)] If $J\subset G\left(\overline{A}_\pi~|~ A\right)$ is a set of representatives of $G\left(\overline{A}_\pi~|~ A\right)/G_\pi$, then the algebraic direct sum
		\begin{equation*}
		\bigoplus_{g\in J} g\widetilde{A}_\pi
		\end{equation*}
		is a dense subalgebra of $\overline{A}_\pi$,
		\item [(c)] For any $n \in \N$ the restriction $h_n|_{G_\pi}$ is an epimorphism, i. e. $h_n\left(G_\pi \right) = G\left( A_n~|~A \right)$.
	\end{enumerate}
	If $\pi$ is the universal representation we say that $\mathfrak{S}$ is \textit{good}.
\end{definition}

\begin{definition}\label{main_defn}
	Let $\mathfrak{S}=\left\{A=A_0 \to A_1 \to ... \to A_n \to ...\right\} \in \mathfrak{FinAlg}$ be  an algebraical  finite covering sequence. Let $\pi: \widehat{A} \to B\left(\H \right)$ be a good representation.    A connected component $\widetilde{A}_\pi \subset \overline{A}_\pi$  is said to be the {\it inverse noncommutative limit of $\downarrow\mathfrak{S}$ (with respect to $\pi$)}. The $\widetilde{A}_\pi$-invariant group $G_\pi$  is said to be the {\it  covering transformation group of $\mathfrak{S}$} ({\it with respect to $\pi$}).  The triple $\left(A, \widetilde{A}_\pi, G_\pi\right)$ is said to be the  {\it infinite noncommutative covering} of $\mathfrak{S}$  ({\it with respect to $\pi$}). %The algebra $A$  is said to be the {\it base algebra} of $\mathfrak{S}$. 
	We will use the following notation 
	\begin{equation*}
	\begin{split}
	\varprojlim_\pi \downarrow \mathfrak{S}\stackrel{\mathrm{def}}{=}\widetilde{A}_\pi,\\
	G\left(\widetilde{A}_\pi~|~ A\right)\stackrel{\mathrm{def}}{=}G_\pi.
	\end{split}
	\end{equation*}	
	If $\pi$ is the universal representation then "with respect to $\pi$" is dropped and we will write $\left(A, \widetilde{A}, G\right)$, $~\varprojlim \downarrow \mathfrak{S}\stackrel{\mathrm{def}}{=}\widetilde{A}$ and  $ G\left(\widetilde{A}~|~ A\right)\stackrel{\mathrm{def}}{=}G$.
\end{definition}
%\begin{remark}
%	Above definition does not depend on choice of $\widetilde{A}_\pi$ up to isomorphism.
%\end{remark}

\section{Quantization of topological coverings}\label{top_chap}

\paragraph*{} The described in the Sections \ref{ff_c_sec} and \ref{bas_constr} theory is motivated by a  quantization of topological coverings. Following theorem provides a purely algebraic construction of topological finite-fold coverings. 
\begin{theorem}\label{comm_fin_thm}\cite{ivankov:qnc}
	If $\mathcal X$, $\widetilde{\mathcal X}$  are  locally compact spaces, and  $\pi: \widetilde{\mathcal X}\to \mathcal X$ is a surjective continuous map, then following conditions are equivalent:
	\begin{enumerate}
		\item [(i)] The map $\pi: \widetilde{\mathcal X}\to \mathcal X$ is a finite-fold covering with a compactification,
		\item[(ii)] There is a natural  noncommutative finite-fold covering $\left(C_0\left(\mathcal  X \right), C_0\left(\widetilde{\mathcal X} \right), G    \right)$.
	\end{enumerate}
\end{theorem}
\begin{remark}
	The definition of a covering with a compactification is presented in \cite{ivankov:qnc}.
\end{remark}
\begin{empt}\label{top_inf_constr}
An algebraic construction of infinite coverings can be given by an "infinite composition" of finite ones. Suppose that there is a sequence
\begin{equation}
\mathfrak{S}_\mathcal{X} = \left\{\mathcal{X}_0 \xleftarrow{}... \xleftarrow{} \mathcal{X}_n \xleftarrow{} ... \right\}
\end{equation}
of topological spaces and finite-fold coverings, and let us consider a following diagram
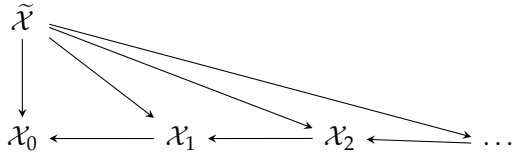
\newline
\begin{tikzpicture}
\matrix (m) [matrix of math nodes,row sep=3em,column sep=4em,minimum width=2em]
{
	\widetilde{\mathcal{X}}	&   &  & \\
	\mathcal{X}_0	& \mathcal{X}_1 &  \mathcal{X}_2 & \dots \\};
\path[-stealth]
(m-1-1) edge node [left] {} (m-2-1)
(m-1-1) edge node [left] {} (m-2-2)
(m-1-1) edge node [left] {} (m-2-3)
(m-1-1) edge node [left] {} (m-2-4)
(m-2-2) edge node [left] {} (m-2-1)
(m-2-4) edge node [left] {} (m-2-3)
(m-2-3) edge node [left] {} (m-2-2);
\end{tikzpicture}
\newline
where all arrows are coverings. It is proven in \cite{ivankov:qnc} that there is "minimal" $\widetilde{\mathcal X}$ which satisfies to the above diagram. This minimal  $\widetilde{\mathcal X}$ is said to be the \textit{topological inverse limit} of $\mathfrak{S}_\mathcal{X} = \left\{\mathcal{X}_0 \xleftarrow{}... \xleftarrow{} \mathcal{X}_n \xleftarrow{} ... \right\}$ and it is denoted by $\varprojlim\downarrow \mathfrak{S}_{\mathcal{X}}$. The topological inverse limits are fully described in \cite{ivankov:qnc}. Any inverse limit yields an infinite topological covering $\varprojlim\downarrow \mathfrak{S}_{\mathcal{X}}\to \mathcal X$. Denote by $G\left(\varprojlim \downarrow \mathfrak{S}_{\mathcal{X}}~|~ \mathcal X\right)$ the group of covering transformations of the covering $\varprojlim \downarrow \mathfrak{S}_{\mathcal{X}}\to\mathcal X$. The following theorem gives an algebraic construction of the topological inverse limit.
\end{empt}

\begin{theorem}\label{comm_main_thm}\cite{ivankov:qnc}
	If $\mathfrak{S}_{\mathcal X} = \left\{\mathcal{X} = \mathcal{X}_0 \xleftarrow{}... \xleftarrow{} \mathcal{X}_n \xleftarrow{} ...\right\} $ is the sequence of topological spaces and coverings and
	$$\mathfrak{S}_{C_0\left(\mathcal{X}\right)}=
	\left\{C_0(\mathcal{X})=C_0(\mathcal{X}_0)\to ... \to C_0(\mathcal{X}_n) \to ...\right\} \in \mathfrak{FinAlg}$$ is an algebraical  finite covering sequence then following conditions hold:
	\begin{enumerate}
		\item [(i)] $\mathfrak{S}_{C_0\left(\mathcal{X}\right)}$ is good,
		\item[(ii)] There are  isomorphisms:

		\begin{itemize}
			\item $\varprojlim \downarrow \mathfrak{S}_{C_0\left(\mathcal{X}\right)} \approx C_0\left(\varprojlim \downarrow \mathfrak{S}_{\mathcal X}\right)$;
			\item $G\left(\varprojlim \downarrow \mathfrak{S}_{C_0\left(\mathcal{X}\right)}~|~ C_0\left(\mathcal X\right)\right) \approx G\left(\varprojlim \downarrow \mathfrak{S}_{\mathcal{X}}~|~ \mathcal X\right)$.
		\end{itemize}
	\end{enumerate}
	
\end{theorem}

\section{Operator algebras of foliations}\label{foliations}

\begin{definition}
	
	Let $M$ be a smooth manifold and $TM$ its tangent bundle, so that
	for each $x \in M$, $T_x M$ is the tangent space of $M$ at $x$. A
	smooth subbundle $\mathcal{F}$ of $TM$ is called {\it integrable} if and only if one of
	the following equivalent conditions is satisfied:
	
	\smallskip
	
	\begin{enumerate}
		
		\item[(a)] Every $x \in M$ is contained in a submanifold $W$ of $M$ such that
		$$
		T_y (W) = \mathcal{F}_y \qquad \forall \, y \in W \, ,
		$$
		
		\smallskip
		
		\item[(b)] Every $x \in M$ is in the domain $U \subset M$ of a
		submersion $p : U \to {\mathbb R}^q$ ($q = {\rm codim} \, \mathcal{F}$) with
		$$
		\mathcal{F}_y = {\rm Ker} (p_*)_y \qquad \forall \, y \in U \, ,
		$$
		
		\smallskip
		\item[(c)] $C^{\infty} \left( \mathcal{F}\right)  = \{ X \in C^{\infty} \left(TM\right) \, , \ X_x \in
		\mathcal{F}_x \quad \forall \, x \in M \}$ is a Lie algebra,
		
		\smallskip
		
		\item[(d)] The ideal $J\left( \mathcal{F}\right) $ of smooth exterior differential forms which
		vanish on $\mathcal{F}$ is stable by exterior differentiation.
	\end{enumerate}
	
\end{definition}

A foliation of $M$ is given by an integrable subbundle $\mathcal{F}$ of $TM$.
The leaves of the foliation $\left(M , \mathcal F\right)$ are the maximal connected
submanifolds $L$ of $M$ with $T_x (L) = \mathcal{F}_x $, $\forall \, x \in L$,
and the partition of $M$ in leaves $$M = \cup
L_{\alpha}\,,\quad\alpha \in X$$ is characterized geometrically by
its ``local triviality'': every point $x \in M$ has a neighborhood
$\mathcal U$ and a system of local coordinates
$(x^j)_{j = 1 , \ldots , \dim V}$ called
{\it foliation charts}, so
that the partition of $\mathcal U$ in connected components of
leaves corresponds to the partition of 
\begin{equation*}
{\mathbb
	R}^{\dim M} = {\mathbb R}^{\dim \mathcal F} \times {\mathbb R}^{\text{codim}
	\, \mathcal F}
\end{equation*}
in the parallel affine subspaces 
$
{\mathbb R}^{\dim \mathcal F}
\times {\rm pt}$.
The corresponding foliation will be denoted by
\begin{equation}\label{fol_chart_eqn}
\left(\R^n, \mathcal{F}_p \right) 
\end{equation}
where $p = \dim   \mathcal{F}_p$.
To each foliation $(M, \mathcal{F})$ is canonically associated a $C^*$- algebra
$C^*_r (M, ~\mathcal{F})$ which encodes the topology of the space of leaves.  To
take this into account one first constructs a manifold $\mathcal G$, $\dim
\, \mathcal G = \dim \,M + \dim \,\mathcal F$, called the graph (or \textit{holonomy groupoid})
of the foliation, which refines the equivalence relation coming from
the partition of $M$ in leaves $M = \cup L_{\alpha}$. 
An element $\gamma$ of $\mathcal G$ is given by two points $x = s(\gamma)$,
$y = r(\gamma)$ of $M$ together with an equivalence class of smooth
paths: $\gamma (t)\in M$, $t \in [0,1]$; $\gamma (0) = x$, $\gamma
(1) = y$, tangent to the bundle $\mathcal{F}$ ( i.e. with $\dot\gamma (t)
\in \mathcal{F}_{\gamma (t)}$, $\forall \, t \in {\mathbb R}$) up to the
following equivalence: $\gamma_1$ and $\gamma_2$ are equivalent if and only if
the {\it holonomy} of the path $\gamma_2 \circ \gamma_1^{-1}$ at the
point $x$ is the {\it identity}. The graph $\mathcal G$ has an obvious
composition law. For $\gamma , \gamma' \in G$, the composition
$\gamma \circ \gamma'$ makes sense if $s(\gamma) = r(\gamma')$. If
the leaf $L$ which contains both $x$ and $y$ has no holonomy, then
the class in $\mathcal G$ of the path $\gamma (t)$ only depends on the pair
$(y,x)$. In general, if one fixes $x = s(\gamma)$, the map from $\mathcal G_x
= \{ \gamma , s(\gamma) = x \}$ to the leaf $L$ through $x$, given
by $\gamma \in \mathcal G_x \mapsto y = r(\gamma)$, is the holonomy covering
of $L$.
Both maps $r$ and $s$ from the manifold $\mathcal G$ to $M$ are smooth
submersions and the map $(r,s)$ to $M \times M$ is an immersion
whose image in $M \times M$ is the (often singular) subset
\begin{equation*}\label{subset}
\{ (y,x)\in M \times M: \, \text{ $y$ and $x$ are on the same leaf}\}.
\end{equation*}
We
assume, for notational convenience, that the manifold $\mathcal G$ is
Hausdorff, but as this fails to be the case in very
interesting examples I shall refer to \cite{connes:foli_survey} for the removal of this hypothesis.  For
$x\in M$ one lets $\Omega_x^{1/2}$ be the one dimensional complex
vector space of maps from the exterior power $\wedge^k \,  \mathcal{F}_x$, $k =
\dim F$, to ${\mathbb C}$ such that
$$
\rho \, (\lambda \, v) = \vert \lambda \vert^{1/2} \, \rho \, (v)
\qquad \forall \, v \in \wedge^k \,  \mathcal{F}_x \, , \quad \forall \,
\lambda \in {\mathbb R} \, .
$$
Then, for $\gamma \in\mathcal G$, one can identify $\Omega_{\gamma}^{1/2}$ with the one
dimensional complex vector space $\Omega_y^{1/2} \otimes
\Omega_x^{1/2}$, where $\gamma : x \to y$. In other words
$$
\Omega_{\mathcal G}^{1/2}=\, r^*(\Omega_M^{1/2})\otimes s^*(\Omega_M^{1/2})\,.
$$
Of course the bundle $\Omega_M^{1/2}$ is trivial on $M$, and we
could choose once and for all  a trivialisation $\nu$ turning
elements of $C_c^{\infty} (\mathcal G , \Omega^{1/2})$ into functions.
Let us
however stress that the use of half densities makes all the
construction completely canonical.
For $f,g \in C_c^{\infty} (\mathcal G , \Omega^{1/2})$, the convolution
product $f * g$ is defined by the equality
$$
(f * g) (\gamma) = \int_{\gamma_1 \circ \gamma_2 = \gamma}
f(\gamma_1) \, g(\gamma_2) \, .
$$
This makes sense because, for fixed $\gamma : x \to y$ and fixing $v_x
\in \wedge^k \,  \mathcal{F}_x$ and $v_y \in \wedge^k \,  \mathcal{F}_y$, the product
$f(\gamma_1) \, g(\gamma_1^{-1} \gamma)$ defines a $1$-density on
$G^y = \{ \gamma_1 \in G , \, r (\gamma_1) = y \}$, which is smooth
with compact support (it vanishes if $\gamma_1 \notin\supp f$),
and hence can be integrated over $G^y$ to give a scalar, namely $(f * g)
(\gamma)$ evaluated on $v_x , v_y$.
The $*$ operation is defined by $f^* (\gamma) =
\overline{f(\gamma^{-1})}$,  i.e. if $\gamma : x \to y$ and
$v_x \in \wedge^k \, \mathcal{F}_x$, $v_y \in \wedge^k \, \mathcal{F}_y$ then $f^*
(\gamma)$ evaluated on $v_x , v_y$ is equal to
$\overline{f(\gamma^{-1})}$ evaluated on $v_y , v_x$. We thus get a
$*$-algebra $C_c^{\infty} (\mathcal G , \Omega^{1/2})$. For each leaf $L$ of
$(M, \mathcal{F})$ one has a natural representation of this $*$-algebra on the
$L^2$ space of the holonomy covering $\tilde L$ of $L$. Fixing a
base point $x \in L$, one identifies $\tilde L$ with $\mathcal G_x = \{
\gamma , s(\gamma) = x \}$ and defines
\begin{equation}\label{fol_repr}
(\rho_x (f) \, \xi) \, (\gamma) = \int_{\gamma_1 \circ \gamma_2 =
	\gamma} f(\gamma_1) \, \xi (\gamma_2) \qquad \forall \, \xi \in L^2
(\mathcal G_x),\
\end{equation}

where $\xi$ is a square integrable half density on $\mathcal G_x$. Given
$\gamma : x \to y$ one has a natural isometry of $L^2 (\mathcal G_x)$ on $L^2
(G_y)$ which transforms the representation $\rho_x$ in $\rho_y$.
By definition $C^*_r (M, \mathcal{F})$ is the $C^*$-algebra completion of
$C_c^{\infty} (\mathcal G , \Omega^{1/2})$ with the norm 
\begin{equation}\label{fol_norm_eqn}
\Vert f \Vert =
\sup_{x \in M} \, \Vert \rho_x (f) \Vert\,.
\end{equation}
Denote by $\mathcal G(M, \mathcal{F})$ the foliation groupoid of  $(M, \mathcal{F})$.

\begin{example}\label{fol_tor_exm} \emph{Linear foliation on torus.}
	Consider a vector field $\tilde{X}$ on $\R^2$ given by
	\[
	\tilde{X}=\alpha\frac{\partial}{\partial
		x}+\beta\frac{\partial}{\partial y}
	\]
	with constant $\alpha$ and $\beta $. Since $\tilde{X}$ is
	invariant under all translations, it determines a vector field $X$
	on the two-dimensional torus ${\T}^2={\R}^2/{\Z}^2$. The vector
	field $X$ determines a foliation $\mathcal{F}$ on ${\T}^2$. The leaves of
	$\mathcal{F}$ are the images of the parallel lines
	$\tilde{L}=\{(x_0+t\alpha, y_0+t\beta): t\in\R\}$ with the slope
	$\theta=\beta/\alpha $ under the projection $\R^2\to \T^2$.
	In the case when $\theta$ is rational, all leaves of $\mathcal{F}$ are
	closed and are circles, and the foliation $\mathcal{F}$ is determined by
	the fibers of a fibration $\T^2\to S^1$. In the case when $\theta$
	is irrational, all leaves of $\mathcal{F}$ are everywhere dense in $\T^2$. Denote by $\left(\T^2, \mathcal{F}_\th \right)$ this foliation. 
\end{example}

\section{Strong Morita equivalence}\label{s:morita}
\paragraph*{}
The notion of the strong Morita
equivalence was introduced by Rieffel.

\begin{defn}\cite{Rieffel74}
	Let $A$ and $B$ be $C^*$-algebras. An $A$-$B$-equivalence bimodule
	is an $A$-$B$-bimodule $X$, endowed with $A$-valued and $B$-valued
	inner products $\langle \cdot, \cdot\rangle_A$ and $\langle \cdot,
	\cdot\rangle_B$ accordingly, such that $X$ is a right Hilbert
	$B$-module and a left Hilbert $A$-module with respect to these
	inner products, and, moreover,
	\begin{enumerate}
		\item $\langle x,y\rangle_A z=x\langle y, z\rangle_B$  for any
		$x, y, z\in X$;
		\item The set $\langle X,X\rangle_A$ generates a dense subset in
		$A$, and the set $\langle X,X\rangle_B$ generates a dense subset in
		$B$.
	\end{enumerate}
	
	We call algebras $A$ and $B$ \textit{strongly Morita equivalent}, if there
	is an $A$-$B$-equivalence bimodule.
\end{defn}
\begin{example}\label{fol_trivial_exm}\cite{kord:foli}
	If $\mathcal{F}$ is a simple foliation given by a submersion $M\to B$, then
	the $C^*$-algebra $C^*_r(M,\mathcal{F})$ is strongly Morita equivalent to
	the $C^*$-algebra $C_0(B)$.
\end{example}

The following theorem yields a relation between the strong Morita
equivalence and the  stable equivalence.

\begin{thm}\label{MoritaK_thm}\cite{BroGreRie}
	Let $A$ and $B$ are $C^*$-algebras with countable approximate
	units. Then these algebras are strongly Morita equivalent if and
	only if they are stably equivalent, i.e. $A\otimes \mathcal{K} \cong
	B\otimes \mathcal{K}$, where $\mathcal{K}$ denotes the algebra of compact
	operators in a separable Hilbert space.
\end{thm}
\begin{remark}\label{stable_rem}
	The $C^*$-algebra $A\otimes \mathcal{K}$ means the $C^*$-norm completion of the algebraic tensor product of  $A$ and $\mathcal{K}$. In general the  $C^*$-norm completion of algebraic tensor product $A$ and $B$ is not unique, because it depends on the $C^*$-norm on  $A\otimes B$. However in \cite{blackadar:ko} it is proven that $\mathcal K$ is a nuclear $C^*$-algebra, so there is the unique $C^*$-norm on the algebraic tensor product of $A$ and $\mathcal{K}$. Hence there is the unique $C^*$-norm completion $A\otimes \mathcal{K}$.
\end{remark}
\begin{example}\label{fol_trivial_k_exm}\cite{kord:foli}
	If $\mathcal{F}$ is a simple foliation given by a bundle $M\to B$, then following condition holds
\begin{equation}\label{fol_stab_eqn}
C^*_r\left( M,\mathcal{F}\right) \approx C_0\left( B\right) \otimes \mathcal{K}
\end{equation}
(cf. Example \ref{fol_trivial_exm}).

\end{example}

%\begin{example}
%	If $\mathcal{F}$ is a simple foliation given by a bundle $M\to B$, then
%	the $C^*$-algebra $C^*_r(M,\mathcal{F})$ is strongly Morita equivalent to
%	the $C^*$-algebra $C_0(B)$.
%\end{example}

\begin{example} \label{ex:morita}\cite{kord:foli}
	Consider a compact foliated manifold $(M,\mathcal{F})$. As usual, let $\mathcal G$
	denote the holonomy groupoid of $\mathcal{F}$. For any subsets $A,
	B\subset M$, denote
	\[
	\mathcal G^A_B=\{\gamma\in \mathcal G : r(\gamma)\in A, s(\gamma)\in B\}.
	\]
	In particular,
	\[
\mathcal	G^M_T=\{\gamma\in \mathcal G : s(\gamma)\in T\}.
	\]
	If $T$ is a transversal, then $\mathcal G^T_T$ is a submanifold and a
	subgroupoid in $\mathcal G$. Let $C^*_r(\mathcal G^T_T)$ be the reduced
	$C^*$-algebra of this groupoid. As shown in \cite{Hil-Skan83}, if
	$T$ is a complete transversal, then the algebras $C^*_r(\mathcal G)$ and
	$C^*_r(\mathcal G^T_T)$ are strongly Morita equivalent. In particular, this
	implies that
	\[
	C^*_r(\mathcal G) \otimes \mathcal{K} \cong  C^*_r(\mathcal G^T_T) \otimes \mathcal{K}.
	\]
In \cite{Hil-Skan83} it is proven that $	C^*_r(\mathcal G) $ is stable, i.e. $	C^*_r(\mathcal G)  \cong 	C^*_r(\mathcal G) \otimes \mathcal{K}$, so what
 	implies that
 \[
 C^*_r(\mathcal G) \cong  C^*_r(\mathcal G^T_T) \otimes \mathcal{K}.
 \]
\end{example}

\begin{example}\label{ex:Atheta}
	Consider the   linear foliation  $\mathcal{F}_\theta$ on the two-dimensional
	torus $\T^2$ (cf. Example \ref{fol_tor_exm}), where $\theta\in \R$ is a fixed irrational number.
	If we choose the transversal $T$ given by the equation $y=0$, then
	the leaf space of the foliation $\mathcal{F}_\theta$ is identified with
	the orbit space of the $\Z$-action on the circle $S^1=\R/\Z$
	generated by the rotation
	\[
	R_\theta(x)=x+\theta \mod 1,\quad x\in S^1.
	\]
	Elements of the algebra $C^\infty_c(\mathcal G^T_T)$ are determined by
	matrices $a(i,j)$, where the indices $(i,j)$ are arbitrary pairs
	of elements $i$ and $j$ of $T$, lying on the same leaf of $\mathcal{F}$,
	that is, on the same orbit of the $\Z$-action $R_\theta$. Since
	in this case the leafwise equivalence relation on the transversal
	is given by a free group action, the algebra $C^*_r(\mathcal G^T_T)=C\left( \T^2_\theta\right) $
	coincides with the crossed product $C(S^1)\rtimes \Z$ of the
	algebra $C(S^1)$ by the group $\Z$ with respect to the
	$\Z$-action $R_\theta$ on $C(S^1)$. Therefore  every element of $C\left( \T^2_\theta\right)$ is given by
	a power series
	\[
	a=\sum_{n\in\Z}a_nU^n, \quad a_n\in C(S^1),
	\]
	the multiplication is given by
	\[
	(aU^n)(bU^m)=a(b\circ R^{-1}_{n\theta})U^{n+m}
	\]
	and the involution by
	\[
	(aU^n)^*=\bar{a}U^{-n}.
	\]
	The algebra $C(S^1)$ is generated by the function $V$ on $S^1$
	defined as
	\[
	V(x)=e^{2\pi i x}, \quad x\in S^1.
	\]
	Hence, the algebra $C\left(\T^2_\th \right) $ is generated by two elements $U$ and
	$V$, satisfying the relation
	\[
	VU=\lambda U V, \quad \lambda = e^{2\pi i \theta}.
	\]
	Thus, for example, a general element of $C^\infty_c(\mathcal G^T_T)$ can be
	represented as a power series
	\[
	a=\sum_{(n,m)\in\Z^2}a_{nm}U^nV^m,
	\]
	where $a_{nm}\in \SS(\Z^2)$ is a rapidly decreasing sequence
	(that is, for any natural $k$ we have
	$\sup_{(n,m)\in\Z^2}(|n|+|m|)^k|a_{nm}|<\infty$). Since, in the commutative case ($\theta=0$), the above description
	defines the algebra $C^\infty\left( \T^2\right) $ of smooth functions on the two-dimensional
	torus, the $C^*$-norm completion of $C^\infty\left( \T^2_\theta\right) $ is   called the algebra of continuous
	functions on a noncommutative torus and denoted by $C(\T^2_\theta)$. Otherwise from  
	$C^\infty_c\left(\mathcal  G^T_T\right) = C^\infty\left( \T^2_\theta\right) $ and from the Example \ref{ex:morita}  it follows that the $C^*$-algebra $C^*_r\left( \T^2, \mathcal F_\th\right) $  is
	strongly Morita equivalent to $C\left( \T^2_\theta\right) $. From the Theorem \ref{MoritaK_thm} it follows that
\begin{equation}\label{fol_th_eqn}
	C^*_r\left( \T^2, \mathcal F_\th\right)  \approx C\left( \T^2_\theta\right) \otimes \mathcal K. 
\end{equation}
\end{example}

\section{Coverings of stable algebras}
\paragraph*{}
Here we find the relation between noncommutative coverings of $C^*$-algebras and noncommutative coverings of their stabilizations.
\subsection{Finite-fold coverings}
\paragraph*{} If $A$ is a $C^*$-algebra and $\mathcal K= \mathcal K\left( \H\right) $ is an algebra of compact operators then the $C^*$-norm completion of $A \otimes \mathcal K$ is said to be the \textit{stable algebra} of $A$ (cf. Remark \ref{stable_rem}). This completion we denote by   $A \otimes \mathcal K$. Any non-degenerated involutive action of $G$ on $A$ uniquely induces the non-degenerated involutive action  of $G$ on $A\otimes \mathcal K$.

\begin{theorem}\label{stable_fin_cov_thm}
	If $\left(A, \widetilde{A}, G \right)$ is a noncommutative finite-fold covering  then a triple $$\left(A  \otimes \mathcal K, \widetilde{A} \otimes \mathcal K, G \right)$$ is a noncommutative finite-fold covering.
\end{theorem}
\begin{proof}
	From the Definition \ref{fin_def} it follows that following conditions hold:
	\begin{enumerate}
		\item[(a)] 
		There are unital $C^*$-algebras $B$, $\widetilde{B}$  and inclusions 
		$A \subset B$,  $\widetilde{A}\subset \widetilde{B}$ such that $A$ (resp. $B$) is an essential ideal of $\widetilde{A}$ (resp. $\widetilde{B}$),
		\item[(b)] There is an unital  noncommutative finite-fold covering $\left(B ,\widetilde{B}, G \right)$,
		\item[(c)] 
		\begin{equation*}\label{wta_eqn1}
		\widetilde{A} =  \left\{a\in \widetilde{B}  ~|~ \left\langle \widetilde{B} ,a  \right\rangle_{\widetilde{B} } \in A \right\}.
		\end{equation*}
	\end{enumerate}
	Let $\mathcal K^+$ be the  unitisation of $\mathcal K$, i.e. underlying vector space of $\mathcal K^+$ is the direct sum $\mathcal K \bigoplus \C$. Algebra $A \otimes \mathcal K$ (resp. $\widetilde{A} \otimes \mathcal K$) is an essential ideal of $B \otimes \mathcal K^+$ (resp. $\widetilde{B} \otimes \mathcal K^+$), i.e. condition (a) of the Definition \ref{fin_def} holds. Since $\left(B ,\widetilde{B}, G \right)$ is an unital  noncommutative finite-fold covering the algebra $\widetilde{B}$ is a finitely generated $B$ module, i.e. there are $\widetilde{b}_1, \dots , \widetilde{b}_n \in \widetilde{B}$ such that any $\widetilde{b} \in\widetilde{B}$ is given by
	$$
	\widetilde{b} = \sum_{j = 1}^n \widetilde{b}_jb_j; \text{	where } b_j \in B.
	$$
	From the above equation it turns out that if
$$
\widetilde{b}^{\mathcal K} = \sum_{k = 1}^m \widetilde{b}_k \otimes x_k \in \widetilde{B} \otimes \mathcal K^+;~ \widetilde{b}_k \in \widetilde{B}, ~ x_k \in \mathcal K^+
$$
and 
$$
\widetilde{b}_k  = \sum_{j = 1}^n \widetilde{b}_jb_{kj}
$$
then
\begin{equation}\label{stab_tensor_eqn}
\begin{split}
\widetilde{b}^{\mathcal K} = \sum_{j = 1}^n \left( \widetilde{b}_j\otimes 1_{\mathcal K^+}\right) b_j^{\mathcal K},\\
\text{ where } b_j^{\mathcal K} = \sum_{k = 1}^m b_{kj}  \otimes x_k \in B \otimes \mathcal K.
\end{split}
\end{equation}
From \eqref{stab_tensor_eqn} and taking into account that the algebraic tensor product of $\widetilde{B}$ and $\mathcal K^+$ is dense in its $C^*$-norm completion, it turns out that  any
$\widetilde{b}^{\mathcal K}\in \widetilde{B} \otimes \mathcal K^+$ can be represented as
	$$
	\widetilde{b}^{\mathcal K} = \sum_{j = 1}^n \left( \widetilde{b}_j\otimes 1_{\mathcal K^+}\right) b^{\mathcal K}_j; \text{	where } b^{\mathcal K}_j \in B \otimes \mathcal K^+.
	$$ 
	It follows that $\widetilde{B} \otimes \mathcal K^+$ is a finitely generated $B \otimes \mathcal K^+$ module.
	From the Kasparov Stabilization Theorem \cite{blackadar:ko} it turns out that $\widetilde{B}\otimes \mathcal K^+$ is a projective $B\otimes \mathcal K^+$ module. So $\left(B\otimes \mathcal K^+, \widetilde {B}\otimes \mathcal K^+,G \right)$ is an unital finite-fold noncommutative covering, i.e. the condition (b) of the Definition \ref{fin_def} holds. Denote by $\widetilde{A}^{\mathcal K}$ a subalgebra given by \eqref{wta_eqn}, i.e.
	\begin{equation*}
	\widetilde{A}^{\mathcal K} = \left\{\widetilde{a}\in \widetilde{B} \otimes   \mathcal K^+ ~|~ \left\langle \widetilde{B} \otimes \mathcal K^+ ,\widetilde{a}  \right\rangle_{\widetilde{B}\otimes  \mathcal K^+ } \in A \otimes \mathcal K \right\}.	\end{equation*}
	If $\widetilde{a}\in  \widetilde{A}\otimes{\mathcal K}$ then for any $\widetilde{b} \in \widetilde{B} \otimes \mathcal K^+$ following condition holds
	\begin{equation*}
\left\langle \widetilde{b} ,\widetilde{a} \right\rangle_{\widetilde{B}\otimes  \mathcal K^+ } = \sum_{g \in G} g\left(\widetilde{b}^*\widetilde{a} \right) \in  \widetilde{A}\otimes{\mathcal K}.
	\end{equation*}
	Since $\left\langle \widetilde{b} ,\widetilde{a} \right\rangle_{\widetilde{B}\otimes  \mathcal K^+ }$ is $G$-invariant one has $\left\langle \widetilde{b} ,\widetilde{a} \right\rangle_{\widetilde{B}\otimes  \mathcal K^+ }\in A\otimes \mathcal K$, i.e. $\widetilde{a} \in \widetilde{A}^{\mathcal K}$. It follows that $\widetilde{A}\otimes{\mathcal K} \subset \widetilde{A}^{\mathcal K}$. Any $\widetilde{a}\in 	\widetilde{A}^{\mathcal K}$ satisfies to $\widetilde{a} \in 	\widetilde{A}^{\mathcal K}\backslash \widetilde{A}\otimes{\mathcal K}$ if and only if one or both of two following conditions hold:
	\begin{enumerate}
		\item [(a)] If $p_{\mathcal K}: \widetilde{B} \otimes \mathcal K^+ \to \widetilde{B} \otimes \mathcal K^+/  \widetilde{B} \otimes \mathcal K$ the natural projection onto the quotient algebra then $p_{\mathcal K}\left(\widetilde{a} \right) \neq 0$.
		\item[(b)] If $p_{\widetilde{B}}: \widetilde{B} \otimes \mathcal K^+ \to \widetilde{B} \otimes \mathcal K^+/  \widetilde{A} \otimes \mathcal K^+$ the natural projection onto the quotient algebra then $p_{ \widetilde{B}}\left(\widetilde{a} \right) \neq 0$
	\end{enumerate}
	If
	$$
	a = \left\langle \widetilde{a}, \widetilde{a} \right\rangle_{\widetilde{B} \otimes \mathcal K^+ }= \sum_{g \in G} g\left( \widetilde{a}^* \widetilde{a}\right) 
	$$
	then 
	\begin{equation}\label{stab_fin_prod_eqn}
	p_{\mathcal K}\left(a \right)=\sum_{g \in G}  g \left( p_{\mathcal K}\left(\widetilde{a} \right)p_{\mathcal K}\left(\widetilde{a} \right)^*\right)  =  p_{\mathcal K}\left(\widetilde{a} \right)p_{\mathcal K}\left(\widetilde{a} \right)^* + \sum_{g \in G \backslash\{e\}} g \left( p_{\mathcal K}\left(\widetilde{a} \right)p_{\mathcal K}\left(\widetilde{a} \right)^*\right)
	\end{equation}
	where $e \in G$ is the unity of $G$.
	If $\widetilde{a} \notin \widetilde{B} \otimes \mathcal K$ then taking into account that all terms of \eqref{stab_fin_prod_eqn} are positive  one has $p_{\mathcal K}\left(a \right) \ne 0$. Hence from $\widetilde{a} \notin  \widetilde{B} \otimes \mathcal K$ it turns out $\widetilde{a}\notin \widetilde{A}^{ \mathcal K}$. Similarly if $\widetilde{a} \notin \widetilde{A} \otimes \mathcal K^+$ then $p_B\left( {\widetilde{a}} \right) \neq 0$ and
	\begin{equation}\label{stab_fin_prod_b_eqn}
	p_B\left(a \right)=\sum_{g \in G}  g \left( p_{B}\left(\widetilde{a} \right)p_{B}\left(\widetilde{a} \right)^*\right)  =  p_{B}\left(\widetilde{a} \right)p_{B}\left(\widetilde{a} \right)^* + \sum_{g \in G \backslash\{e\}} g \left( p_{B}\left(\widetilde{a} \right)p_{B}\left(\widetilde{a} \right)^*\right).
	\end{equation}
	Taking into account that all terms of \eqref{stab_fin_prod_b_eqn} are positive  one has $p_{B}\left(a \right) \ne 0$. Hence from $\widetilde{a} \notin  \widetilde{A} \otimes \mathcal K^+$ it turns out $\widetilde{a}\notin \widetilde{A}^{ \mathcal K}$. In result one has
	\begin{equation*}
	\begin{split}
	\widetilde{A}^{\mathcal K}\backslash \widetilde{A}\otimes{\mathcal K} = \emptyset,\\
	\widetilde{A}^{\mathcal K}= \widetilde{A}\otimes{\mathcal K},
	\end{split}
	\end{equation*}
	i.e. the condition (c) of the Definition \ref{fin_def} holds.

\end{proof}

\subsection{Infinite coverings}
\paragraph*{}
Suppose that
\begin{equation}\label{stable_seq_pure_eqn}
\mathfrak{S} =\left\{ A =A_0 \xrightarrow{\pi_1} A_1 \xrightarrow{\pi_2} ... \xrightarrow{\pi_n} A_n \xrightarrow{\pi_{n+1}} ...\right\} \subset \mathfrak{FinAlg}
\end{equation}
is an (algebraical)  finite covering sequence. From the Theorem \ref{stable_fin_cov_thm} it follows that
\begin{equation}\label{stable_seq_eqn}
\mathfrak{S}_{\mathcal K} =\left\{ A \otimes \mathcal K =A_0 \otimes \mathcal K  \xrightarrow{\pi_1 \otimes \Id_{ \mathcal K}} A_1  \otimes \mathcal K\xrightarrow{\pi_2\otimes \Id_{ \mathcal K}} ... \xrightarrow{\pi_n \otimes \Id_{ \mathcal K}} A_n  \otimes \mathcal K \xrightarrow{\pi_{n+1}\otimes \Id_{ \mathcal K}} ...\right\} 
\end{equation}
is an (algebraical)  finite covering sequence, i.e. $\mathfrak{S}_{\mathcal K} \subset \mathfrak{FinAlg}$.
Denote by $\widehat{A} = \varinjlim A_n$, $G_n = G\left( A_n ~|~A\right)$, $\widehat{G} = \varprojlim G_n$. Clearly $\varinjlim \left( A_n \otimes \mathcal K\right)  = \widehat{A}  \otimes \mathcal K$. If $\pi: \widehat{A} \to B\left( \widehat{\H}\right)$ is an equivariant representation and $\mathcal K = \mathcal K\left(\H \right)$  then $\pi_{\mathcal K}= \pi \otimes \Id_{ \mathcal K} :\widehat{A}  \otimes \mathcal K\to B\left( \H_{\mathcal K} = \widehat{\H} \otimes \H \right) $ is an equivariant representation. Let us  consider a $\widehat{A}''$-Hilbert module $\ell^2\left( \widehat{A}''\right) $ then  $\widehat{A}''  \otimes \mathcal K \approx  \mathcal K\left(\ell^2\left( \widehat{A}''\right) \right)$ 
\begin{lemma}\label{stab_spec_in_lem}
	If $\overline{a}\in \widehat{A}''$ is a special element with respect to $\mathfrak{S}$, and $p \in \mathcal K$ is a rank-one projector then $\overline{a}\otimes p \in \left(\widehat{A} \otimes\mathcal K\right)''$ is a special element with respect to $\mathfrak{S}_{\mathcal K}$.
\end{lemma}
\begin{proof}
	One can select a basis of the Hilbert $\H$ space such that $p$ is represented by a following infinite matrix
	$$
	\begin{pmatrix}
	1& 0 &\ldots \\
	0& 0 &\ldots \\
	\vdots& \vdots &\ddots\\
	\end{pmatrix}.
	$$
	If $z \in A \otimes \mathcal K \subset \mathcal K\left(\ell^2\left(\widehat{A}\right)\right) $ is  represented by an infinite ,atrix
	$$
	z = \begin{pmatrix}
	a_{11}& a_{12} &\ldots \\
	a_{21}& a_{22} &\ldots \\
	\vdots& \vdots &\ddots\\
	\end{pmatrix}; \text{ where } a_{jk} \in A
	$$
	then
	$$
	z \left(\overline{a}\otimes p \right)z^* 
	$$
	is represented by the matrix
	$$
	z \left(\overline{a}\otimes p \right)z^* =  \begin{pmatrix}
	a_{11}& a_{12} &\ldots \\
	0& 0 &\ldots \\
	\vdots& \vdots &\ddots\\
	\end{pmatrix}\overline{a} \begin{pmatrix}
	a^*_{11}& 0 &\ldots \\
	a^*_{12}& 0 &\ldots \\
	\vdots& \vdots &\ddots\\
	\end{pmatrix}.
	$$
	Let us select an orthogonal basis $\left\{\xi_n \in \ell^2\left(\widehat{A}\right)  \right\}_{n \in \N}$ such that
	$$
	\xi_1=\begin{pmatrix}
	a^*_{11} \\
	a^*_{12} \\
	\vdots\\
	\end{pmatrix}.
	$$
	If $p_{\mathcal K} \in \mathcal K\left( \ell^2\left( A\right) \right)\subset \mathcal K\left(  \ell^2\left({\widehat{A}''}\right)\right)  $ a projector onto a submodule $A\xi_1\subset\ell^2\left(A \right) $ there is $\overline z \in A$ such that
	$$
	\begin{pmatrix}
	a^*_{11}& 0 &\ldots \\
	a^*_{12}& 0 &\ldots \\
	\vdots& \vdots &\ddots\\
	\end{pmatrix}= \overline z^* p_{\mathcal K}.
	$$
	It turns out
	\begin{equation}\label{stab_zpz}
	z\left( \overline{a}\otimes p\right)  z^*=  p_{\mathcal K} \overline z~ \overline{   a}~\overline z^* p_{\mathcal K}.
	\end{equation}
	
	Let $f_\eps$ is given by \eqref{f_eps_eqn}. The
	element $\overline{a}$ is special, it follows that the series
	\begin{equation*}
	\begin{split}
	a_n = \sum_{g \in \ker\left( \widehat{G} \to  G\left( A_n~|~A \right)\right)} g \  \overline{a} ,\\
	b_n = \sum_{g \in \ker\left( \widehat{G} \to  G\left( A_n~|~A \right)\right)} g \left(\overline z ~ \overline{a} ~ \overline z^*\right) ,\\
	c_n = \sum_{g \in \ker\left( \widehat{G} \to  G\left( A_n~|~A \right)\right)} g \left(\overline z ~ \overline{a} ~ \overline z^*\right)^2,\\
	d_n = \sum_{g \in \ker\left( \widehat{G} \to  G\left( A_n~|~A \right)\right)} g f_\eps\left( \overline z ~ \overline{a} ~ \overline z^* \right) 
	\end{split}
	\end{equation*}
	are strongly convergent and the sums lie in   $A_n$, i.e. $b_n,~ c_n,~ d_n \in A_n $; 
	For any $\eps > 0$ there is $N \in \N$ (which depends on $\overline{a}$ and $z$) such that for any $n \ge N$ a following condition holds
	\begin{equation*}
	\begin{split}
	\left\| b_n^2 - c_n\right\| < \eps.
	\end{split}
	\end{equation*}	
	It turns out 
	\begin{equation*}
	\begin{split}
	a^{\mathcal K}_n = \sum_{g \in \ker\left( \widehat{G} \to  G\left( A_n~|~A \right)\right)} g  \left(\overline{a}\otimes p \right)= a_n \otimes p \in A_n \otimes \mathcal K,\\
	b^{\mathcal K}_n = \sum_{g \in \ker\left( \widehat{G} \to  G\left( A_n~|~A \right)\right)} g \left( z \left(\overline{a}\otimes p \right)z^*\right) =  p_\mathcal{K} b_n p_\mathcal{K} \in A_n \otimes \mathcal K ,\\
	c^{\mathcal K}_n = \sum_{g \in \ker\left( \widehat{G} \to  G\left( A_n~|~A \right)\right)} g \left( z \left(\overline{a}\otimes p \right)z^*\right)^2 =  p_\mathcal{K} c_n p_\mathcal{K}\in A_n \otimes \mathcal K,\\
	d^{\mathcal K}_n = \sum_{g \in \ker\left( \widehat{G} \to  G\left( A_n~|~A \right)\right)} g f_\eps\left(z \left(\overline{a}\otimes p \right)z^* \right) =p_\mathcal{K} d_n p_\mathcal{K}\in A_n \otimes \mathcal K,
	\end{split}
	\end{equation*}
	i.e. conditions (a), (b) of the Definition \ref{special_el_defn} hold. From 
	\begin{equation*}
	\begin{split}
	\left\| b_n^2 - c_n\right\|=\left\|  p_\mathcal{K}b_n^2 p_\mathcal{K} -  p_\mathcal{K}c_n p_\mathcal{K}\right\|=\left\| \left( b^{\mathcal K}_n\right)^2 - c^{\mathcal K}_n\right\|
	\end{split}
	\end{equation*}
	it follows that for any $\eps > 0$ there is $N \in \N$ such that
	$$
	\left\| \left( b^{\mathcal K}_n\right)^2 - c^{\mathcal K}_n\right\| < \eps
	$$	
	for any $n \ge N$. It means that the conditions (c) of the Definition \ref{special_el_defn} holds.
\end{proof}
\begin{corollary}\label{stab_spec_in_cor}
	Let $\overline{A}_\pi \subset B\left( \widehat{\H}\right)$ be the  disconnected inverse noncommutative limit of $\downarrow\mathfrak{S}$ with respect to $\pi$.
	If $\overline{A}^\mathcal{K}_{\pi \otimes \Id_{ \mathcal K}}$ is be the  disconnected inverse noncommutative limit of $\downarrow\mathfrak{S}_\mathcal{K}$ with respect to $\pi \otimes \Id_{ \mathcal K}$ then
	$$
	\overline{A}_\pi \otimes \mathcal K \subset \overline{A}^\mathcal{K}_{\pi \otimes \Id_{ \mathcal K}}.
	$$
\end{corollary}
\begin{proof}
	The linear span of given by the Lemma \ref{stab_spec_in_lem} special elements $\overline{a}\otimes p \in \left(\widehat{A} \otimes\mathcal K\right)''$ is dense in $\overline{A}_\pi \otimes \mathcal K$, so $
	\overline{A}_\pi \otimes \mathcal K \subset \overline{A}^\mathcal{K}_{\pi \otimes \Id_{ \mathcal K}}$.
\end{proof}

\begin{lemma}\label{stab_spec_out_lem}
	Let $\mathfrak{S}$ be an  algebraical  finite covering sequence given by \eqref{stable_seq_pure_eqn}. Let  $\pi:\widehat{A} \to B\left( \widehat{\H}\right) $ be an equivariant representation.	 Let $\overline{A}_\pi \subset B\left( \widehat{\H}\right)$ be the  disconnected inverse noncommutative limit of $\downarrow\mathfrak{S}$ with respect to $\pi$. Let $\mathfrak{S}_\mathcal{K}$ be an  algebraical  finite covering sequence given by \eqref{stable_seq_eqn}.
	If  $\overline a^\mathcal{K} \in B\left( \widehat{\H} \otimes \H\right)_+$ is a is special element (with respect to $\mathfrak S_\mathcal{K}$) then
	$$
	\overline a^\mathcal{K} \in \overline{A}_\pi  \otimes \mathcal K.
	$$
\end{lemma}
\begin{proof}
	From the  Corollary \ref{special_cor} it turns out that $\overline a^\mathcal{K} \in \left(\widehat{A} \otimes \mathcal K \right)''$, i.e. $\overline a^\mathcal{K}\in B\left(\ell^2\left({\widehat{A}''} \right)  \right)$. If  $\overline a^\mathcal{K}\notin \overline{A}_\pi  \otimes \mathcal K$ then there are $\xi, \eta \in \ell^2\left(\overline{A}_\pi \right) $ such that 
	\begin{equation}\label{stab_prod_eqn}
	\left\langle \xi, \overline a^\mathcal{K} \eta\right\rangle_{\ell^2\left(H_{\widehat{A}''} \right) } \notin \overline{A}_\pi.
	\end{equation}
The element $\overline a^\mathcal{K}$ is positive, so it is self-adjoint, hence \eqref{stab_prod_eqn} is equivalent to the existence of $\xi\in \ell^2\left(\overline{A}_\pi \right) $ such that
	\begin{equation}\label{stab_prod_x_eqn}
	\left\langle \xi, \overline a^\mathcal{K} \xi\right\rangle_{\ell^2\left(H_{\widehat{A}''} \right) } \notin \overline{A}_\pi.
	\end{equation}
	If we select a basis $\left\{\xi_n \in \ell^2\left(\widehat{A}'' \right) \right\}_{n \in \N}$ of $\ell^2\left(\widehat{A}'' \right)$  of such that $\xi_1 = \xi$ then $\overline{a}^{\mathcal{K}}$ is represented by an infinite matrix
	\begin{equation*}
	\overline{a}^{\mathcal{K}} = \begin{pmatrix}
	\overline{a}_{11}& \overline{a}_{12} &\ldots \\
	\overline{a}_{21}& \overline{a}_{22} &\ldots \\
	\vdots& \vdots &\ddots\\
	\end{pmatrix}
	\end{equation*}
	and from \eqref{stab_prod_x_eqn} it turns out $a_{11} \notin \overline{A}_\pi$.
	Let us select $z \in A$  such that $a_{11}$ does not satisfy to the Definition \ref{special_el_defn}, i.e. at last one of the following (a)-(c) does not hold:
	\begin{enumerate}
		\item[(a)] For any $n \in \mathbb{N}^0$  the following  series 
		\begin{equation*}
		\begin{split}
		a_n = \sum_{g \in \ker\left( \widehat{G} \to  G\left( A_n~|~A \right)\right)} g  \overline{a}_{11} 
		\end{split}
		\end{equation*}
		is strongly convergent and the sum lies in   $A_n$, i.e. $a_n \in A_n $;		
		\item[(b)]
		If $f_\eps: \R \to \R$ is given by 
		\eqref{f_eps_eqn}
		then any $n \in \mathbb{N}^0$ and for any $z \in A$   following  series 
		\begin{equation*}
		\begin{split}
		b_n = \sum_{g \in \ker\left( \widehat{G} \to  G\left( A_n~|~A \right)\right)} g \left(z  \overline{a}_{11} z^*\right) ,\\
		c_n = \sum_{g \in \ker\left( \widehat{G} \to  G\left( A_n~|~A \right)\right)} g \left(z  \overline{a}_{11} z^*\right)^2,\\
		d_n = \sum_{g \in \ker\left( \widehat{G} \to  G\left( A_n~|~A \right)\right)} g f_\eps\left( z  \overline{a}_{11} z^* \right) 
		\end{split}
		\end{equation*}
		are strongly convergent and the sums lie in   $A_n$, i.e. $b_n,~ c_n,~ d_n \in A_n $; 
		\item[(c)] For any $\eps > 0$ there is $N \in \N$ (which depends on $\overline{a}$ and $z$) such that for any $n \ge N$ a following condition holds
		\begin{equation*}
		\begin{split}
		\left\| b_n^2 - c_n\right\| < \eps.
		\end{split}
		\end{equation*}	
	\end{enumerate}
	If $z_\mathcal{K}$ is given by
	$$
	z_\mathcal{K}  = \begin{pmatrix}
	z& 0 &\ldots \\
	0& 0 &\ldots \\
	\vdots& \vdots &\ddots\\
	\end{pmatrix}
	$$
	then
	\begin{equation*}
	\begin{split} 
	a^\mathcal{K}_n = \sum_{g \in \ker\left( \widehat{G} \to  G\left( A_n~|~A \right)\right)} g  \overline{a}= \begin{pmatrix}
	a_n& \sum_{g \in \ker\left( \widehat{G} \to  G\left( A_n~|~A \right)\right)} g  \overline{a}_{12} &\ldots \\
	\sum_{g \in \ker\left( \widehat{G} \to  G\left( A_n~|~A \right)\right)} g  \overline{a}_{21}& \sum_{g \in \ker\left( \widehat{G} \to  G\left( A_n~|~A \right)\right)} g  \overline{a}_{22} &\ldots \\
	\vdots& \vdots &\ddots\\
	\end{pmatrix},\\
	b^\mathcal{K}_n = \sum_{g \in \ker\left( \widehat{G} \to  G\left( A_n~|~A \right)\right)} g \left(  z_\mathcal{K}\overline{a}z^*_\mathcal{K}\right) = \begin{pmatrix}
	b_n& 0 &\ldots \\
	0& 0 &\ldots \\
	\vdots& \vdots &\ddots,\\
	\end{pmatrix},\\
	c^\mathcal{K}_n = \sum_{g \in \ker\left( \widehat{G} \to  G\left( A_n~|~A \right)\right)} g \left(  z_\mathcal{K}\overline{a}z^*_\mathcal{K}\right)^2 = \begin{pmatrix}\\
	c_n& 0 &\ldots \\
	0& 0 &\ldots \\
	\vdots& \vdots &\ddots,
	\end{pmatrix},\\
	d^\mathcal{K}_n = \sum_{g \in \ker\left( \widehat{G} \to  G\left( A_n~|~A \right)\right)} g f_\eps \left(  z_\mathcal{K}\overline{a}z^*_\mathcal{K}\right) = \begin{pmatrix}\\
	d_n& 0 &\ldots \\
	0& 0 &\ldots \\
	\vdots& \vdots &\ddots\\
	\end{pmatrix}.
	\end{split}
	\end{equation*}
	
From above equation it turns out
	\begin{equation*}
	\begin{split} 
	a_n \notin A_n \Rightarrow a^\mathcal{K}_n \notin A_n \otimes \mathcal K,\\
	b_n \notin A_n \Rightarrow b^\mathcal{K}_n \notin A_n \otimes \mathcal K,\\
	c_n \notin A_n \Rightarrow c^\mathcal{K}_n \notin A_n \otimes \mathcal K,\\
	d_n \notin A_n \Rightarrow d^\mathcal{K}_n \notin A_n \otimes \mathcal K.
	\end{split}
	\end{equation*}
	Following condition holds
	$$
	\left\| \left( b^{\mathcal K}_n\right)^2 - c^{\mathcal K}_n\right\| = \left\| b^2_n - c_n\right\|.
	$$
	Hence if there is $\eps > 0$ such that for any $N \in \N$ there is $n \ge N$ which satisfy to the following condition
	
	$$
	\left\| b^2_n - c_n\right\|\ge \eps
	$$
	then for any $N \in \N$ there is $n \ge N$ such that
	$$
	\left\| \left( b^{\mathcal K}_n\right)^2 - c^{\mathcal K}_n\right\| \ge \eps.
	$$
	We have a contradiction. From this contradiction it turns out
	$$
	\overline a^\mathcal{K} \in \overline{A}_\pi  \otimes \mathcal K.
	$$
\end{proof}

\begin{corollary}\label{stab_spec_out_cor}
	Let $\overline{A}_\pi \subset B\left( \widehat{\H}\right)$ be the  disconnected inverse noncommutative limit of $\downarrow\mathfrak{S}$ with respect to $\pi$.
	If $\overline{A}^\mathcal{K}_{\pi \otimes \Id_{ \mathcal K}}$ is  the  disconnected inverse noncommutative limit of $\downarrow\mathfrak{S}_\mathcal{K}$ with respect to $\pi \otimes \Id_{ \mathcal K}$ then
	$$
	\overline{A}^\mathcal{K}_{\pi \otimes \Id_{ \mathcal K}}\subset	\overline{A}_\pi \otimes \mathcal K.
	$$
\end{corollary}

\
\begin{theorem}\label{stab_inf_thm}
	Let $\mathfrak{S}$ be an  algebraical  finite covering sequence given by \eqref{stable_seq_pure_eqn}. Let  $\pi:\widehat{A} \to B\left( \widehat{\H}\right) $ be good representation.	 Let $\overline{A}_\pi \subset B\left( \widehat{\H}\right)$ be the  disconnected inverse noncommutative limit of $\downarrow\mathfrak{S}$ with respect to $\pi$. If $\mathfrak{S}_\mathcal{K}$ be an  algebraical  finite covering sequence given by \eqref{stable_seq_pure_eqn} then $\pi \otimes \Id_{ \mathcal K}: \varprojlim A_n \otimes \mathcal{K} \to B\left(\widetilde{   \H}\otimes \H \right) $ is a good representation. Moreover if
	\begin{equation*}
	\begin{split}
	\varprojlim_\pi \downarrow \mathfrak{S}=\widetilde{A}_\pi,\\
	G\left(\widetilde{A}_\pi~|~ A\right)=G_\pi.
	\end{split}
	\end{equation*}	
	then
	\begin{equation*}
	\begin{split}
	\varprojlim_{\pi \otimes \Id_{ \mathcal K}} \downarrow \mathfrak{S}_\mathcal{K}\cong\widetilde{A}_\pi \otimes \mathcal K,\\
	G\left(\varprojlim_{\pi \otimes \Id_{ \mathcal K}} \downarrow \mathfrak{S}_\mathcal{K}~|~ A\otimes \mathcal K\right)\cong	G\left(\widetilde{A}_\pi \otimes \mathcal K~|~ A\otimes \mathcal K\right)\cong G_\pi.
	\end{split}
	\end{equation*}	
	It means that if $\left(A, \widetilde{A}_\pi, G_\pi\right)$ is an  infinite noncommutative covering of $\mathfrak{S}$  (with respect to $\pi$)	 then $\left(A \otimes \mathcal K, \widetilde{A}_\pi \otimes \mathcal K, G_\pi\right)$ is an  infinite noncommutative covering of $\mathfrak{S}$  (with respect to $\pi \otimes \Id_{ \mathcal K}$).
\end{theorem}
\begin{proof}
	From the Corollaries \ref{stab_spec_in_cor} and \ref{stab_spec_out_cor} it follows that if $\overline{A}^\mathcal{K}_{\pi \otimes \Id_{ \mathcal K}}$ is  the  disconnected inverse noncommutative limit of $\downarrow\mathfrak{S}_\mathcal{K}$ with respect to $\pi \otimes \Id_{ \mathcal K}$ then
	$$
	\overline{A}^\mathcal{K}_{\pi \otimes \Id_{ \mathcal K}}=	\overline{A}_\pi \otimes \mathcal K.
	$$
	If $J\subset \varprojlim G_n$ is a set of representatives of $\varprojlim G_n/G_\pi$ then 	$\overline{A}_\pi$ is the $C^*$-norm completion of the following algebraic direct sum
	$$
	\bigoplus_{g \in J}g \widetilde{A}_\pi.
	$$
	Hence $\overline{A}^\mathcal{K}_{\pi \otimes \Id_{ \mathcal K}}$ is the $C^*$-norm completion of the following algebraic direct sum
	$$
	\bigoplus_{g \in J}g \left( \widetilde{A}_\pi \otimes \mathcal K\right) .
	$$
	So $\widetilde{A}_\pi \otimes \mathcal K$ is a maximal irreducible component of $\overline{A}^\mathcal{K}_{\pi \otimes \Id_{ \mathcal K}}$ and $G_\pi$ is the maximal subgroup of  among subgroups $G \subset \varprojlim G_n$ such that
	$$
	G \left(  \widetilde{A}_\pi \otimes \mathcal K\right) =  \widetilde{A}_\pi \otimes \mathcal K.
	$$
	Clearly (a)-(c) of the Definition \ref{good_seq_defn} hold.
	
\end{proof}

\section{Foliations and coverings}
\subsection{Lifts and restrictions of foliations}

Let $M$ be a smooth manifold and let
is an  $\mathcal F\subset TM$ be an integrable subbundle.
If $\pi:\widetilde M \to M$ is a covering and $\widetilde{ \mathcal F} \subset T\widetilde M$ is the lift of $\mathcal F$ given by a following diagram
\newline
\newline
\begin{tikzpicture}
\matrix (m) [matrix of math nodes,row sep=3em,column sep=4em,minimum width=2em]
{
	\widetilde{\mathcal F} &    T\widetilde{M} \\
	\mathcal F	& TM    \\};
\path[-stealth]
(m-1-1) edge node [above] {$\hookto$} (m-1-2)
(m-1-1) edge node [right] {} (m-2-1)
(m-1-2) edge node [right] {} (m-2-2)
(m-2-1) edge node [above] {$\hookto$}  (m-2-2);

\end{tikzpicture}
\newline
then $\widetilde{\mathcal F}$ is integrable.

\begin{definition}\label{fol_cov_defn}
	In the above situation we say that a foliation $\left(\widetilde{M},~ \widetilde{\mathcal F} \right)$ is the \textit{induced by} $\pi$ \textit{covering} of $\left(M,\mathcal F \right)$ or the $\pi$-\textit{lift} of $\left(M,\mathcal F \right)$. 
\end{definition}
\begin{remark}
	The $\pi$-lift of a foliation is described in  \cite{ouchi:cov_fol}.
\end{remark}

\begin{definition}
If $\left(M,\mathcal F \right)$ is a foliation and $\mathcal{U} \subset M$ be an open subset then there is a \textit{restriction} $\left(\mathcal{U},\mathcal F|_{\mathcal{U}} \right)$ of $\left(M,\mathcal F \right)$.
\end{definition}
\begin{remark}
	It is proven in \cite{connes:foli_survey} that any restriction of foliation induces an injective *-homomorphism 
	\begin{equation}\label{fol_res_hom_eqn}
C^*_r\left(\mathcal{U},\mathcal F|_{\mathcal{U}} \right)\hookto C^*_r\left(M,\mathcal F \right).
	\end{equation}
\end{remark}
\subsection{Extended algebra of foliation}
\paragraph*{} 
Let $\left(M, \mathcal F \right)$ be a foliation $x \in M$, and let us consider a representation, $\rho_x$ given by \eqref{fol_repr}, i.e.  

\begin{equation*}
(\rho_x (f) \, \xi) \, (\gamma) = \int_{\gamma_1 \circ \gamma_2 =
	\gamma} f(\gamma_1) \, \xi (\gamma_2) \qquad \forall \, \xi \in L^2
(\mathcal G_x).\end{equation*}
Denote by $\H = \bigoplus_{x \in M} L^2\left(\mathcal G_x \right)$. Representations $\rho_x$ yield an inclusion  
\begin{equation}\label{fol_inc_eqn}
C^* _r(M, \mathcal{F}) \to  B\left(\H\right)
\end{equation}

	If $C_b\left(M \right)$ is a $C^*$-algebra of bounded continuous functions then for any $x \in M$ there is the natural representation $C_b\left(M \right) \to B\left(L^2\left(\mathcal G_x \right) \right)$ 
	given by
	\begin{equation*}
	a \mapsto \left(\xi \mapsto a\left( x\right)\xi \right) 
\text{ where } ~a \in 	C_b\left( M\right), ~  \xi \in L^2\left(\mathcal  G_x\right).
	\end{equation*}
From this fact it turns out that there is the natural faithful representation
\begin{equation}\label{fol_ince_eqn}
C_b\left(M \right) \to  B\left(\H\right).
\end{equation}

	Let us consider both $C^* _r(M, \mathcal{F})$ and $C_b\left( M\right)$ as subalgebras of $B\left( \H\right) =B\left(\bigoplus_{x \in M} L^2\left(\mathcal  G_x\right) \right)$.	

\begin{definition}
	The $C^*$-subalgebra of $B\left(\H \right)$ generated by images of $C^*_r\left( M, \mathcal{F}\right)$ and  $C_b\left( M\right)$ given by \eqref{fol_inc_eqn}, \eqref{fol_ince_eqn}  is said to be the \textit{extended algebra of the foliation} $\left( M, \mathcal F\right)$. The extended algebra will be denoted by $E^*\left(M,\mathcal F \right)$. 
\end{definition}
\begin{lemma}\label{fol_ess_lem}
	An algebra of a foliation $C^*_r\left( M, \mathcal{F}\right)$ is an essential ideal of an extended algebra  $E^*\left(M,\mathcal F \right)$ of the foliation.
\end{lemma}
\begin{proof}
	If $a \in E^*\left(M,\mathcal F \right) \backslash C^*_r\left( M, \mathcal{F}\right)$ then there are $x \in M$, $z \in \C$ and $k \in \mathcal{K}\left(L^2\left( \mathcal G_x\right) \right)$ such that
	\begin{equation}\label{fol_ideal_eqn}
	\begin{split}
	z \neq 0,\\
	\rho_x\left(a \right) \xi = \left( z+k\right) \xi, ~\forall \xi \in L^2\left( \mathcal G_x\right). 
	\end{split}
	\end{equation}
	From \eqref{fol_ideal_eqn} it follows that $C^*_r\left( M, \mathcal{F}\right)$ is an ideal.
	If $a \in C_b\left(M \right)$ is not trivial then there is $x \in M$ such that $a\left( x\right)\neq 0$. On the other hand there is $b \in C^*_r\left( M, \mathcal{F}\right)$ such that $\rho_x\left(b \right)\neq 0$. However $\rho_x\left(ab \right)= a\left( x\right)\rho_x\left(b \right)\neq 0$. It turns out that  $C^*_r\left( M, \mathcal{F}\right)$ is an essential ideal.  
\end{proof}
\subsection{Finite-fold coverings}
\begin{empt}
	If $\gamma: \left[0,1\right]\to M$ is a path which corresponds to an element of the holonomy groupoid then we denote by $\left[\gamma\right]$ its equivalence class, i.e. element of groupoid.
	There is the space of half densities $\Omega_{\widetilde{M}}^{1/2}$ on $\widetilde{M}$ which is a lift  the space of half densities $\Omega_{M}^{1/2}$ on $M$. If $L$ is a leaf of $\left(M,\mathcal F \right)$, $L'=\pi^{-1}\left( L\right)$   then a space $\widetilde{L}$ of holonomy covering of  $L$ coincides with the space of the holonomy covering of $L'$. It turns out that $L^2\left( \widetilde{\mathcal G}_{\widetilde{x}}\right)\approx L^2\left(\mathcal G_{ \pi\left(\widetilde{x}\right)} \right)$ for any $\widetilde{x} \in \widetilde{M}$.
	If $\mathcal G$ (resp. $\widetilde{\mathcal G}$) is a holonomy groupoid of $\left(M,\mathcal F \right)$ (resp. $\left(\widetilde{M},~ \widetilde{\mathcal F} \right)$) then there is a map $\pi_{\mathcal G}:\widetilde{\mathcal G} \to \mathcal G$ given by
	\begin{equation*}
	\begin{split}
	\left[\widetilde{\gamma}\right] \mapsto \left[\pi \circ\widetilde{\gamma}\right]
	\end{split}
	\end{equation*}
	The map $\pi_{\mathcal G}:\widetilde{\mathcal G} \to \mathcal G$ induces 
	a natural involutive homomorphism
	\begin{equation*}
	\begin{split}
	\Coo_c\left(\mathcal G,  \Omega_{M}^{1/2}\right) \hookto  \Coo_c\left(\widetilde{\mathcal G},  \Omega_{\widetilde{M}}^{1/2}\right)
	\end{split}
	\end{equation*}	
Completions of $\Coo_c\left(\mathcal G,  \Omega_{M}^{1/2}\right)$ and $\Coo_c\left(\widetilde{\mathcal G},  \Omega_{\widetilde{M}}^{1/2}\right)$ with respect to given by \eqref{fol_norm_eqn} norms gives an injective *- homomorphism $C^*_r\left( M,\mathcal F\right) \hookto C^*_r\left(\widetilde{M},~ \widetilde{\mathcal F} \right)$ of $C^*$-algebras. The action of the group $G\left(\widetilde{M}~|~M \right)$ of covering transformations on $\widetilde{M}$ naturally induces an action of $G\left(\widetilde{M}~|~M \right)$ on  $C^*_r\left(\widetilde{M},~ \widetilde{\mathcal F} \right)$ such that $C^*_r\left( M, \mathcal{F}\right)   = C^*_r\left(\widetilde{M},~ \widetilde{\mathcal F} \right)^{G\left(\widetilde{M}~|~M \right)}$.
\end{empt}
\begin{empt}\label{fol_frame_constr}
	If $\pi:\widetilde{\mathcal X} \to \mathcal X$ is a finite-fold covering of  compact Hausdorff spaces then there is a finite set $\left\{\mathcal U_\iota\subset  \mathcal X\right\}_{\iota \in I}$ of connected open subsets of $ \mathcal X$ evenly covered by $\pi$ such that $ \mathcal X = \bigcup_{\iota \in I} \mathcal U_\iota$. There is a partition of unity, i.e.
	$$
	1_{C\left(\mathcal X \right) }= \sum_{\iota \in I}a_\iota
	$$
	where $a_\iota \in C\left(\mathcal X \right)_+$ is such that $\supp a_\iota \subset \mathcal U_\iota$. Denote by $e_\iota = \sqrt{a_\iota}\in C\left(\mathcal X \right)_+$
	For any $\iota \in I$ we select $\widetilde{\mathcal U}_\iota \subset \widetilde{\mathcal X}$ such that $\widetilde{\mathcal U}_\iota$ is homemorphically mapped onto $\mathcal U_\iota$. If $\widetilde{e}_\iota \in C\left( \widetilde{\mathcal X}\right) $ is given by
	\begin{equation*}
	\widetilde{e}_\iota\left(\widetilde{x} \right)  = \left\{
	\begin{array}{c l}
	e_\iota\left(\pi\left( \widetilde{x}\right) \right)  & \widetilde{x} \in \widetilde{\mathcal U}_\iota \\
	0 & \widetilde{x} \notin \widetilde{\mathcal U}_\iota
	\end{array}\right.
	\end{equation*}
	and $G = G\left( \widetilde{\mathcal X}~|{\mathcal X} \right)$ then
	\begin{equation*}
	\begin{split}
	1_{C\left(\widetilde{\mathcal X} \right) }= \sum_{ \left(g, \iota\right)\in G \times I}g\widetilde{e}^2_\iota,\\
	\widetilde{e}_\iota \left( g\widetilde{e}_\iota\right) =0; \text{ for any nontrivial } g \in G.
	\end{split}
	\end{equation*}
	If $\widetilde{I}= G\times I$ and $\widetilde{e}_{\left(g, \iota\right)}= g\widetilde{e}_\iota$ the from the above equation it turns out
	\begin{equation}\label{ctr_unity_eqn}
	\begin{split}
	1_{C\left(\widetilde{\mathcal X} \right) }= \sum_{\iota \in \widetilde{I}}\widetilde{e}_\iota \left\rangle \right\langle \widetilde{e}_\iota.
	\end{split}
	\end{equation}
\end{empt}
\begin{theorem}\label{fol_fin_thm}
	Let $\left(M,\mathcal  F \right)$ be a foliation, and let $\pi:\widetilde{M}\to M$ is a finite-fold covering. Let $\left(\widetilde{M},~ \widetilde{\mathcal F} \right)$ be a covering of $\left(M,\mathcal F \right)$ induced by $\pi$. If $\widetilde{M}$ is compact then the triple $$\left(C^*_r\left( M, \mathcal{F}\right),  C^*_r\left(\widetilde{M},~ \widetilde{\mathcal F} \right), G = G\left( \widetilde{M}~|~M\right) \right)$$ is a noncommutative finite-fold covering.
\end{theorem}
\begin{proof}
	Let $E^*\left(M,\mathcal F \right)$ (resp. $E^*\left(\widetilde{M},~ \widetilde{\mathcal F} \right)$) be an extended algebra of the foliation $\left( M,\mathcal F\right)$ (resp. $\left(\widetilde{M},~ \widetilde{\mathcal F} \right)$). Denote by $A \stackrel{\text{def}}{=} C^*_r\left( M, \mathcal{F}\right)$,  $B \stackrel{\text{def}}{=} E^*\left(M,\mathcal F \right)$, $\widetilde{B} \stackrel{\text{def}}{=} E^*\left(\widetilde{M},~ \widetilde{\mathcal F} \right)$. Both $B$ and $\widetilde{B}$ are unital. From the Lemma \ref{fol_ess_lem} it turns out that $A$, (resp. $C^*_r\left(\widetilde{M},~ \widetilde{\mathcal F} \right)$) is an essential ideal of  $B$, (resp. $\widetilde{B}$), i.e. these algebras satisfy to the condition (a) of the Definition \ref{fin_def}.  From the Theorem \ref{pavlov_troisky_thm} it turns out that $C\left( \widetilde{M}\right) = C_b\left( \widetilde{M}\right)$ is a finitely generated $C\left(M \right) = C_b\left(M \right)$ module. Moreover from \eqref{ctr_unity_eqn} it turns out that there is a finite set $\left\{e_\iota\right\}_{\iota \in I}$ such that
	\begin{equation*}
	\begin{split}
	1_{C\left(\widetilde{M} \right) }= \sum_{\iota \in I}\widetilde{e}_\iota \left\rangle \right\langle \widetilde{e}_\iota.
	\end{split}
	\end{equation*}
	Clearly $1_{C\left(\widetilde{M} \right) } = 1_{E^*\left(\widetilde{M},~ \widetilde{\mathcal F} \right)}= 1_{\widetilde{B}}$. It turns out that any $\widetilde{b} \in \widetilde{B}$ is given by 
	\begin{equation*}
	\begin{split}
	\widetilde{b} = \sum_{\iota \in I}\widetilde{e}_\iota b_\iota, \\ 
	b_\iota = \left\langle \widetilde{b}, \widetilde{e}_\iota \right\rangle_{\widetilde{B}} \in B,
	\end{split}
	\end{equation*} 
	i.e. $\widetilde{B}$ is a finitely generated right $B$ module. From the Kasparov Stabilization Theorem \cite{blackadar:ko} it turns out that $\widetilde{B}$ is a projective $B$ module. From  $\widetilde{B}^G = B$ it follows that the triple $\left(B, \widetilde{B}, G\right)$ is  an unital noncommutative finite-fold  covering, i.e. the condition (b) of the Definition \ref{fin_def} is satisfied. Let  $\widetilde{A}$ be a $C^*$-algebra which satisfies to  the condition (c) of the Definition \ref{fin_def}, i.e.
	$$
	\widetilde{A} = \left\{a\in 	\widetilde{B}~|~ \left\langle 	\widetilde{B},a\right\rangle_{	\widetilde{B}} \subset A \right\}.
	$$
	If $\widetilde{a} \in C^*_r\left( \widetilde{M},~ \widetilde{\mathcal F}\right)$ then from the Lemma \ref{fol_ess_lem} it follows that $\widetilde{b}\widetilde{a} \in C^*_r\left( \widetilde{M},~ \widetilde{\mathcal F}\right)$ for any $\widetilde{b} \in \widetilde{B} = E^*\left( \widetilde{M},~ \widetilde{\mathcal F}\right)$, hence
	$$
	\left\langle \widetilde{b},\widetilde{a}\right\rangle_{	\widetilde{B}} = \sum_{g \in G}g\left(\widetilde{b}\widetilde{a} \right)\in  C^*_r\left( \widetilde{M},~ \widetilde{\mathcal F}\right). 
	$$
	Otherwise $\left\langle \widetilde{b},\widetilde{a}\right\rangle_{	\widetilde{B}}$ is $G$-invariant, so $\left\langle \widetilde{b},\widetilde{a}\right\rangle_{	\widetilde{B}} \in A$. It turns out $\widetilde{a} \in \widetilde{A}$, hence $C^*_r\left( \widetilde{M},~ \widetilde{\mathcal F}\right) \subset \widetilde{A}$. If $\widetilde{a} \in \widetilde{A} \backslash C^*_r\left( \widetilde{M},~ \widetilde{\mathcal F}\right)$ then there is $\widetilde{x} \in \widetilde{M}$, $z \in \C$ and $k \in \mathcal{K}\left(L^2\left( \widetilde{\mathcal G}_{\widetilde{x}}\right) \right)$ such that
	\begin{equation*}
	\begin{split}
	z \neq 0,\\
	\rho_{\widetilde{x}}\left(\widetilde{a} \right) =  z1_{B\left( L^2\left( \widetilde{\mathcal G}_{\widetilde{x}}\right)\right)  }+k.
	\end{split}
	\end{equation*}
	It follows that
	\begin{equation*}
	\begin{split}
	\rho_{\pi\left( \widetilde{x}\right) }\left(\left\langle \widetilde{a},\widetilde{a}\right\rangle_{	\widetilde{B}} \right) = \left|z\right|^21_{B\left( L^2\left( \widetilde{\mathcal G}_{\widetilde{x}}\right) \right) } + k'
	\end{split}
	\end{equation*}
	where $k' \in \mathcal{K}\left(L^2\left( \widetilde{\mathcal G}_{\widetilde{x}}\right) \right)$. From the above equation it turns out that  $\left\langle \widetilde{a},\widetilde{a}\right\rangle_{	\widetilde{B}} \notin A$, hence $\widetilde{a} \notin \widetilde{A}$. From this contradiction it turns out
	\begin{equation*}
	\begin{split}
\widetilde{A} \backslash C^*_r\left( \widetilde{M},~ \widetilde{\mathcal F}\right) = \emptyset,\\
\widetilde{A} \subset C^*_r\left( \widetilde{M},~ \widetilde{\mathcal F}\right).
	\end{split}
	\end{equation*}
	In result one has and $\widetilde{A} = C^*_r\left( \widetilde{M},~ \widetilde{\mathcal F}\right)$.  
	
\end{proof}

\begin{example}\label{fol_tor_cov_exm} Let us consider a foliation $\left(\T^2, \mathcal F_\th \right)$ given by the Example \ref{fol_tor_exm}. From the  Example   \ref{ex:Atheta} and the Theorem \ref{MoritaK_thm} it turns out
	$$
	C_r^*\left(\T^2, \mathcal{F}_\th \right) = C\left( \T^2_\th\right) \otimes \mathcal K.
	$$
	
	There is a homeomorphism $\T^2 = S^1 \times S^1$.  Let $m \in \N$ be such that $m > 1$, and let $\widetilde{\T}^2 \approx  \T^2$ and let us consider an $m^2$-fold covering $\pi: \widetilde{\T}^2 \to \T$ given by
	$$
	\widetilde{\T}^2 \approx S^1 \times S^1 \xrightarrow{\left( \times m, \times m\right) }  S^1 \times S^1 \approx \T^2
	$$
	where $\times m$ is an $m$-listed covering of the circle $S^1$. If $\left(\widetilde{M},~ \widetilde{\mathcal F} \right)$ is the $\pi$-lift of $\left(\T^2, \mathcal F_\th \right)$ then 
	$$
	\left(\widetilde{M},~ \widetilde{\mathcal F} \right) = \left(\widetilde{\T}^2, \mathcal{F}_{\th/{m^2}} \right).
	$$
	From the Theorem \ref{fol_fin_thm} it follows that the triple 
	$$\left(C^*_r\left(\T^2, \mathcal{F}_\th \right),~ C^*_r\left(\widetilde{\T}^2, \mathcal{F}_{\th/m^2} \right), ~G\left(\widetilde{\T}^2~|~\T^2 \right)= \Z_{m^2}  \right)$$
	is a noncommutative finite-fold covering.
	From the Theorem \ref{stable_fin_cov_thm} it turns out that the triple 
	$$\left(C^*_r\left(\T^2, \mathcal{F}_\th \right) \otimes \mathcal K,~ C^*_r\left(\widetilde{\T}^2, \mathcal{F}_{\th/m^2} \right)\otimes \mathcal K,~  \Z_{m^2}  \right)$$
	is a noncommutative finite-fold covering.  Otherwise in \cite{ivankov:qnc} it is proven that the triple $\left(C\left(\T^2_\theta \right),~ C\left(\T^2_{\theta/m^2}\right) ,~ \Z_{m^2} \right)$ is a noncommutative finite-fold covering, and from the Theorem \ref{stable_fin_cov_thm} it turns out that the triple $\left(C\left(\T^2_\theta \right)\otimes \mathcal K,~ C\left(\T^2_{\theta/m^2}\right) \otimes \mathcal K, ~\Z_{m^2} \right)$ is a noncommutative finite-fold covering. There are natural *-isomorphisms 
	\begin{equation*}
	\begin{split}
	C\left(\T^2_\theta \right)\otimes \mathcal K \approx C^*_r\left(\T^2, \mathcal{F}_\th \right),\\
	C\left(\T^2_{\theta/m^2} \right)\otimes \mathcal K \approx C^*_r\left(\widetilde{\T}^2, \mathcal{F}_{\th/m^2} \right).
	\end{split}
	\end{equation*}
	From the above isomorphism it turns out that a noncommutative finite fold covering  $$\left(C^*_r\left(\T^2, \mathcal{F}_\th \right),~ C^*_r\left(\widetilde{\T}^2, \mathcal{F}_{\th/m^2} \right), ~\Z_{m^2}  \right)$$
	is equivalent to $\left(C\left(\T^2_\theta \right)\otimes \mathcal K,~ C\left(\T^2_{\theta/m^2}\right) \otimes \mathcal K, ~\Z_{m^2} \right)$.

\end{example}

\subsection{Infinite coverings}
\paragraph*{}
Let $\left(M, \mathcal F \right)$ be a Hausdorff foliation, on a compact $M$ and let
\begin{equation}\label{fol_man_seq_eqn}
\mathfrak{S}_M = \left\{
M = M_0 \leftarrow M_1 \leftarrow \dots  \leftarrow M_1 \leftarrow \dots\right\}
\end{equation} 
be a sequence of regular finite fold coverings. From the Theorem \ref{fol_fin_thm} it follows that there is an  (algebraical)  finite covering sequence
\begin{equation}\label{fol_alg_seq_eqn}
\begin{split}
\mathfrak{S}_{C^*_r\left(M, \mathcal F \right)}=\\= \left\{C^*_r\left(M, \mathcal F \right) = C^*_r\left(M_0, \mathcal F_0 \right)\to  C^*_r\left(M_1, \mathcal F_1 \right)\to \dots \to  C^*_r\left(M_n, \mathcal F_n \right)\to \dots
\right\}\in \\ \in \mathfrak{FilAlg}\end{split}
\end{equation} 
and for any $n \in \N$ the there is an isomorphism of
covering transformation groups
\begin{equation*}
G\left(C^*_r\left(M_n, \mathcal F_n \right)~|~C^*_r\left(M, \mathcal F \right) \right) \xrightarrow{\approx} C\left(M_n ~|M \right). 
\end{equation*}
Denote by $G_n \stackrel{\mathrm{def} }{=} C\left(M_n ~|M \right)$ and $\widehat{G} \stackrel{\mathrm{def} }{=} \varprojlim G_n$ an inverse limit of groups \cite{spanier:at}. Denote by  $\widehat{M}\stackrel{\mathrm{def} }{=} \varprojlim M_n$ the inverse limit of topological spaces \cite{spanier:at} and $\widehat{C^*_r\left(M, \mathcal F \right)}\stackrel{\mathrm{def} }{=}\varinjlim C^*_r\left(M_n, \mathcal F_n \right)$ $C^*$-inductive limit \cite{murphy}. In \cite{ivankov:qnc} it is proven that there is the disconnected inverse limit of $\mathfrak{S}_M$ i.e. a topological space $\overline{M}$ and a bicontinuous map $\overline{M} \to \widehat{M}$
such that for any $n \in N$ the composition map $\overline{\pi}_n: \overline{M} \to \widehat{M} \to M_n$ of the following diagram.  

 \begin{tikzpicture}
\matrix (m) [matrix of math nodes,row sep=3em,column sep=4em,minimum width=2em]
{
 \varprojlim \downarrow \mathfrak{S}_M & & & \\
	\overline{M} & & & \\
	\widehat{M}	&   &  & \\
	{M}_0	& {M}_1 &  {M}_2 & \dots \\};
\path[-stealth]
(m-1-1) edge node [left] {$\iota$} (m-2-1)
(m-2-1) edge node [left] {} (m-3-1)
(m-3-1) edge node [left] {} (m-4-1)
(m-3-1) edge node [left] {} (m-4-2)
(m-3-1) edge node [left] {} (m-4-3)
(m-3-1) edge node [left] {} (m-4-4)
(m-4-2) edge node [left] {} (m-4-1)
(m-4-4) edge node [left] {} (m-4-3)
(m-4-3) edge node [left] {} (m-4-2);
\end{tikzpicture} 
\newline
is a covering.  In the above diagram $\varprojlim\downarrow \mathfrak{S}_M$ is the topological inverse noncommutative limit (cf. \ref{top_inf_constr}) which is a connected component of $\overline{M}$. The group $\widehat{G} = \varprojlim G\left(M_n~|~M \right)$ naturally acts on $\widehat{M}$ and $\overline{M}$. In \cite{ivankov:qnc} it is proven that
	\begin{equation}\label{fol_sqcup}
\overline{M} = \bigsqcup_{g \in J} g \left(  \varprojlim \downarrow \mathfrak{S}_M\right) 
	\end{equation}  
	where $J \subset \widehat{G}$ is a set of representatives of $\widehat{G}/ G\left( \varprojlim \downarrow \mathfrak{S}_M~|~M \right)$.
The group  $G\left( \varprojlim \downarrow \mathfrak{S}_M~|~M \right)$ is a maximal among subgroups $G \subset \widehat{G}$ such that
$$
 G\left( \varprojlim \downarrow \mathfrak{S}_M\right)  =  \varprojlim \downarrow \mathfrak{S}_M.
$$

\subsubsection{Equivariant geometric representation}
\paragraph*{} 
Denote by
$$
\overline{\H} = \bigoplus_{\overline{x} \in \overline{M}} L^2\left(\mathcal G_{\overline{x} } \right)
$$
the Hilbert norm completion of an algebraic direct sum. There is a natural isomorphism $L^2\left(\mathcal G_{\overline{x} } \right) \approx L^2\left(\mathcal G_{\overline{   \pi }_n\left( \overline{x}\right)  } \right)$. It turns out that for 
 any $a_n \in C^*_r\left(M_n, \mathcal F_n \right)$ and $ \overline{x}\in  \overline{M}$ one can define $\rho_{ \overline{x}}\left( a_n\right)$ as isomorphic image of  $\rho_{\overline{\pi}_n\left(  \overline{x}\right) }\left( a_n\right)$, so there is a representation $\rho_{ \overline{x}}:C^*_r\left(M_n, \mathcal F_n \right) \to L^2\left(\mathcal G_{\overline{x} } \right)$. The direct sum of representations $\rho_{ \overline{x}}$ yields
a representation $\rho_n:C^*_r\left(M_n, \mathcal F_n \right) \to B\left( \overline{\H} \right)$.
There is a following commutative diagram.
\newline
\begin{tikzpicture}
\matrix (m) [matrix of math nodes,row sep=3em,column sep=4em,minimum width=2em]
{
	 C^*_r\left(M_n, \mathcal F_n \right) 	&   & C^*_r\left(M_{n+1}, \mathcal F_{n+1} \right) \\
	& B\left(\overline{\H}\right) &  \\};
\path[-stealth]
(m-1-1) edge node [left] {} (m-1-3)
(m-1-1) edge node [right] {$~~\rho_n$} (m-2-2)
(m-1-3) edge node [above] {$\rho_{n+1}~~$}  (m-2-2);

\end{tikzpicture}
\newline
From the above diagram it follows that there is an equivariant representation
\begin{equation}\label{fol_equ_rep_eqn}
\pi_{\mathrm{geom}}:\widehat{C^*_r\left(M, \mathcal F \right)}=\varinjlim C^*_r\left(M, \mathcal F \right)\to B\left(\overline{\H}\right)
\end{equation}
Otherwise if $C^*_r\left(\overline{M}, \overline{\mathcal F} \right)$ is the $\overline{\pi}$-lift of  $C^*_r\left(M_n, \mathcal F_n \right)$ then there is the natural representation
\begin{equation}\label{fol_rr_eqn}
C^*_r\left(\overline{M}, \overline{\mathcal F} \right) \hookto B\left(\overline{\H}\right)
\end{equation}
given by \eqref{fol_repr}.
\begin{definition}
	We say that representation \eqref{fol_equ_rep_eqn} is a \textit{geometric representation} of the sequence \eqref{fol_alg_seq_eqn}.
\end{definition}
\subsubsection{Inverse noncommutative limit}
\paragraph*{} From \eqref{fol_rr_eqn} it follows that any $\overline{a} \in C^*_r\left(\overline{M}, \overline{\mathcal F} \right)$ can be regarded as element in $B\left(\overline{\H}\right)$, i.e. $\overline{a} \in B\left(\overline{\H}\right)$. 
\begin{lemma}\label{fol_com_inc_lem}
	Let $\overline{\mathcal{U}} \subset \overline{M}$ be a connected open subset homeomorphically mapped onto $\mathcal{U} = \overline{\pi} \left( \overline{\mathcal{U}}\right)$. Let $\overline{\th}: C^*_r\left(\overline{\mathcal{U}}, \overline{\mathcal F}|_{\overline{\mathcal{U}}} \right) \hookto C^*_r\left(\overline{M}, \overline{\mathcal F} \right)$ be given by \eqref{fol_res_hom_eqn} *- homomorphism. If $\overline{b} \in C^*_r\left(\overline{\mathcal{U}}, \overline{\mathcal F}|_{\overline{\mathcal{U}}} \right)_+$ and $\overline{a} = \overline{\th}\left(\overline{b} \right)$ then $\overline{a} \in  B\left(\overline{\H}\right)_+$ is a special element  of the sequence \eqref{fol_alg_seq_eqn}. 
	\end{lemma}
\begin{proof}
	Denote by ${\mathcal{U}_n} = \overline{\pi}_n\left(\overline{\mathcal{U}} \right)$. For any $n \in \N^0$ there is the natural isomorphism
	$$
	\varphi_n : C^*_r\left(\overline{\mathcal{U}}, \overline{\mathcal F}|_{\overline{\mathcal{U}}} \right) \xrightarrow{\approx} C^*_r\left({\mathcal{U}_n}, {\mathcal F_n}|_{{\mathcal{U}_n}} \right)
	$$
	such that 
\begin{equation}
\sum_{g \in \ker\left(\widehat G \to G_n \right)} g \overline \th \left( \overline b\right)  = \th_n \circ	\varphi_n\left( \overline b \right) 
\end{equation}
where $\th_n: C^*_r\left({\mathcal{U}_n}, {\mathcal F_n}|_{{\mathcal{U}_n}} \right) \hookto C^*_r\left(M_n, \mathcal F_n \right) $ is given by \eqref{fol_res_hom_eqn}. It follows that
$$
a_n = \sum_{g \in \ker\left(\widehat G \to G_n \right)} g \overline a = \sum_{g \in \ker\left(\widehat G \to G_n \right)} g \overline \th \left( \overline b\right) =\th_n \circ	\varphi_n\left( \overline b \right) \in  C^*_r\left({M_n}, {\mathcal F_n} \right),
$$
i.e. $\overline{   a}$ satisfies to the condition (a) of the Definition \ref{special_el_defn}. 	If  $f_\eps: \R \to \R$ is given by 
\eqref{f_eps_eqn} then for any $n \in \mathbb{N}^0$ and for any $z \in C^*_r\left(M, \mathcal F \right)$ following conditions hold
\begin{equation*}
\begin{split}
b_n = \sum_{g \in \ker\left( \widehat{G} \to  G_n\right)} g \left(z  \overline{a} z^*\right) = z\left( \th_n\circ \varphi_n\left(\overline b \right) \right)z^* \in C^*_r\left(M, \mathcal F \right), \\
c_n = \sum_{g \in \ker\left( \widehat{G} \to  G_n\right)} g \left(z  \overline{a} z^*\right)^2=\left( z\left( \th_n\circ \varphi_n\left(\overline b \right) \right)z^*\right)^2 = b_n^2 \in C^*_r\left(M, \mathcal F \right),\\
d_n = \sum_{g \in \ker\left( \widehat{G} \to  G_n\right)} g f_\eps\left( z  \overline{a} z^* \right) =f_\eps\left( z\left( \th_n\circ \varphi_n\left(\overline b \right) \right)z^*\right)  \in C^*_r\left(M, \mathcal F \right),
\end{split}
\end{equation*}
i.e. the condition (b) of the Definition \ref{special_el_defn} holds.
From $c_n = b_n^2$ it turns out the condition (c) of the Definition \ref{special_el_defn} holds.
\end{proof}
Denote by $\overline{C^*_r\left(M, \mathcal F \right)}$ the disconnected inverse limit of the sequence 
$\mathfrak{S}_{C^*_r\left(M, \mathcal F \right)}\in \mathfrak{FinAlg}$  given by \eqref{fol_alg_seq_eqn}, with respect to the geometric representation $\pi_{\mathrm{geom}}$ (cf. \eqref{fol_equ_rep_eqn}).
\begin{cor}\label{fol_in_cor}
Following condition holds
$$
C^*_r\left(\overline{M}, \overline{\mathcal F} \right)\subset\overline{C^*_r\left(M, \mathcal F \right)}.
$$
\end{cor}
\begin{proof}
	The $C^*$-norm completion of the algebra generated by special elements given by the Lemma \ref{fol_com_inc_lem} coincides with $C^*_r\left(\overline{M}, \overline{\mathcal F} \right)$.
\end{proof}
\begin{empt}
	Denote by $p = \dim \mathcal{F}$, $q = {\rm codim} \, \mathcal{F}$.
 Let us consider a foliation chart $\phi:\R^{\dim M} \to M$ such that leafs of the foliation $\left(\R^{\dim M}, \mathcal{F}_p \right)$  given by \eqref{fol_chart_eqn} are mapped into leafs of $\left( M, \mathcal F\right) $. The set 
 \begin{equation}\label{fol_trans_eqn}
\mathcal V = \phi\left(\left\{0\right\}\times \R^q \right)
 \end{equation}
 is a submanifold of $M$ which is traversal  to $\mathcal{F}$.  If $b \in C_0\left(\R^p \times \R^q \right)_+ = C_0\left(\R^{\dim M}\right)_+$ is a positive function such that $b\left(\R^{\dim M} \right)\subset\left[ 0,1\right]$ then $b$  can be represented by a following way
 \begin{equation}\label{fol_r_fact_eqn}
b = \left(  1_{C_b\left(\R^p\right)  }\otimes b'\right) b''
 \end{equation}
 where $b' \in C_0\left(\R^q \right)_+$ and  $b'' \in C_0\left(\R^{\dim M} \right)$ and following conditions hold:
 \begin{equation*}
 \begin{split}
  b'\left(\R^q \right)\subset\left[ 0,1\right] , ~~ \\ b''\left(\R^{\dim M} \right)\subset\left[ 0,1\right].
 \end{split}
 \end{equation*}
  Clearly $b', b''$ correspond to contractible operators.
  If $\mathcal U = \phi\left(\R^{\dim M} \right)$ then any $a \in C_0 \left(\mathcal U \right)$ is an image of  $b \in C_0\left(\R^p \times \R^q \right)$ and the factorization \eqref{fol_r_fact_eqn} corresponds to the factorization 
  \begin{equation}\label{fol_u_fact_eqn}
 a = a' a''
 \end{equation}
 where $a'$ (resp. $a''$) corresponds to $1_{C_b\left(\R^p \right)}\otimes b'$ (resp. $b''$). The element $a'$ can be regarded as an element of $C_0\left(\mathcal V \right)$. The manifold $M$ is compact, so there is a finite set $$\left\{\phi_\iota: \R^{\dim M} \to \mathcal U_\iota \subset M\right\}_{\iota \in I}$$ of foliation charts such that $M = \bigcup_{\iota \in I}\mathcal U_\iota$. There is a partition of unity
\begin{equation}\label{fol_dec_eqn}
 1_{C\left( M\right) } = \sum_{\iota \in I} a_\iota = \sum_{\iota \in I} a'_\iota a''_\iota
\end{equation}
 where $a_\iota \in C\left( M\right)_+ $, $\supp a_\iota \subset \mathcal U_\iota$ and $a_\iota = a'_\iota a''_\iota$ is a factorization given by \eqref{fol_u_fact_eqn}.  
   From the Example \ref{fol_trivial_k_exm} it follows that $C^*_r\left( \R^{\dim M}, \mathcal{F}_p\right) \cong C_0\left(\R^p \right) \otimes \mathcal K $. Similarly one has $C^*_r\left( \mathcal U_\iota, \mathcal{F}_{\mathcal U_\iota}\right) \cong C_0\left(\mathcal V_\iota \right) \otimes \mathcal K$ where $\mathcal V_\iota$ is a transversal manifold given by \eqref{fol_trans_eqn}. The formula $C^*_r\left( \mathcal U_\iota, \mathcal{F}_{\mathcal U_\iota}\right) \cong C_0\left(\mathcal V_\iota \right) \otimes \mathcal K$ implies that there is a separable Hilbert space $\H$ with an orthogonal basis $\xi_1, \dots , \xi_j, \dots$ such that elements of $C^*_r\left( \mathcal U_\iota, \mathcal{F}_{\mathcal U_\iota}\right)$ are operators
   $$
   C_b\left(\mathcal V_\iota \right) \otimes \H \to  C_0\left(\mathcal V_\iota \right) \otimes \H.
   $$
   Above operators can be regarded as compact operators in $\mathcal K\left(\ell^2\left(C_0\left(\mathcal V_\iota \right) \right)  \right)$. 
   If   $p_j \in \mathcal K = \mathcal K\left(\H \right)$ is a projector along $\xi_j$ then from
   $$
   1_{M\left(\mathcal K \right) } = \sum_{j = 1}^\infty p_j
   $$
   it follows that
   $$
   1_{M\left( C_0\left(\mathcal V_\iota \right) \otimes \mathcal K\right) } = 1_{ C_b\left(\mathcal V_\iota \right)} \otimes \sum_{j = 1}^\infty p_j
   $$
   where the sum of the above series implies the strict convergence \cite{blackadar:ko}.
   From above formulas it follows that the series
   \begin{equation}\label{fol_unity_decomp_eqn}
   1_{C\left( M\right) } = 1_{M\left(C^*_r\left(M, \mathcal F \right) \right) } = \sum_{\iota \in I}  a'_\iota a''_\iota \otimes \sum_{j = 1}^\infty p_j
   \end{equation}
   is strictly convergent.
\end{empt}
\begin{remark}
	Clearly $a'_\iota a''_\iota \otimes  p_j,  a''_\iota \otimes  p_j \in C^*_r\left(M, \mathcal F \right)$.
\end{remark}
\begin{empt}
	Let us consider the geometric representation 
\begin{equation*}
	\pi_{\mathrm{geom}}:\widehat{C^*_r\left(M, \mathcal F \right)}=\varinjlim C^*_r\left(M_n, \mathcal F_n \right)\to B\left(\overline{\H}\right)
\end{equation*}
given by \eqref{fol_equ_rep_eqn}. Suppose that $\overline{a} \in B\left(\overline{\H}\right)_+$ is a special element of the sequence \eqref{fol_alg_seq_eqn}. We would like to proof that
\begin{equation}\label{fol_out_eqn}
\overline{a} \in C^*_r\left(\overline{M}, \overline{\mathcal F} \right).
\end{equation}
The equation \eqref{fol_out_eqn} follows from three facts:
\begin{enumerate}
	\item There is a decomposition 
\begin{equation}\label{fol_spec_decomp_eqn}
	\overline{a} = \sum_{ \iota'\in I} \sum_{ \iota''\in I}  \sum_{j = 1}^\infty  \sum_{k = 1}^\infty \left( a'_{\iota'} a''_{\iota'} \otimes p_j\right) \overline{a} \left( a'_{\iota''} a''_{\iota''} \otimes p_k\right),
\end{equation}
	\item One has $ \left( a'_{\iota'} a''_{\iota'} \otimes p_j\right) \overline{a} \left( a'_{\iota''} a''_{\iota''} \otimes p_k\right) \in C^*_r\left(\overline{M}, \overline{\mathcal F} \right)$ for any $\iota', \iota'' \in I$, $j,k \in \N$, 
	\item The series \eqref{fol_spec_decomp_eqn} is norm convergent.
\end{enumerate}
The decomposition \eqref{fol_spec_decomp_eqn} follows from \eqref{fol_unity_decomp_eqn}. Facts 2 and 3 are proven below.
\end{empt}
\begin{empt}\label{fol_constr}
Let us fix $\iota', \iota'' \in I$. The subset $$
\mathcal G_{\iota'\iota''} = \left\{\gamma \in \mathcal G\left(M, \mathcal F \right) ~|~ s\left(\gamma \right)\in \mathcal U_{\iota'} , r\left(\gamma \right)\in \mathcal U_{\iota''}\right\}
$$
can be decomposed into connected components, i.e.
\begin{equation}\label{fol_decomp_eqn}
\mathcal G_{\iota'\iota''} = \bigsqcup_{\lambda \in \Lambda_{\iota'\iota''}} \mathcal G_\lambda.
\end{equation}

 If $\mathcal W'_\lambda \subset  \mathcal V_{\iota'}$ and $\mathcal W''_\lambda \subset  \mathcal V_{\iota''}$ are given by
\begin{equation*}
\begin{split}
\mathcal W'_\lambda  = \left\{x \in \mathcal V_{\iota'}~|~ \exists \gamma \in \mathcal G_\lambda;~ x \in s\left(\gamma \right) \right\},\\
\mathcal W''_\lambda  = \left\{x \in \mathcal V_{\iota''}~|~ \exists \gamma \in \mathcal G_\lambda;~ x \in r\left(\gamma \right) \right\}
\end{split}
\end{equation*}
then $\mathcal G_\lambda$  corresponds to a diffeomorphism $\varphi_\lambda: \mathcal W'_\lambda \xrightarrow{\approx} \mathcal W''_\lambda$. Below we drop the index $\lambda$ and $\mathcal W'$, $\mathcal W''$ instead of $\mathcal W'_\lambda$, $\mathcal W'_\lambda$.  For any subset $\mathcal U$ subset  of a topological set $\mathcal X$ denote by $\partial \mathcal U = \overline{\mathcal U} \backslash \mathcal U \subset \mathcal X$ where $\overline{\mathcal U}$ is a closure of $\mathcal U$. Any $a_{\mathcal W'} \in C_0\left( \mathcal W'\right)$ can be regarded as element of $C_b\left( \mathcal W'\right)$ such that $a_{\mathcal W'}\left(\partial \mathcal W' \right)=\{0\}$. Clearly $\partial \mathcal W' \subset \partial \mathcal V' \bigcup \partial\varphi^{-1}_\lambda\left(\mathcal V'' \right)$. If $b' \in  C_0\left( \mathcal W'\right)$ and $b'' \in  C_0\left( \mathcal W''\right)$ then the isomorphism $\varphi_\lambda$ gives a product $b'b'' \in C_b\left(\mathcal W' \right)$ such that $\left(b'b'' \right)\left(\partial \mathcal W' \right)=\{0\}$.  It follows that $b'b'' \in C_0\left(\mathcal W' \right)$.
 Let us define an element $y \in C^*_r\left(M, \mathcal F \right)$ such that $\supp y \in \mathcal G_{\lambda}$ and $y$ is given by
\begin{equation}\label{fol_yljk_eqn}
\begin{split}
y: C_b\left( \mathcal W'\right) \otimes \H \to C_0\left( \mathcal W''\right)\otimes \H,\\
1_{C_b\left( \mathcal W'\right) } \otimes \xi_l \mapsto \delta_{jl}  \sqrt[3]{a'_{\iota'}}~ \sqrt[2]{a'_{\iota''}} \xi_k; ~(\delta_{jj}=1, ~ \delta_{jk}=0 \text{ if } j\neq k),
\end{split}
\end{equation}
and note that $\sqrt[3]{a'_{\iota'}}~ \sqrt[2]{a'_{\iota''}} \in  C_0\left(\mathcal W'' \right)$. Moreover $y$ can be regarded as a compact operator in $\mathcal K\left(\ell^2\left( C_0\left(\mathcal W' \right)\right),\ell^2\left( C_0\left(\mathcal W'' \right)\right)  \right)$.   
Denote by 
\begin{equation}\label{foll_e_eqn}
e_\iota = \sqrt[3]{a'_{\iota'}}\otimes p_j\in C^*_r\left( {\mathcal U_{\iota'}}, \mathcal F_{\mathcal U_{\iota'}}\right)\subset C^*_r\left( M, \mathcal F\right)
\end{equation}
 where  $p_j \in \mathcal K = \mathcal K\left(\H \right)$ is a projector along $\xi_j$, and let 
\begin{equation}\label{fol_z_eqn}
z = y^* + e
\end{equation}
Let $\overline{a} \in B\left(\overline{\H}\right)_+$ be a special element of the sequence \eqref{fol_alg_seq_eqn}. From the (a) of the Definition \ref{special_el_defn} it turns out
$$
a = \sum_{g \in G}g \overline{a} \in  C^*_r\left(M, \mathcal F \right)
$$
If $\mathcal G_{\iota'} \subset \mathcal G\left( M, \mathcal F\right)$ is a set of path which a homotopic to a trivial path $\gamma\left( \left[0,1\right] \right) \in\mathcal U_{\iota'}$ and $s\left(\gamma \right), r\left(\gamma \right) \in \mathcal U_{\iota'}$ then the restriction of $zaz^*$ on  $\mathcal G_{\iota'}$ is a "rank-one" positive operator, given by
\begin{equation}\label{fol_rank_one_eqn}
zaz^*|_{\mathcal G_{\iota'}}= \theta_{\iota' \iota'' \lambda jk}\otimes p_j,
\end{equation}

where  $\theta_{\iota' \iota'' \lambda jk}  \in C_0\left(\mathcal W_{\iota'} \right)_+$. If $f_\eps: \R \to \R$ is given by 
\eqref{f_eps_eqn}
then from (b) of the Definition \ref{special_el_defn} it turns out 
that for any $n \in \mathbb{N}^0$ following conditions hold 
\begin{equation*}
\begin{split}
b_n = \sum_{g \in \ker\left( \widehat{G} \to  G\left( A_n~|~A \right)\right)} g \left(z  \overline{a} z^*\right)  \in C^*_r\left(M_n, \mathcal F_n \right) ,\\
c_n = \sum_{g \in \ker\left( \widehat{G} \to  G\left( A_n~|~A \right)\right)} g \left(z  \overline{a} z^*\right)^2 \in C^*_r\left(M_n, \mathcal F_n \right),\\
d_n = \sum_{g \in \ker\left( \widehat{G} \to  G\left( A_n~|~A \right)\right)} g f_\eps\left( z  \overline{a} z^* \right) \in C^*_r\left(M_n, \mathcal F_n \right).
\end{split}
\end{equation*}
From the condition (c) of the Definition \ref{special_el_defn} it follows that for any $\eps > 0$ there is $N \in \N$ such that 
$$
\left\|b_n^2 - c_n \right\|< \eps
$$ 
for any $n \ge N$. If  $\mathcal G_{\iota'n} \in \mathcal G\left( M_n, \mathcal F_n\right)$ is given by $\mathcal G_{\iota'n} = \pi^{-1}_n\left(\mathcal G_{\iota'} \right)$ then similarly to \eqref{fol_rank_one_eqn} one has
\begin{equation*}
\begin{split}
b_n|_{\mathcal G_{\iota'n}} = b'_n \otimes p_j,\\
c_n|_{\mathcal G_{\iota'n}} = c'_n \otimes p_j,\\
d_n|_{\mathcal G_{\iota'n}} = d'_n \otimes p_j,\\
\end{split}
\end{equation*}
where $b'_n, c'_n, d'_n \in C_0\left(\pi^{-1}_n\left( \mathcal W_{\iota'}\right)  \right)_+$.
 If $\overline{b}' \in C_0\left(\overline{\pi}^{1} \left(\mathcal W_{\iota'} \right)  \right)''$ is a strong limit
 \begin{equation}\label{fol_strong_lim_eqn}
 \overline{b}' = \lim_{n \to \infty} b'_n
 \end{equation} 
  then following condition holds:
 \begin{equation*}
 \begin{split}
 b'_n = \sum_{g \in \ker\left(\widehat{G}\to G_n \right)} g \overline{b}',\\
 c'_n = \sum_{g \in \ker\left(\widehat{G}\to G_n \right)} g \overline{b}'^2,\\
d'_n = \sum_{g \in \ker\left(\widehat{G}\to G_n \right)} g f_\eps\left( \overline{b}'\right).
 \end{split}
 \end{equation*}
 From $\left\|b_n^2 - c_n \right\|< \eps$ it follows that $\left\|b'^2_n - c'_n \right\|< \eps$.
  Now we need a following lemma.
\end{empt}
\begin{lemma}\label{comm_main_lem}\cite{ivankov:qnc}
	Suppose that $\mathcal X$ is a  locally compact Hausdorff space. Let $\overline{a} \in C_0\left(\overline{   \mathcal X }\ \right)''_+$ be such that following conditions hold:
	\begin{enumerate}
		\item[(a)]  If $f_\eps$ is given by \eqref{f_eps_eqn} then following series
		\begin{equation*}
		\begin{split}
		a_n = \sum_{g \in \ker\left( \widehat{G} \to  G_n\right)} g \overline a,\\
		b_n = \sum_{g \in \ker\left( \widehat{G} \to  G_n\right)} g \overline a^2,\\
		c_n = \sum_{g \in \ker\left( \widehat{G} \to  G_n\right)} g f_\eps\left( \overline a\right) ,\\
		\end{split}
		\end{equation*}
		are strongly convergent and  $a_n, b_n, c_n \in C_0\left(\mathcal X_n \right)$,
		\item[(b)] For any $\eps$ there is $N \in \N$ such that 
		\begin{equation*}
		\begin{split}
		\left\|a^2_n-b_n\right\| < \eps; ~\forall n \ge N.
		\end{split}
		\end{equation*}
	\end{enumerate}
	Then $\overline{a} \in C_0\left(\overline{   \mathcal X }\ \right)_+$.
\end{lemma}

\begin{corollary}\label{fol_cor}
If  $\overline{a} \in B\left(\overline{\H}\right)_+$ is a special element of the sequence \eqref{fol_alg_seq_eqn} and $z \in C^*_r\left( M, \mathcal F\right)$ is given by \eqref{fol_z_eqn} then 
$$
\overline{b} = z \overline{a} z^* \in C^*_r\left(\overline M, \overline{\mathcal F}\right).
$$ 
\end{corollary}
\begin{proof}
	Let $\eps > 0$.
	If $f_{\eps}$ is given by \eqref{f_eps_eqn} and 
 $\overline{b}_{\eps/2} = f_{\eps/2}\left(\overline{b} \right)$ then 
 $$
 \left\|\overline{b}-\overline{b}_{\eps/2} \right\| < \frac{\eps}{2}.
 $$
 If $\overline{b} = \overline{b}' \otimes p_j$  then $\overline{b}_{\eps/2} = \overline{b}'_{\eps/2} \otimes p_j$ where $\overline{b}'_{\eps/2} = f_{\eps/2}\left(\overline{b}' \right) $. From  the construction \ref{fol_constr} and the Lemma \ref{comm_main_lem} it follows that if $\overline{b}'$ is given by \eqref{fol_strong_lim_eqn} then $\overline{b}' \in C_0\left(\overline{\pi }^{-1}\left( \mathcal W_{\iota'}\right)  \right)$. It follows that $\supp~\overline{b}'_{\eps/2} \subset \overline{\pi }^{-1}\left( \mathcal W_{\iota'}\right)$ is compact. There is $N \in \N$ such that for any $n \ge N$ a restriction 
 $$
 \overline{\pi}_n|_{\supp~\overline{b}'_{\eps/2}}:\supp~\overline{b}'_{\eps/2} \xrightarrow{\approx}  \overline{\pi}_n\left( \supp~\overline{b}'_{\eps/2}\right)\subset M_n 
 $$
 is a homeomorphism. For any $\overline{x} \in \supp~\overline{b}'_{\eps/2}$ there is an open neighborhood $\overline{   \mathcal U }$ which is homeomorphically mapped onto $\mathcal U_{\iota'}$. Since $\supp~\overline{b}'_{\eps/2}$ is compact the set 
 $$
F \subset \widehat{G} = \left\{g \in \widehat{G}~|~ g\overline{   \mathcal U } \bigcap \supp \overline{b}'_{\eps/2} \neq \emptyset\right\}
 $$
 is finite. If $\overline{   \mathcal U }^{\eps/2}= F\overline{ \mathcal U}\subset \overline M$ then for any $n > N$ the set $\overline{   \mathcal U }^{\eps/2}$ is mapped homeomorphically onto $\overline{\pi}_n\left(\overline{   \mathcal U }^{~\eps/2} \right) =
 \mathcal U^{\eps/2}_n$. Let $b^{\eps/2}_N \in C^*_r\left( M_N, \mathcal F_N\right)$ is given by
 $$
 b^{\eps/2}_N = \sum_{g \in \ker\left( \widehat{G} \to  G_n\right)} g \overline{b}_{\eps/2}.
 $$
 There is $b^\infty_N \in C_c^\infty\left( M_N, \mathcal F_N \right)_+$ such that following conditions hold:
 \begin{equation*}
\begin{split}
\left\|b^{\eps/2}_N-b^\infty_N \right\| < \frac{\eps}{2},\\
	\mathcal U^\infty_N = \left\{x \in M_N ~ | ~ \exists \gamma \in \supp b^\infty_N;~ x = s\left(\gamma \right) \right\} \subset \mathcal U^{\eps/2}_N.
\end{split}
 \end{equation*}
For any path $\gamma \in \supp b^\infty_N$ there is the unique $\overline{\pi}_N$-lift $\overline{\gamma}\in \mathcal G\left(\overline{M}, \overline{   \mathcal F}\ \right) $ such that $s\left(\overline{\gamma} \right) \in \overline{   \mathcal U }^{\eps/2}$. We say this lift \textit{special}. There is an element $\overline{b}^\infty \in C^\infty_c\left(\overline M, \overline{\mathcal F}\right)$  such that any "value" (half density) of $b^\infty_N$ on a path $\gamma \in \mathcal G\left( M_N, {\mathcal F}_N \right)$ coincides  with "value" of $\overline{b}^\infty$ on the special lift $\overline{\gamma}\subset \mathcal G\left(\overline M, \overline{\mathcal F}\right)$ of $\gamma$. We also require that $\overline{b}^\infty$ is trivial on paths which are not special lifts. From $	\left\|b^{\eps/2}_N-b^\infty_N \right\| < \frac{\eps}{2}$ it follows that $\left\|\overline{b}_{\eps/2}-\overline{b}^\infty \right\| < \frac{\eps}{2}$ and from $\left\|\overline{b}-\overline{b}_{\eps/2} \right\| < \eps/{2}$ it follows that $\left\|\overline{b}-\overline{b}^\infty \right\|< \eps$. An algebra $C^*_r\left(\overline M, \overline{\mathcal F}\right)$ is the $C^*$-norm completion of $C^\infty_c\left(\overline M, \overline{\mathcal F}\right)$ it turns out 
$$
\overline{b} \in C^*_r\left(\overline M, \overline{\mathcal F}\right).
$$
\end{proof}
\begin{empt}
	The described in \ref{fol_constr} construction depends on indexes $\iota', \iota'' \in I$,  a connected component $\mathcal G_\lambda \subset \mathcal G_{\iota'\iota''}$ ($\lambda \in \Lambda_{\iota'\iota''}$) in the decomposition \eqref{fol_decomp_eqn}, and $j, k \in \N$. Now for clarity we use $y_{\iota'\iota''\lambda jk}$ (resp. $z_{\iota'\iota''\lambda jk}$) instead $y$ given by \eqref{fol_yljk_eqn} (resp. $z$  given by \eqref{fol_z_eqn}. Thus from the Corollary \ref{fol_cor} it turns out 
	$$
	z_{\iota'\iota''\lambda jk} \overline{a}	z^*_{\iota'\iota''\lambda jk} \in  C^*_r\left(\overline M, \overline{\mathcal F}\right).
	$$
	From $C^*_r\left( M, {\mathcal F}\right)\subset M\left( C^*_r\left(\overline M, \overline{\mathcal F}\right)\right) $ it follows that $y_{\iota'\iota''\lambda jk} \in M\left( C^*_r\left(\overline M, \overline{\mathcal F}\right)\right)$, hence
	$$
	y_{\iota'\iota''\lambda jk}	z_{\iota'\iota''\lambda jk} \overline{a}	z^*_{\iota'\iota''\lambda jk} \in  C^*_r\left(\overline M, \overline{\mathcal F}\right).
	$$
	 From \eqref{fol_yljk_eqn} and \eqref{foll_e_eqn} on has a following "formal decomposition"
	\begin{equation}\label{fol_formal_eqn}
\left( a'_{\iota'} \otimes p_j\right) \overline{a} \left( a'_{\iota'} \otimes p_k\right) = \sum_{\lambda \in \Lambda_{\iota'\iota''}} 	y_{\iota'\iota''\lambda jk}	z_{\iota'\iota''\lambda jk} \overline{a}	z^*_{\iota'\iota''\lambda jk}~,
	\end{equation}
(The "formal decomposition" word means that one \textit{should} prove that the series \eqref{fol_formal_eqn} is norm convergent).	
	Denote by $\Xi$ the set of quintuplets $\left(\iota',\iota'',\lambda, j,k\right)$ where  $\iota',\iota'' \in I$, $\lambda \in \Lambda_{_{\iota'\iota''}}$ and $j,k \in \N$. From \eqref{fol_spec_decomp_eqn} and \eqref{fol_formal_eqn} it follows a "formal decomposition"
\begin{equation}\label{fol_full_eqn}
	\overline{a} = \sum_{\left(\iota',\iota'',\lambda, j,k\right)\in \Xi} a''_{\iota'}	y_{\iota'\iota''\lambda jk}	z_{\iota'\iota''\lambda jk} \overline{a}	z^*_{\iota'\iota''\lambda jk}a''_{\iota''}
\end{equation}

All terms in the series \eqref{fol_full_eqn} lie in $C^*_r\left(\overline M, \overline{\mathcal F}\right)$. Any operator
$$
	y_{\iota'\iota''\lambda jk}	z_{\iota'\iota''\lambda jk} \overline{a}	z^*_{\iota'\iota''\lambda jk}a''_{\iota''}
$$
is given by
$$
1_{C_b\left(\mathcal W_{\iota'} \right)}\otimes \xi_l \mapsto \delta_{lj} \vartheta_{\iota' \iota'' \lambda jk} \otimes p_k
$$
where $\vartheta_{\iota' \iota'' \lambda jk} \in C_0\left(\mathcal W_{\iota''} \right)$. For any $\iota', \iota'' \in I$ and $\lambda \in \Lambda_{\iota'\iota''}$ the sum  
\begin{equation}\label{fol_c_sum}
\overline{a}_{\iota' \iota'' \lambda}=	\sum_{\substack{j = 1\\ k=1}}^{\substack{j = \infty\\ k=\infty}} 	y_{\iota'\iota''\lambda jk}	z_{\iota'\iota''\lambda jk} \overline{a}	z^*_{\iota'\iota''\lambda jk}
\end{equation}
can be regarded as a compact operator from $\ell^2\left(C_0\left(\mathcal W_{\iota'} \right)\right)$ to from $\ell^2\left(C_0\left(\mathcal W_{\iota''} \right)\right)$, i.e.  $\overline{a}_{\iota' \iota'' \lambda} \in \mathcal K\left(\ell^2\left(C_0\left(\mathcal W_{\iota'} \right)\right) ,  \ell^2\left(C_0\left(\mathcal W_{\iota''} \right)\right) \right)$. If follows that for any $\delta > 0$ there is $N_\delta$ such that for any $n \ge N_\delta$ following condition holds
$$
\left\|\overline{a}_{\iota' \iota'' \lambda}-\sum_{\substack{j = 1\\ k=1}}^{\substack{j = n\\ k=n}} 	y_{\iota'\iota''\lambda jk}	z_{\iota'\iota''\lambda jk} \overline{a}	z^*_{\iota'\iota''\lambda jk} \right\| < \delta.
$$
Taking into account  $\left\|a''_{\iota'} \right\|\le 1$, $\left\|a''_{\iota''} \right\|\le 1$ one has
\begin{equation}\label{fol_le_de_eqn}
\left\|a''_{\iota''}\overline{a}_{\iota' \iota'' \lambda}a''_{\iota'}-\sum_{\substack{j = 1\\ k=1}}^{\substack{j = n\\ k=n}} 	a''_{\iota''}y_{\iota'\iota''\lambda jk}	z_{\iota'\iota'\lambda jk} \overline{a}	z^*_{\iota'\iota''\lambda jk} a''_{\iota'}\right\| < \delta.
\end{equation}
\end{empt}
\begin{lemma}
	If  $\overline{a} \in B\left(\overline{\H}\right)_+$ is a special element of the sequence \eqref{fol_alg_seq_eqn}  then 
	$
	\overline{a} \in C^*_r\left(\overline M, \overline{\mathcal F}\right).
	$ 
\end{lemma}
\begin{proof}
	Let $\eps > 0$ be a small number.
 Let $a = \sum_{g \in \widehat{G}} g \overline{a}\in C^*_r\left( M, {\mathcal F}\right)$ and let $a' \in  C^\infty_c\left( M, {\mathcal F}\right)$ be a positive element such that
 \begin{equation}\label{fol_e2_eqn}
 \left\|a - a' \right\| < \frac{\eps}{2}.
 \end{equation}
 From \eqref{fol_decomp_eqn} it follows the decomposition
 \begin{equation*}
 \mathcal G_{\iota'\iota''} = \bigsqcup_{\lambda \in \Lambda_{\iota'\iota''}} \mathcal G_\lambda.
 \end{equation*}
From $a' \in  C^\infty_c\left( M, {\mathcal F}\right)$ it turns out that for any $\iota', \iota'' \in I$ a set
 $$
 \Lambda'_{\iota'\iota''}= \left\{ \gamma  \in \mathcal G_{\iota'\iota''}~|~ a' \text{ is not trivial on }  \gamma\right\}
 $$
 is finite. Since the set $I$ from the decomposition \eqref{fol_dec_eqn} is finite the set
 $\Theta = \bigsqcup_{\iota', \iota'' \in I} \Lambda_{\iota'\iota''}$ is also finite. Let $C = \left| \Theta \right|\in \N$ be the cardinal number of $\Theta$ and let $\delta = \frac{\eps}{2C}$.
 From \eqref{fol_e2_eqn} it follows that
 \begin{equation}\label{fol_path_eqn}
\left\|\overline{a} - \sum_{\iota' \iota''\in I}\sum_{\lambda \in \Lambda'_{\iota'\iota''}} \sum_{\substack{j = 1\\ k=1}}^{\substack{j = \infty\\ k=\infty}} a''_{\iota'}	y_{\iota'\iota''\lambda jk}	z_{\iota'\iota''\lambda jk} \overline{a}	z^*_{\iota'\iota''\lambda jk}a''_{\iota''}\right\| \le \frac{\eps}{2}.
 \end{equation}
 From \eqref{fol_le_de_eqn} it turns out that for any $\lambda \in \Lambda_{\iota'\iota''}$ there is $N_\lambda$ such that for any $n \ge N_\lambda$ following condition holds
 \begin{equation}\label{fol_jk_eqn}
 \left\|\sum_{\substack{j = 1\\ k=1}}^{\substack{j = \infty\\ k=\infty}} a''_{\iota'}	y_{\iota'\iota''\lambda jk}	z_{\iota'\iota''\lambda jk} \overline{a}	z^*_{\iota'\iota''\lambda jk}a''_{\iota''}-\sum_{\substack{j = 1\\ k=1}}^{\substack{j = n\\ k=n}} a''_{\iota'}	y_{\iota'\iota''\lambda jk}	z_{\iota'\iota''\lambda jk} \overline{a}	z^*_{\iota'\iota''\lambda jk}a''_{\iota''} \right\| < \delta.
 \end{equation}
 If $\overline{a}_f$ is given by
$$
\overline{a}_f = \sum_{\iota' \iota''\in I}\sum_{\lambda \in \Lambda'_{\iota'\iota''}} \sum_{\substack{j = 1\\ k=1}}^{\substack{j = N_\lambda\\ k=N_\lambda}} a''_{\iota'}	y_{\iota'\iota''\lambda jk}	z_{\iota'\iota''\lambda jk} \overline{a}	z^*_{\iota'\iota''\lambda jk}a''_{\iota''}
$$
then  $\overline{a}_f \in C^*_r\left(\overline M, \overline{\mathcal F}\right)$ because $\overline{a}_f$ is finite sum of elements each of which lies in  $C^*_r\left(\overline M, \overline{\mathcal F}\right)$.
 From \eqref{fol_path_eqn} and \eqref{fol_jk_eqn}  it follows that $\left\|\overline{a}-\overline{a}_f \right\|< \eps$. Hence $\overline{a} \in C^*_r\left(\overline M, \overline{\mathcal F}\right)$
\end{proof}

\begin{corollary}\label{fol_out_cor}
	If $\overline{C^*_r\left(M, \mathcal F \right)}$ is the disconnected inverse limit of the sequence 
	$\mathfrak{S}_{C^*_r\left(M, \mathcal F \right)}\in \mathfrak{FinAlg}$  given by \eqref{fol_alg_seq_eqn}, with respect to the geometric representation $\pi_{\mathrm{geom}}$ (cf. \eqref{fol_equ_rep_eqn}) then following condition holds
	$$
	C^*_r\left(\overline{M}, \overline{\mathcal F} \right)\subset\overline{C^*_r\left(M, \mathcal F \right)}.
	$$
\end{corollary}
\begin{theorem}\label{fol_inf_thm}
	Let us consider a sequence  
	
	\begin{equation*}
	\begin{split}
	\mathfrak{S}_{C^*_r\left(M, \mathcal F \right)}=\\= \left\{C^*_r\left(M, \mathcal F \right) = C^*_r\left(M_0, \mathcal F_0 \right)\to  C^*_r\left(M_1, \mathcal F_1 \right)\to \dots \to  C^*_r\left(M_n, \mathcal F_n \right)\to \dots
	\right\}\in \\ \in \mathfrak{FilAlg}\end{split}
	\end{equation*}
	given by 	\eqref{fol_alg_seq_eqn}. Let $\widetilde{   \pi}:\varprojlim  \downarrow\mathfrak{S}_M  \to M$ be a topological covering associated with the topological inverse limit (cf. \ref{top_inf_constr}) of the sequence \begin{equation*}
	\mathfrak{S}_M = \left\{
	M = M_0 \leftarrow M_1 \leftarrow \dots  \leftarrow M_1 \leftarrow \dots\right\}
	\end{equation*}
	given by \eqref{fol_man_seq_eqn}. Let $\left(\varprojlim  \downarrow\mathfrak{S}_M, \widetilde{\mathcal F} \right)$ be the $\widetilde{   \pi}$-lift of the foliation $\left(M, \mathcal F \right)$.
	If $\pi_{\mathrm{geom}}$ is a geometric representation
	\begin{equation}
\pi_{\mathrm{geom}}:\widehat{C^*_r\left(M, \mathcal F \right)}=\varinjlim C^*_r\left(M_n, \mathcal F_n \right)\to B\left(\overline{\H}\right)
	\end{equation}
	given by \eqref{fol_equ_rep_eqn} then following conditions hold:
\begin{enumerate}
	\item [(i)] The representation $\pi_{\mathrm{geom}}$ is good,
	\item[(ii)] There are  isomorphisms:
	\begin{equation*}
	\begin{split}
\varprojlim_{\pi_{\mathrm{geom}}}
\downarrow 	\mathfrak{S}_{C^*_r\left(M, \mathcal F \right)} \cong C^*_r\left(\varprojlim  \downarrow\mathfrak{S}_M, \widetilde{\mathcal F} \right),\\
G\left(\varprojlim_{\pi_{\mathrm{geom}}}
\downarrow 	\mathfrak{S}_{C^*_r\left(M, \mathcal F \right)}~|~ C^*_r\left(M, \mathcal F \right)\right) \cong G\left(\varprojlim  \downarrow\mathfrak{S}_M~|~ M\right).
\end{split}
	\end{equation*}

\end{enumerate}
\end{theorem}
\begin{proof}
	From the Corollaries \ref{fol_in_cor} and \ref{fol_out_cor} it turns out that $\overline{C^*_r\left(M, \mathcal F \right)} = C^*_r\left(\overline{M}, \overline{\mathcal F} \right)$. The group $G\left(\varprojlim  \downarrow\mathfrak{S}_M~|~ M \right)$ is the maximal subgroup of $\widehat{G}$ maximal among subgroups $G \subset \widehat{G}$ such that
	$$
	G\left( \varprojlim  \downarrow\mathfrak{S}_M\right)  = \varprojlim  \downarrow\mathfrak{S}_M.
	$$
	
	From \eqref{fol_sqcup} it follows that 
	
	\begin{equation}\label{fol_man_sum}
\overline{M} = \bigsqcup_{g \in J} g \left( \varprojlim  \downarrow\mathfrak{S}_M\right) 
	\end{equation}
	where $J \subset \widehat{G}$ is a set of representatives of $\widehat{G}/G\left(\varprojlim  \downarrow\mathfrak{S}_M~|~ M \right)$.
	From \eqref{fol_man_sum} it turns out that the algebraic direct sum of irreducible algebras
	$$
	\bigoplus_{g \in J} g C^*_r\left(\varprojlim  \downarrow\mathfrak{S}_M, \widetilde{\mathcal F}\right) 
	$$
	is dense in $C^*_r\left(\overline{M}, \overline{\mathcal F} \right)$ and $C^*_r\left(\varprojlim  \downarrow\mathfrak{S}_M, \widetilde{\mathcal F}\right) \subset  C^*_r\left(\overline{M}, \overline{\mathcal F} \right)$ is a maximal irreducible subalgebra. 
	\newline
	(i) We need check conditions (a)-(c) of the Definition \ref{good_seq_defn}. Clearly the map
	$$
	\varinjlim C^*_r\left(  M_n, {\mathcal F}_n\right) \hookto M\left( C^*_r\left(\varprojlim  \downarrow\mathfrak{S}_M, \widetilde{\mathcal F}\right) \right) 
	$$ 
	is injective and $\bigoplus_{g \in J} g C^*_r\left(\varprojlim  \downarrow\mathfrak{S}_M, \widetilde{\mathcal F}\right)$ is dense in $C^*_r\left(\overline{M}, \overline{\mathcal F} \right)$, i.e. conditions (a), (b) of the Definition \ref{good_seq_defn} hold. For any $n \in N$ the homomorphism $G\left(\varprojlim  \downarrow\mathfrak{S}_M~|~ M\right) \to G\left( M_n~|~M\right)$ is surjective it follows that $$G\left(\varprojlim_{\pi_{\mathrm{geom}}}
	\downarrow 	\mathfrak{S}_{C^*_r\left(M, \mathcal F \right)}~|~ C^*_r\left(M, \mathcal F \right)\right)\to G\left(  C^*_r\left(  M_n, {\mathcal F}_n\right)~|~C^*_r\left(M, \mathcal F \right)\right) $$ is surjective, i.e. the condition (c) of the Definition  \ref{good_seq_defn} holds.
	\newline 
	(ii) Follows from the proof of (i).
\end{proof}

\subsection{Alternative equivariant representation}
\paragraph*{}
Let $\{p_k \in \mathbb{N}\}_{k \in \mathbb{N}}$ be an infinite sequence of natural numbers such that $p_k > 1$ for any $k$, and let $m_j = \Pi_{k = 1}^{j} p_k$. Let us consider a foliation $\left(\T^2, \mathcal{F}_\th \right)$ given by the Example \ref{fol_tor_exm}. There is a sequence of finite-fold topological coverings 
$$
\mathfrak{S}_{\T^2} = \left\{ \T^2 \xleftarrow{\left( \times p_1, \times p_1\right)} \T^2_1 \xleftarrow{\left( \times p_2, \times p_2\right)}\dots  \xleftarrow{\left( \times p_n, \times p_n\right)} \T^2_n \xleftarrow{\left( \times p_{n+1}, \times p_{n+1}\right)}\dots\right\}
$$
where $\T^2_n \cong \T^2$ for any $n \in \N$ (cf. Example \ref{fol_tor_cov_exm}). Denote by $$\pi_n= \left( \times p_n, \times p_n\right)\circ \dots \circ \left( \times p_1, \times p_1\right): \T^2_n \to  \T^2.$$
From the Theorem \ref{fol_fin_thm} it follows that there is 	an (algebraical)  finite covering sequence 
\begin{equation}\label{fol_seq_eqn}
\begin{split}
\mathfrak{S}_{C^*_r\left(\T^2,~ \mathcal{F}_\th \right)} =\left\{C^*_r\left(\T^2, \mathcal{F}_\th \right) \xrightarrow{} C^*_r\left(\T^2_1, \mathcal{F}_{\th/m_1^2} \right) \xrightarrow{} ... \xrightarrow{} C^*_r\left(\T^2_n, \mathcal{F}_{\th/m_n^2} \right)  \xrightarrow{} ...\right\}
\end{split}
\end{equation}
where $\left(\T^2_n, \mathcal{F}_{\th/m_n^2} \right)$ is the $\pi_n$-lift of $\left(\T^2, \mathcal{F}_\th \right)$. The topological inverse limit of $\downarrow\mathfrak{S}_{\T^2} $ is $\R^2$, i.e. $\varprojlim \downarrow\mathfrak{S}_{\T^2} = \R^2$. 
 From the natural infinite covering $\pi:\R^2 \to \T^2$ it follows that there is an induced by $\pi$ covering $\left(\R^2,~ \widetilde{\mathcal F} \right)$  of $\left(\T^2, \mathcal{F}_\th \right)$. The foliation $\left(\R^2,~ \widetilde{\mathcal F} \right)$ is simple and given by the bundle $p : \R^2\to \R^1$, so from  \eqref{fol_stab_eqn} it follows that
\begin{equation}\label{fol_r_eqn}
 C^*_r\left(\R^2,~ \widetilde{\mathcal F} \right) \approx C_0\left(
\R \right) \otimes \mathcal K.
\end{equation}
A following equation
\begin{equation}\label{fol_z2_eqn}
 G\left( \varprojlim \downarrow\mathfrak{S}_{\T^2} = \R^2~|~\T^2\right) = \Z^2
\end{equation}
is well known.
\paragraph{Geometric representation}
If $\pi_{\mathrm{geom}}: \varinjlim C^*_r\left(\T^2_n, \mathcal{F}_{\th/m_n^2} \right) \to B\left(\overline{ \H}\right)$ is a geometric representation given by \eqref{fol_equ_rep_eqn} then from the Theorem \ref{fol_inf_thm} it follows that $\pi_{\mathrm{geom}}$ is good. Moreover from \eqref{fol_r_eqn} and \eqref{fol_z2_eqn} it follows that
\begin{equation}\label{fol_tor_geom_eqn}
\begin{split}
\varprojlim_{\pi_{\mathrm{geom}}} \mathfrak{S}_{C^*_r\left(\T^2,~ \mathcal{F}_\th \right)}\cong C_0\left(
\R \right) \otimes \mathcal K, \\
G\left( \varprojlim_{\pi_{\mathrm{geom}}} \mathfrak{S}_{\left(\T^2,~ \mathcal{F}_\th \right)}~|~ C^*_r\left(\T^2,~ \mathcal{F}_\th \right) \right) \cong G\left(\R^2~|~\T^2 \right) \cong \Z^2.
\end{split}
\end{equation} 
\paragraph{Alternative representation}

A following (algebraical)  finite covering sequence 
\begin{equation}\label{torus_seq}
\mathfrak{S}_{C\left(\T^2_{\theta}\right)}  =\left\{C\left(\T^2_{\theta}\right)  \xrightarrow{} C\left(\T^2_{\theta/m_1^2}\right)  \xrightarrow{} ... \xrightarrow{} C\left(\T^2_{\theta/m^2_n}\right)  \xrightarrow{} ...\right\}
\end{equation}
 and an equivariant representation 
$$
\widehat{\pi}^\oplus: \varinjlim C\left(\T^2_{\theta/m^2_n}\right)  \to B\left(\H \right) 
$$
are described in \cite{ivankov:qnc}, and these objects  satisfy to the following theorem.
\begin{thm}\label{nt_inf_cov_thm} Following conditions hold:
	\begin{enumerate}
		\item[(i)] The representation $\widehat{\pi}^\oplus$ is good,
		\item[(ii)] 
		\begin{equation*}	\begin{split}
		\varprojlim_{\widehat{\pi}^\oplus} \downarrow \mathfrak{S}_{C\left(\T^2_{\theta}\right)}  = C_0\left(\R^{2N}_\th\right); \\
		G\left(\varprojlim_{\widehat{\pi}^\oplus} \downarrow \mathfrak{S}_{C\left(\T^2_{\theta}\right)} ~|~ C\left(\mathbb{T}^{2N}_\th \right)\right)  = \Z^{2N}.
		\end{split}
		\end{equation*}
%		\item[(iii)] The triple $\left(C\left(\mathbb{T}^{2N}_\th \right), C_0\left(\R^{2N}_\th\right), \Z^{2N} \right)$ is an  infinite noncommutative covering of $\mathfrak{S}_{C\left(\T^2_{\theta}\right)} $ with respect to $\widehat{\pi}^\oplus$.
	\end{enumerate}

\end{thm}
 The explanation of the $C^*$-algebra $C_0\left(\R^2_\th \right)$ is given in \cite{ivankov:qnc}. In \cite{ivankov:qnc,varilly_bondia:phobos} it is shown that $C_0\left(\R^{2N}_\th \right)$ is the $C^*$-norm completion of the Schwartz space $\SS\left(\R^{2N} \right)$ with the twisted Moyal product $\star_\th$. In \cite{varilly_bondia:phobos} it is proven that there are elements $\left\{f_{nm}\in\SS\left(\R^{2N} \right)\right\}_{m,n \in \N^0}$ which satisfy to following theorems.
%Theorem 6: 
\begin{thm}\label{moy_m_1}\cite{varilly_bondia:phobos}
	Let $\mathbf{s}$ be the Fr\'echet space of rapidly decreasing
	double sequences $c=c_{nm}$ such that
	$$
	r_k\left(c \right)  \stackrel{\mathrm{def}}{=} \sqrt{\sum_{m, n =0}^{\infty}\left|c_{mn}\right|^2 \left(  \left( 2m + 1\right)^{2k} \left( 2n + 1\right)^{2k}\right)} 
	$$
	is finite for all $k \in \N$, topologized by the seminorms $\left\{r_k\right\}_{k \in \N}$.
	For $f \in \SS\left(\R^{2N} \right)$  let $c$ be the sequence of coefficients in the expansion
	$$
	f = \sum_{m, n= 0}^{\infty}c_{mn}f_{mn}
	$$
	Then $f \mapsto c$ an isomorphism of Fr\'echet space spaces from $\SS\left(\R^{2N}\right) $
	onto $\mathbf{s}$.
\end{thm}
%Theorem 7: 
\begin{thm}\label{moy_m_2}\cite{varilly_bondia:phobos}
	If $a, b \in  \mathbf{s}$ correspond respectively to $f, g \in \SS\left(\R^{2N} \right)$
	as coefficient sequences in the twisted Hermite basis, then
	the sequence corresponding to the twisted product $f \star_\th g$ is
	the matrix product $ab$, where
	\begin{equation}\label{mp_mult_eqn}
	\left( ab\right)_{mn} \stackrel{\mathrm{def}}{=} \sum_{k= 0}^{\infty} a_{mk}b_{kn}.
	\end{equation}
	
\end{thm}
From the Theorems \ref{moy_m_1} and \ref{moy_m_2} it follows that  there is a *-isomorphism of $C^*$-algebras $C_0\left(\R^{2N}_\th \right)\cong \mathcal K$.
From the Theorem \ref{stable_fin_cov_thm} it follows that 
\begin{equation}\label{tor_seq_eqn}
\mathfrak{S}_{C\left( \T^2_{\theta}\right) \otimes \mathcal K} =\left\{C\left(\T^2_{\theta}\right) \otimes \mathcal K  \xrightarrow{} C\left(\T^2_{\theta/m_1^2}\right) \otimes \mathcal K \xrightarrow{} ... \xrightarrow{} C\left(\T^2_{\theta/m^2_n}\right) \otimes \mathcal K \xrightarrow{} ...\right\}.
\end{equation}
is an algebraical  finite covering sequence. There is a representation 
$$
\widehat{\pi}^\oplus \otimes \Id_{ \mathcal K}:\varinjlim C\left(\T^2_{\theta/m^2_n}\right) \otimes \mathcal K \to B\left(\H \otimes \H \right).  
$$
From the Theorems \ref{stab_inf_thm} and \ref{nt_inf_cov_thm} it follows that
\begin{itemize}
	\item The representation $\widehat{\pi}^\oplus \otimes \Id_{ \mathcal K}$ is good,
	\item 
	\begin{equation*}
\begin{split}
\varprojlim_{\widehat{\pi}^\oplus \otimes \Id_{ \mathcal K}} \downarrow \mathfrak{S}_{C\left( \T^2_{\theta}\right) \otimes \mathcal K}\cong\mathcal K \otimes \mathcal K\cong \mathcal K,\\
G\left(\varprojlim_{\widehat{\pi}^\oplus \otimes \Id_{ \mathcal K}} \downarrow \mathfrak{S}_{C\left( \T^2_{\theta}\right) \otimes \mathcal K}~|~ C\left( \T^2_{\theta}\right) \otimes \mathcal K\right)=\Z^2.
\end{split}
\end{equation*}		
\end{itemize}

From \eqref{fol_th_eqn} it follows that sequence \eqref{fol_seq_eqn} is isomorphic to the sequence \eqref{tor_seq_eqn}. From this fact it follows that there is a good equivariant representation
$$
{\pi}_{\mathrm{alg}} : \varinjlim C^*_r\left(\T^2_n, \mathcal{F}_{\th/m_n^2} \right) \to B\left(\H\right)
$$
such that
\begin{equation}\label{fol_comp_eqn}
\begin{split}
\varprojlim_{{\pi}_{\mathrm{alg}} }\downarrow \mathfrak{S}_{\left(\T^2,~ \mathcal{F}_\th \right)}\cong \varprojlim_{\widehat{\pi}^\oplus \otimes \Id_{ \mathcal K}} \downarrow \mathfrak{S}_{C\left( \T^2_{\theta}\right) \otimes \mathcal K}\cong \mathcal K,\\
G\left(\varprojlim_{{\pi}_{\mathrm{alg}} } \mathfrak{S}_{\left(\T^2,~ \mathcal{F}_\th \right)}~|~ C^*_r\left(\T^2,~ \mathcal{F}_\th \right)\right)\cong\\\cong G\left(\varprojlim_{\widehat{\pi}^\oplus \otimes \Id_{ \mathcal K}} \downarrow \mathfrak{S}_{C\left( \T^2_{\theta}\right) \otimes \mathcal K}~|~ C\left( \T^2_{\theta}\right) \otimes \mathcal K\right)\cong \Z^2.
\end{split}
\end{equation}
		
A comparison of \eqref{fol_tor_geom_eqn} and \eqref{fol_comp_eqn} gives a following result
$$
\varprojlim_{\pi_{\mathrm{alg}}}\downarrow \mathfrak{S}_{\left(\T^2,~ \mathcal{F}_\th \right)}\not\approx\varprojlim_{\pi_{\mathrm{geom}}}\downarrow \mathfrak{S}_{\left(\T^2,~ \mathcal{F}_\th \right)}.
$$
 From the above equation it follows that the  construction given by the Theorem \ref{fol_inf_thm} does not yield a one to one correspondence between geometric and algebraic infinite coverings of foliations.


\begin{thebibliography}{10}
	

%\bibitem{ant_azz_scan:flat_k}Paolo Antonini, Sara Azzali, Georges Skandalis {\it Flat bundles, von Neumann algebras and $K$-theory with $\mathbb{R}/\mathbb{Z}$-coefficients}, arXiv:1308.0218, 2013.

\bibitem{arveson:c_alg_invt} W. Arveson. {\it An Invitation to $C^*$-Algebras}, Springer-Verlag. ISBN 0-387-90176-0, 1981.

%\bibitem{bezandry_diagana:bound_unbound}Paul H. Bezandry, Toka Diagana {\it Bounded and Unbounded Linear Operators}, in {\it Almost Periodic Stochastic Processes}, Springer, 2011.

%\bibitem{ballentine:qm} Leslie E Ballentine. {\it Quantum Mechanics: A Modern Development.} World Scientific Publishing Co. Pte. Ltd. 2000.

\bibitem{blackadar:ko} B. Blackadar. {\it K-theory for Operator Algebras}, Second edition. Cambridge University Press, 1998.

\bibitem{BroGreRie}
L.~G. Brown, P.~Green, and M.~A. Rieffel.
\newblock \textit{Stable isomorphism and strong {Morita} equivalence of
{$C^*$}-algebras}.
\newblock { Pacific J. Math.} {\bf 71} (1977), 349--363. 1977.

%\bibitem{blackadar:kocalg_neumann} B. Blackadar {\it Operator Algebras Theory of C* - Algebras and von Neumann Algebras}. Springer-Verlag Berlin Heidelberg 2006.

%\bibitem{blackadar:shape_theory} B. Blackadar, {\it Shape theory for $C^*$-algebras}, Math. Scand. 56 , 249-275, 1985.

%\bibitem{blecher:hilb_gen} D.P. Blecher. {\it A generalization of Hilbert modules}, J.Funct. An. 136, 365-421 1996.

%\bibitem{bost:oka}J.-B. Bost, {\it Principe d’Oka, $K$-th\'eorie et syst$\grave{e}$mes dynamiques non commutatifs}. Invent. Math. 101 (1990), 261–333. 1990.

%\bibitem{bogachev_measure_v1}V. I. Bogachev. {\it Measure Theory} (volume 1). Springer-Verlag, Berlin, 2007.
%\bibitem{bogachev_measure_v2}V. I. Bogachev. {\it Measure Theory}. (volume 2). Springer-Verlag, Berlin, 2007.



%\bibitem{bourbaki_sp:gt} N. Bourbaki, {\it General Topology}, Chapters 1-4, Springer, Sep 18, 1998.

%\bibitem{bratteli_robinson:oa_qsm}O. Bratteli and D. W. Robinson, {\it Operator Algebras and Quantum Statistical Mechanics 1}. Springer, New York. 1987.




%\bibitem{bruckler:tensor} Franka Miriam Br\"uckler. {\it Tensor products of $C^*$-algebras, operator spaces and Hilbert $C^*$-modules}. Mathematical Communications 4(1999), 1999.
%\bibitem{chakraborty_pal:quantum_su_2} Partha Sarathi Chakraborty,  Arupkumar Pal. \textit{Equivariant spectral triples on the quantum $SU(2)$ group}. arXiv:math/0201004v3, 2002.


%\bibitem{chang:fermionic} Ee Chang-Young, Hiroaki Nakajima, Hyeonjoon Shin. {\it Fermionic $T$-duality and Morita Equivalence}, arXiv:1101.0473, 2011.

%\bibitem{morita_hopf_galois}S. Caenepeel, S. Crivei, A. Marcus, M. Takeuchi. {\it Morita equivalences induced by bimodules over Hopf-Galois extensions.} arXiv:math/0608572, 2007.

%\bibitem{cheng_li:gauge} Cheng, T.-P.; Li, L.-F. {\it Gauge Theory of Elementary Particle Physics}. Oxford University Press. ISBN 0-19-851961-3. 1983.

%\bibitem{chavel:riemann} Isaac Chavel. {\it Riemannian Geometry: A Modern Introduction} (Cambridge Studies in Advanced Mathematics) Paperback – July 6, 2006.

%\bibitem{chun-yen:separability} Chun-Yen Chou. {\it Notes on the Separability of $C^*$-Algebras.} TAIWANESE JOURNAL OF MATHEMATICS Vol. 16, No. 2, pp. 555-559, April 2012 This paper is available online at http://journal.taiwanmathsoc.org.tw, 2012.

%\bibitem{cohn_measure} Donald L. Cohn \textit{Measure Theory}. Basel: Birkh\"auser,  — 457 p. Second edition. ISBN 978-1-4614-6955-1. 2013.


%\bibitem{connes:c_alg_dg} Alain Connes. {\it $C^*$-algebras and differential geometry}. arXiv:hep-th/0101093, 2001.

%\bibitem{clare_crisp_higson:adj_hilb} Pierre Clare, Tyrone Crisp, Nigel Higson {\it Adjoint functors between categories of Hilbert modules}.  arXiv:1409.8656, 2014.

%\bibitem{connes:grav}A. Connes. {\it Gravity coupled with matter and foundation of noncommutative geometry}. Commun. Math. Phys. 182 (1996), 155–176. 1996.

%\bibitem{connes:ncg94} Alain Connes. {\it Noncommutative Geometry}, Academic Press, San Diego, CA,  661 p., ISBN 0-12-185860-X, 1994.

%\bibitem{connes_landi:isospectral} Alain Connes, Giovanni Landi. {\it Noncommutative Manifolds the Instanton Algebra and Isospectral Deformations}, arXiv:math/0011194, 2001.

%\bibitem{connes_marcolli:motives}
%Alain Connes, Matilde Marcolli. {\it Noncommutative Geometry, Quantum Fields and Motives},  American Mathematical Society, Colloquium Publications, 2008.

% \bibitem{connes_moscovici:local_index} A. Connes and H. Moscovici, {\it The local index theorem in noncommutative geometry"}. Geom. and Funct. Anal., 1996.


%\bibitem{cuntz_quillen:alg_ext} Joachim Cuntz, Daniel Quillen.  {\it Algebra extensions and nonsingularity}, J. Amer. Math. Soc. 8 251-289, 1995

%\bibitem{davis_kirk_at}James F. Davis. Paul Kirk. {\it Lecture Notes in Algebraic Topology}. Department of Mathematics, Indiana University, Blooming- ton, IN 47405, 2001.

\bibitem{connes:foli_survey} A. Connes. \textit{A survey of foliations and operator algebras}. Operator algebras and applications, Part
1, pp. 521-628, Proc. Sympos. Pure Math., 38, Amer. Math. Soc, Providence, R.I., 1982; MR
84m:58140. 1982.

%\bibitem{dixmier_tr}J. Dixmier. {\it Traces sur les $C^*$-algebras}. Ann. Inst. Fourier, 13, 1(1963), 219-262, 1963.


%\bibitem{engelking:general_topology} Ryszard Engelking. \textit{General topology}, PWN, Warsaw. 1977.


%\bibitem{spectral_action_nt}Driss Essouabri, Bruno Iochum, Cyril Levy, and Andrzej Sitarz. \textit{Spectral action on noncommutative torus.} Journal of Noncommutative Geometry, 2 (2008), 53–123, 2008. 


%\bibitem{frank:frames} Michael Frank, David R. Larson, {\it Frames in Hilbert $C^*$-modules and $C^*$-algebras}, arXiv:math/0010189, 2000.


%\bibitem{moyal_spectral} V. Gayral, J. M. Gracia-Bond\'{i}a, B. Iochum, T. Sch\"{u}cker, J. C. Varilly. {\it Moyal Planes are Spectral Triples}. arXiv:hep-th/0307241, 2003.

%\bibitem{gilkey:odd_space}P.B. Gilkey. {\it The eta invariant and the $K$-theory of odd dimensional spherical space forms}.Inventiones mathematicae, Springer-Verlag, 1984.

\bibitem{varilly_bondia:phobos}Jos\'e M. Gracia-Bondia, Joseph C. Varilly.  \textit{Algebras of Distributions suitable for phase-space quantum mechanics. I}. Escuela de Matem\'{a}tica, Universidad de Costa Rica, San Jos\'e, Costa Rica J. Math. Phys 29 (1988), 869-879, 1988.


%\bibitem{elliot:an} Elliott H. Lieb, Michael Loss. \textit{Analysis}, American Mathematical Soc., 2001.



%\bibitem{nicolas_ginoux:dirac_spectrum}Nicolas Ginoux. {\it The Dirac Spectrum.} Springer, Jun 11, 2009.

%\bibitem{varilly_bondia} Jos\'e M. Gracia-Bondia, Joseph C. Varilly, Hector Figueroa, {\it Elements of Noncommutative Geometry}, Springer, 2001.   92  96    142

%\bibitem{green_schwarz_witten:superstring} {\it Superstring Theory: Volume 2, Loop Amplitudes, Anomalies and Phenomenology}. (Cambridge Monographs on Mathematical Physics)  by Michael B. Green, John H. Schwarz, Edward Witten. 1988.


%\bibitem{gross_gauge}David J. Gross. {\it Gauge Theory-Past, Present, and Future?} Joseph Henry Luborutoties, Ainceton University, Princeton, NJ 08544, USA. (Received November 3,1992).

%\bibitem{ful:gr_repr} Fulton William, Harris Joe. {\it Representation theory. A first course} Graduate Texts in Mathematics, Readings in Mathematics 129, New York: Springer-Verlag. 1991.

%\bibitem{halmos:set} Paul R.  Halmos {\it Naive Set Theory.} D. Van Nostrand Company, Inc., Prineston, N.J., 1960.

\bibitem{Hil-Skan83}
M.~Hilsum and G.~Skandalis.
\newblock \textit{Stabilit\'e des {$C\sp{\ast} $}-alg\`ebres de feuilletages}.
\newblock { Ann. Inst. Fourier (Grenoble)} {\bf 33} (1983), 201--208. 1983.

\bibitem{ivankov:qnc} Petr Ivankov. \textit{Quantization of noncompact coverings}, arXiv:1702.07918, 2017.

\bibitem{kord:foli} Yuri A. Kordyukov. \textit{Noncommutative geometry of foliations}. Journal of K-Theory, Cited by 2 Get access Volume 2, Issue 2 (In Memory of Yurii Petrovich Solovyev October 8, 1944 – September 11, 2003), October 2008, pp. 219-327, 2008.

%\bibitem{helemsky:qfa} A. Ya. Helemsky. {\it Quantum Functional Analysis. Non-Coordinate Approach.} Providence, R.I. : American Mathematical Society, 2010.

%\bibitem{ivankov:infinite_cov_pr} Petr Ivankov.  {\it Infinite Noncommutative Covering Projections}.  arXiv:1405.1859, 2014.
 
% \bibitem{ivankov:inv_lim}  Petr R. Ivankov. {\it Inverse Limits of Noncommutative Covering Projections},  arXiv:1412.3431, 2014.
 

% \bibitem{ivankov:nc_cov_k_hom}Petr Ivankov. {\it Noncommutative covering projections and $K$-homology}, 	arXiv:1402.0775, 2014.
 

 %\bibitem{ivankov:nc_wilson_lines} Petr Ivankov.  {\it Noncommutative Generalization of Wilson Lines}. arXiv:1408.4101, 2014.

%\bibitem{hajac:toknotes}{\it Lecture notes on noncommutative geometry and quantum groups}, Edited by Piotr M. Hajac.


%\bibitem{ivankov:nc_cov_k_hom}Petr Ivankov. {\it Noncommutative covering projections and $K$-homology}, 	arXiv:1402.0775, 2014.

%\bibitem{ivankov:uni_nc_cov}
%Petr R. Ivankov. {\it Universal covering space of the noncommutative torus},
%arXiv:1401.6748, 2014.

%\bibitem{kakariadis:corr}Evgenios T.A. Kakariadis, Elias G. Katsoulis, {\it Operator algebras and $C^*$-correspondences: A survey.} 	arXiv:1210.6067, 2012.

%\bibitem{kaku:loc}Kaku, M. {\it Locality in the gauge-covariant field theory of strings}. Phys. Lett. 162B, 97. Kaku, M. 1986.

%\bibitem{karaali:ha} Gizem Karaali {\it On Hopf Algebras and Their Generalizations}, arXiv:math/0703441, 2007.

%\bibitem{karoubi:k} M. Karoubi. {\it K-theory, An Introduction.} Springer-Verlag 1978.

%\bibitem{kastler:connes_lott} Daniel Kastler, Thomas Schucker, {\it The Standard Model a la Connes-Lott}, arXiv:hep-th/9412185, 1994.


%\bibitem{koba_nomi:fgd} S. Kobayashi, K. Nomizu. {\it Foundations of Dif and only iferential Geometry}. Volume 1. Interscience publishers a division of John Willey \& Sons, New York - London. 1963.

%\bibitem{lee:smooth} John M. Lee. {\it Introduction to Smooth Manifolds}. University of Washington. Department of Mathematics. Version 3.0, December 31, 2000.

%\bibitem{demeyer:genreal_galois}DeMeyer, F. R. \textit{Some notes on the general Galois theory of rings}. Osaka J. Math.2 (1965), 117-127, 1965.

%\bibitem{mitchener:c_cat}Paul D. Mitchener. \textit{$C^*$-categories}. Odense University,  September 13, 2001.



%\bibitem{milne:etale}J.S. Milne. {\it \'Etale cohomology.} Princeton Univ. Press  1980.

%\bibitem{miyashita_fin_outer_gal} Y\^oichi Miyashita, {\it Finite outer Galois theory of noncommutative rings}. Department of Mathematics, Hokkaido, University, 1966.

%\bibitem{bram:atricle} Bram Mesland. {\it Unbounded bivariant $K$-theory and correspondences in noncommutative geometry} arXiv:0904.4383, 2009.


%\bibitem{miyashita_fin_outer_gal} Y\^oichi Miyashita, {\it Finite outer Galois theory of noncommutative rings}. Department of Mathematics, Hokkaido, University, 1966.

%\bibitem{miyashita_infin_outer_gal} Y\^oichi Miyashita, {\it Locally finite outer Galois theory}. Department of Mathematics, Hokkaido, University, 1967.

%\bibitem{muhly_williams:groupoid_ctr}  Paul S. Muhly and Dana P. Williams. {\it Continuous trace groupoid $C^*$-algebras.}, Math. Scand. 1990

%\bibitem{munkres:topology} James R. Munkres. {\it Topology.} Prentice Hall, Incorporated, 2000.


\bibitem{murphy}G.J. Murphy. {\it $C^*$-Algebras and Operator Theory.} Academic Press, 1990.

%\bibitem{Pa1} {W.~L.~Paschke}, Inner product modules over B*-algebras,  {\it Trans.~Amer.~Math.~Soc.} {\bf 182}(1973), 443-468.
  

 \bibitem{ouchi:cov_fol}Moto O'uchi \textit{Coverings of foliations and associated $C^*$-algebras}.  Mathematica Scandinavica Vol. 58 (1986), pp. 69-76. 1986 
\bibitem{pavlov_troisky:cov} Alexander Pavlov, Evgenij Troitsky. {\it Quantization of branched coverings.}   Russ. J. Math. Phys. (2011) 18: 338. doi:10.1134/S1061920811030071, 2011.
 
  

\bibitem{pedersen:ca_aut}Pedersen G.K. {\it $C^*$-algebras and their automorphism groups}. London ; New York : Academic Press, 1979.

%\bibitem{ros_scho:kt_uct} Jonathan Rosenberg, Claude Schochet, {\it The K\"unneth theorem and the universal coefficient theorem for Kasparov's generalized K -functor}, Duke Math. J. Volume 55, Number 2 1987.

%\bibitem{phillips:inv_lim_app} N. Christopher Phillips {\it Inverse Limits of $C^*$ - algebras and Applications.} University of California at Los Angeles, Los Angeles, CA 90024, 1991

%\bibitem{reed_simon:mp_1}Michael Reed, Barry Simon. {\it Methods of modern mathematical physics 1: Functional Analysis}. Academic Press, 1972.

%\bibitem{switzer:at} Switzer R M, {\it Algebraic Topology - Homotopy and Homology},
%Springer. 2002



%\bibitem{adams:infinite_loop_spaces} J. F. Adams. {\it Infinite loop spaces}. Ann. of Math. Studies no. 90, Princeton Univ. Press, Princeton, N. J., 1978

%\bibitem{phillips:c_infty_loop} N. Christopher Philllips. {\it $C^{\infty}$ Loop Algebras and Noncommutative Bott Periodicity}. Transactions of the American Matematical Society, Volume 325, Number 2, June 1991

%\bibitem{sitarz:equiv} Andrzej Sitarz {\it Equivariant spectral triples}, Noncommutative Geometry and Quantum Groups (Piotr M. Hajac and Wieslaw Pusz, eds.), Banach Center Publ., vol 61, Polish Acad. Sci., pp. 231-268,  Warsaw 2003



%\bibitem{cuntz:o_n} J. Cuntz, {\it Simple $C^*$ - algebras generated by isometries}, Comm. Math. Phys. 57:2, 1977

%\bibitem{cuntz:k_o_n} J. Cuntz, {\it$K$ - theory of certain $C^*$ - algebras}, Ann. of Math. (2), 113:1 1981


%\bibitem{Cohn:68} Paul~Moritz Cohn. {\it {M}orita equivalence and duality}, Queen Mary College   Mathematics Notes, Dillon's Q.M.C.\ Bookshop, London, 1968.


%\bibitem{bourbaki_sp:gt} N. Bourbaki, {\it General Topology}. Chapters 1-4, Springer, Sep 18, 1998

%\bibitem{williams_sp:morita_cont_trace_alg} Iain Raeburn, Dana P. Williams. {\it Morita Equivalence and Continuous-Trace $C^*$-Algebras}. American Mathematical Soc., 1998

%\bibitem{dixmier_tr}J.Dixmier. {\it Traces sur les $C^*$-algebras}. Ann. Inst. Fourier, 13, 1(1963), 219-262, 1963

%\bibitem{baum_higson_schik:kh}Paul Baum, Nigel Higson, and Thomas Schick. {\it On the Equivalence of Geometric and Analytic $K$-Homology}. Pure and Applied Mathematics Quarterly Volume 3, Number 1 (Special Issue: In honor of Robert MacPherson, Part 3 of 3) 1-24, 2007

%\bibitem{meyer:morita} Ralf Meyer. {\it Morita Equivalence In Algebra And Geometry.} math.berkeley.edu/~alanw/277papers/meyer.tex, 1997



%\bibitem{rumynin_hopf_galois_ci} Dmitriy Rumynin  {\it Hopf-Galois extensions with central invariants.}  arXiv:q-alg/9707021 1997

%\bibitem{dixmier_a_r} Jacques Dixmier {\it Les C*-alg\`{e}bres et leurs repr\'esentations} 2e \'ed. Gauthier-Villars in Paris 1969



%\bibitem{rieffel_finite_g} Marc A. Reif and only ifel, {\it Actions of Finite Groups on $C^*$ - Algebras}. 	Department of Mathematics University of California Berkeley. Cal. 94720 U.S.A. 1980.

%\bibitem{rieffel_morita}Marc A. Reif and only ifel, {\it Morita equivalence for $C^*$-algebras and $W^*$-algebras }, Journal of Pure and Applied Algebra 5 (1974), 51-96. 1974.

%\bibitem{dixmier_douady_d} Claude Schochet, {\it Dixmier-Douady for Dummies}. 	arXiv:0902.2025 2009.


%\bibitem{rennie:smooth_nonunital} A. Rennie, {\it Smoothness and locality for nonunital spectral triples}, $K$-Theory 28 (2003), 127–165. 2003.

%\bibitem{schweitzer:m_local_frechet} L. B. Schweitzer. {\it  A short proof that $M_n(A)$ is local if $A$ is local and Fr\'echet}, Internat. J. Math. 3 (1992), 581-589. MR 93i:46082. 1992.


\bibitem{spanier:at}E.H. Spanier. {\it Algebraic Topology.} McGraw-Hill. New York, 1966.

%\bibitem{woronowicz:su2} S.L. Woronowicz.\textit{Twisted SU(2) Group. An Example of a Non-Commutative Dif and only iferential Calculus}	PubL RIMS, Kyoto Univ. 23 (1987), 117-181, 1987.


%	\bibitem{takeda:inductive} Zir\^{o} Takeda. \textit{Inductive limit and infinite direct product of operator algebras.} Tohoku Math. J. (2) 	Volume 7, Number 1-2 (1955), 67-86. 1955.

%\bibitem{takesaki:oa_ii} Takesaki, Masamichi. {\it Theory of Operator Algebras II } Encyclopaedia of Mathematical Sciences, 2003. 

%\bibitem{thomsen:ho_type_uhf} Klaus Thomsen. {\it The homotopy type of the group of automorphisms of a $UHF$-algebra}. Journal of Functional Analysis. Volume 72, Issue 1, May 1987.


%\bibitem{takeuchi:inf_out_cov}Takeuchi, Yasuji {\it Infinite outer Galois theory of non commutative rings} Osaka J. Math. Volume 3, Number 2, 1966.

%\bibitem{inikolaev:c_bundles} Igor Nikolaev, {\it Topology of the $C^*$ algebra bundles}. Centre interuniversitaire de recherche en g\'eom\'etrie diff\'erentielle et topologie UQAM Montr\'eal H3C 3P8 Canada 1999.




%\bibitem{bass} H. Bass. {\it Algebraic K-theory.} W.A. Benjamin, Inc. 1968.

%\bibitem{thompsen:homtop}Klaus Thompsen. {\it Homotopy classes of * - homomorphisms between stable $C^*$ - algebras and their muliplier algebras.} Duke Matematical Journal (C) August 1990.




%\bibitem{blackadar:oa}
%B. Blackadar. {\it Operator Algebras Theory of $C^*$ Algebras and von Neumann Algebras}. Springer-Verlag Berlin Heidelberg 2006




%\bibitem{murre:fund}
%J.P. Murre. {\it Lectures on An Introduction to Grothendieck's  Theory of the Fundamental Group.} Notes by S. Anantharaman, Tata Institute of Fundamental Research, Bombay, 1967.


%\bibitem{connes_marcolli::motives} Alain Connes Matilde Marcolli. {\it Noncommutative Geometry, Quantum Fields and Motives.} Preliminatry version. www.alainconnes.org/docs/bookwebfinal.pdf

%\bibitem{mesland::unbounded_biviariant} Bram Mesland. {\it Unbounded biviariant $K$-theory and correspondences in noncommutative geometry}. arXiv:0904.4383. 2009.



%\bibitem{connes:ng} A. Connes. {\it Noncommutative Geometry.} Academic Press, London, 1994.

%\bibitem{brown_green_rieffel:morita_stable} Lawrence G. Brown, Philip Green, and Marc A. Rieffel.  {\it Stable isomorphism and strong Morita equivalence of $C^‚àó$  -algebras.} Source: Pacific J. Math. Volume 71, Number 2 , 349-363, 1977


%\bibitem{faith:I} C. Faith. Algebra: {\it Rings, Modules and Cathegories I}. Springer-Verlag 1973




%\bibitem{varilly:noncom} J.C. V\'arilly. {\it An Introduction to Noncommutative Geometry}. EMS 2006.

%\bibitem{Wegge-Olsen} {N.~E.~Wegge-Olsen}, {\it K-theory and C*-algebras -- a Friendly Approach}, Oxford University Press, Oxford, England,   1993.
   
%\bibitem{weil:basic_number_theory}Andre Weil {\it Basic Number Theory}. Springer 1995

%\bibitem{Partha_quantum_su} Partha Sarathi Chakraborty, Arupkumar Pal. {\it Equivariant spectral triples on the quantum $SU(2)$ group.} arXiv:math.KT/0201004, 2003.

%\bibitem{geom_anal_k_homology} Paul Baum, Nigel Higson, and Thomas Schick {\it On the Equivalence of Geometric and Analytic K-Homology} arXiv:math/0701484, 2009.






%\bibitem{connesdebois:3dsphere}
%A. Connes, M. Dubois-Violette. Moduli space and structure of
%noncommutative 3-spheres.  LPT-ORSAY 03-34 ; IHES/M/03/56.
%Lett.Math.Phys. 66 91-121. 2003.

%\bibitem{conneslandi:isospectal}
%A. Connes, G. Landi. Noncommutative Manifolds the Instanton Algebra
%and Isospectral Deformations, math.QA/0011194, 2000.

%\bibitem{suprsym:qt}
%J. Fr\"ohlich, O. Grandjean, A. Recknagel. Supersymmetric Quantum
%Theory and (Non-Commutative) Dif and only iferential Geometry, ETH-TH/96-45
%1996.

%\bibitem{ivankov:fund}
%P.R. Ivankov, N.P. Ivankov. The Noncommutative Geometry
%Generalization of Fundamental Group. arXiv:math.KT/0604508, 2006

%\bibitem{johnstone:topos}
%P.T. Johnstone. Topos Theory, L. M. S. Monographs no. 10, Academic
%Press 1977.

%\bibitem{lang}
%S. Lang. Algebra. Addison-Wesley Publishing Company, Reading, Mass
%1965.

%\bibitem{reconstr}
%A. Rennie, J.C. V\'arilly. Reconstruction of Manifolds in
%Noncommutative Geomery. \newline arXiv:math/0610418v3 [math.OA] 24
%Mar 2007.


%\bibitem{varilly:lecture}
%J.C. V\'arilly. Dirac operators and Spectral Geometry. Lecture notes
%by Pave{\l} Witkowsky from Warshaw Noncommutative Geometry, January
%2006.%http://ncg.mimuw.edu.pl/index.php?option=com_docman&task=doc_download&gid=10&Itemid=58

%\bibitem{wolf:const_curv} Wolf, J. {\it Spaces of constant curvature}. New York: McGraw-Hill, 1967.

%\bibitem{raimar_wulkenhaar:nc_spectral_triple}
%Raimar Wulkenhaar.{\it Non-compact spectral triples with finite volume}. 	arXiv:0907.1351, 2009.


\bibitem{Rieffel74}
M.~A. Rieffel.
\textit{ Morita equivalence for {$C\sp{\ast} $}-algebras and {$W\sp{\ast}
	$}-algebras}.
\newblock { J. Pure Appl. Algebra} {\bf 5} (1974), 51--96. 1974.

\end{thebibliography}
\end{document}